\documentclass[letterpaper]{pspum-l}

\usepackage{amssymb}

\usepackage[all]{xy}

\newtheorem{thm}{Theorem}

\newtheorem{theorem}{Theorem}[subsection]
\newtheorem{proposition}[theorem]{Proposition}
\newtheorem{lemma}[theorem]{Lemma}
\newtheorem{corollary}[theorem]{Corollary}

\theoremstyle{definition}
\newtheorem{definition}[theorem]{Definition}
\newtheorem{example}[theorem]{Example}
\newtheorem{examples}[theorem]{Examples}

\theoremstyle{remark}
\newtheorem{remark}[theorem]{Remark}
\newtheorem{remarks}[theorem]{Remarks}

\numberwithin{equation}{subsection}

\newcommand{\C}{{\rm C}}

\newcommand{\N}{{\rm N}}
\newcommand{\R}{{\rm R}}

\newcommand{\ad}{{\rm ad}}
\newcommand{\aff}{{\rm aff}}
\newcommand{\ant}{{\rm ant}}
\newcommand{\car}{{\rm char}}
\newcommand{\diag}{{\rm diag}}
\newcommand{\et}{{\rm \acute{e}t}}
\newcommand{\id}{{\rm id}}
\newcommand{\gp}{{\rm gp.}}
\newcommand{\gpsch}{{\rm gp.sch.}}
\newcommand{\ptsch}{{\rm pt.sch.}}
\newcommand{\red}{{\rm red}}
\newcommand{\reg}{{\rm reg}}
\newcommand{\sab}{{\rm sab}}

\newcommand{\Ad}{{\rm Ad}}
\newcommand{\Aff}{\operatorname{Aff}}
\newcommand{\Alb}{\operatorname{Alb}}
\newcommand{\Aut}{\operatorname{Aut}}
\newcommand{\Cl}{\operatorname{Cl}}
\newcommand{\Der}{\operatorname{Der}}
\newcommand{\End}{\operatorname{End}}
\newcommand{\Ext}{\operatorname{Ext}}
\newcommand{\Hilb}{\operatorname{Hilb}}
\newcommand{\Hom}{\operatorname{Hom}}
\newcommand{\Int}{\operatorname{Int}}
\newcommand{\Ker}{\operatorname{Ker}}
\newcommand{\GL}{\operatorname{GL}}
\newcommand{\Lie}{\operatorname{Lie}}
\newcommand{\NS}{\operatorname{NS}}
\newcommand{\Pic}{\operatorname{Pic}}
\newcommand{\Spec}{\operatorname{Spec}}
\newcommand{\Stab}{\operatorname{Stab}}
\newcommand{\Sym}{\operatorname{Sym}}

\newcommand{\bA}{\mathbb{A}}
\newcommand{\bG}{\mathbb{G}}
\newcommand{\bP}{\mathbb{P}}
\newcommand{\bR}{\mathbb{R}}
\newcommand{\bZ}{\mathbb{Z}}

\newcommand{\cD}{\mathcal{D}}

\newcommand{\cG}{\mathcal{G}}
\newcommand{\cI}{\mathcal{I}}
\newcommand{\cJ}{\mathcal{J}}
\newcommand{\cL}{\mathcal{L}}
\newcommand{\cM}{\mathcal{M}}
\newcommand{\cO}{\mathcal{O}}
\newcommand{\cR}{\mathcal{R}}

\newcommand{\fa}{\mathfrak{a}}
\newcommand{\fg}{\mathfrak{g}}
\newcommand{\fm}{\mathfrak{m}}
\newcommand{\fz}{\mathfrak{z}}

\begin{document}

\title[Structure of algebraic groups]
{Some structure theorems for algebraic groups}

\author{Michel Brion}

\address{Institut Fourier, CS 40700, 38058 Grenoble Cedex 9, France}

\email{Michel.Brion@univ-grenoble-alpes.fr}

\urladdr{https://www-fourier.ujf-grenoble.fr/~mbrion/}


\date{}



\begin{abstract}
These are extended notes of a course given at Tulane University 
for the 2015 Clifford Lectures. Their aim is to present structure
results for group schemes of finite type over a field, with 
applications to Picard varieties and automorphism groups. 
\end{abstract}

\maketitle

\setcounter{tocdepth}{2}

\tableofcontents

\section{Introduction}
\label{sec:i}

The algebraic groups of the title are the group schemes of finite
type over a field; they occur in many questions of algebraic 
geometry, number theory and representation theory. To analyze
their structure, one seeks to build them up from algebraic groups 
of a specific geometric nature, such as smooth, connected, affine,
proper... A first result in this direction asserts that
\textit{every algebraic group $G$ has a largest connected normal 
subgroup scheme $G^0$, the quotient $G/G^0$ is finite and 
\'etale, and the formation of $G^0$ commutes with field extensions}. 
The main goal of this expository text is to prove two more
advanced structure results:

\begin{thm}\label{thm:affi}
Every algebraic group $G$ over a field $k$ has a smallest 
normal subgroup scheme $H$ such that the quotient $G/H$ 
is affine. Moreover, $H$ is smooth, connected and contained 
in the center of $G^0$; in particular, $H$ is commutative. 
Also, $\cO(H) = k$ and $H$ is the largest subgroup scheme 
of $G$ satisfying this property. The formation of $H$ commutes 
with field extensions.
\end{thm}

\begin{thm}\label{thm:che}
Every algebraic group $G$ over $k$ has a smallest normal subgroup 
scheme $N$ such that $G/N$ is proper. Moreover, $N$ is affine 
and connected. If $k$ is perfect and $G$ is smooth, then $N$ is 
smooth as well, and its formation commutes with field extensions.
\end{thm}

In particular, every smooth connected algebraic group over a perfect 
field is an extension of an \textit{abelian variety}
(i.e., a smooth connected proper algebraic group) by a
smooth connected algebraic group which is affine, or equivalently
\textit{linear}. Both building blocks, abelian varieties and linear 
algebraic groups, have been extensively studied; see e.g. the books 
\cite{Mum} for the former, and \cite{Bo,Sp} for the latter.

Also, every algebraic group over a field is an extension of
a linear algebraic group by an \textit{anti-affine} algebraic
group $H$, i.e., every global regular function on $H$ is constant.
Clearly, every abelian variety is anti-affine; but the converse
turns out to be incorrect, unless $k$ is algebraic over a finite 
field (see \S \ref{subsec:saag}). Still, the structure 
of anti-affine groups over an arbitrary field can be reduced to that 
of abelian varieties; see \cite{Br1,SS} and also 
\S \ref{subsec:saag} again.

As a consequence, taking for $G$ an anti-affine group which is not 
an abelian variety, one sees that the natural maps $H \to G/N$
and $N \to G/H$ are generally not isomorphisms with the notation 
of the above theorems. But when $G$ is smooth and connected, one may 
combine these theorems to obtain more information on its structure,
see \S \ref{subsec:trd}.

The above theorems have a long history. Theorem \ref{thm:affi}
was first obtained by Rosenlicht in 1956 for smooth connected
algebraic groups, see \cite[Sec.~5]{Ro1}. The version presented
here is due to Demazure and Gabriel, see \cite[III.3.8]{DG}.
In the setting of smooth connected algebraic groups again,
Theorem \ref{thm:che} was announced by Chevalley in the early 
1950's. But he published his proof in 1960 only (see \cite{Ch2}),
as he had first to build up a theory of Picard and Albanese 
varieties. Meanwhile, proofs of Chevalley's theorem had been 
published by Barsotti and Rosenlicht (see \cite{Ba}, and
\cite[Sec.~5]{Ro1} again). The present version of Theorem 
\ref{thm:che} is a variant of a result of Raynaud (see 
\cite[IX.2.7]{Ra}).

The terminology and methods of algebraic geometry have much evolved
since the 1950's; this makes the arguments of Barsotti, Chevalley 
and Rosenlicht rather hard to follow. For this reason, modern proofs
of the above results have been made available recently: first,
a scheme-theoretic version of Chevalley's proof of his structure 
theorem by Conrad (see \cite{Co1}); then a version of Rosenlicht's 
proof for smooth connected algebraic groups over algebraically closed 
fields (see \cite[Chap.~2]{BSU} and also \cite{Mi}). 

In this text, we present scheme-theoretic proofs of Theorems 
\ref{thm:affi} and \ref{thm:che}, with (hopefully) modest 
prerequisites. More specifically, we assume familiarity with 
the contents of Chapters 2 to 5 of the book \cite{Li}, which will
be our standard reference for algebraic geometry over an arbitrary
field. Also, we do not make an explicit use of sheaves for the
fpqc or fppf topology, even if these notions are in the background 
of several arguments.

To make the exposition more self-contained, we have gathered
basic notions and results on group schemes over a field in 
Section \ref{sec:bnr}, referring to the books \cite{DG} and 
\cite{SGA3} for most proofs. Section \ref{sec:ptaffi} 
is devoted to the proof of Theorem \ref{thm:affi}, and Section 
\ref{sec:ptche} to that of Theorem \ref{thm:che}. Although 
the statements of both theorems are very similar, the first one 
is actually much easier. Its proof only needs a few preliminary 
results: some criteria for an algebraic group to be affine 
(\S \ref{subsec:aag}), the notion of affinization of 
a scheme (\S \ref{subsec:tat}) and a version of the 
rigidity lemma for ``anti-affine'' schemes (\S \ref{subsec:aaag}).
In contrast, the proof of Theorem \ref{thm:che} is based on
quite a few results on abelian varieties. Some of them are 
taken from \cite{Mum}, which will be our standard reference on 
that topic; less classical results are presented in
\S\S \ref{subsec:av} and \ref{subsec:at}.

Section \ref{sec:sfd} contains applications and developments
of the above structure theorems, in several directions. 
We begin with the Rosenlicht decomposition, which reduces 
somehow the structure of smooth connected algebraic groups 
to the linear and anti-affine cases (\S \ref{subsec:trd}). 
We then show in \S \ref{subsec:echs} that every homogeneous 
space admits a projective equivariant compactification. 
\S \ref{subsec:cag} gathers some known results on the structure 
of commutative algebraic groups. In \S \ref{subsec:sav}, 
we provide details on semi-abelian varieties, i.e., 
algebraic groups obtained as extensions of an abelian variety 
by a torus; these play an important r\^ole in various aspects 
of algebraic and arithmetic geometry. \S \ref{subsec:saag} 
is devoted to the classification of anti-affine algebraic groups,
based on results from \S\S \ref{subsec:cag} and 
\ref{subsec:sav}. The final \S \ref{subsec:cagc} contains 
developments on algebraic groups in positive characteristics, 
including a recent result of Totaro (see \cite[\S 2]{To}).

Further applications, of a geometric nature, are presented 
in Sections \ref{sec:tps} and \ref{sec:tags}. We give a brief 
overview of the Picard schemes of proper schemes in 
\S \ref{subsec:dbp}, referring to \cite{Kl} for a detailed 
exposition. \S \ref{subsec:spv} is devoted to structure
results for the Picard variety of a proper variety $X$,
in terms of the geometry of $X$. Likewise, \S \ref{subsec:bra}
surveys the automorphism group schemes of proper schemes.
\S \ref{subsec:bl} presents a useful descent property for
actions of connected algebraic groups. In the final 
\S \ref{subsec:vpcag}, based on \cite{Br2}, we show that
every smooth connected algebraic group over a perfect field
is the connected automorphism group of some normal projective
variety.

Each section ends with a paragraph of notes and references, 
which also contains brief presentations of recent work, and
some open questions. A general problem, which falls out of 
the scope of these notes, asks for a version of Theorem 
\ref{thm:che} in the setting of group schemes over (say) 
discrete valuation rings. A remarkable analogue of Theorem 
\ref{thm:affi} has been obtained by Raynaud in that setting 
(see \cite[VIB.12.10]{SGA3}). But Chevalley's structure theorem 
admits no direct generalization, as abelian varieties degenerate 
to tori. So finding a meaningful analogue of that theorem over 
a ring of formal power series is already an interesting challenge.

\medskip

\noindent
{\bf Notation and conventions}. 
Throughout this text, we fix a ground field $k$ with algebraic closure 
$\bar{k}$; the characteristic of $k$ is denoted by $\car(k)$.

We denote by $k_s$ the separable closure of $k$ in $\bar{k}$ 
and by $\Gamma$ the Galois group of $k_s$ over $k$.
Also, we denote by $k_i$ the perfect closure of $k$ in $\bar{k}$,
i.e., the largest subfield of $\bar{k}$ that is purely inseparable 
over $k$. If $\car(k) =0$ then $k_s = \bar{k}$ and $k_i = k$; 
if $\car(k) = p > 0$ then $k_i = \bigcup_{n\geq 0} k^{1/p^n}$.

We consider separated schemes over $\Spec(k)$ unless otherwise 
stated; we will call them $k$-schemes, or just schemes 
if this creates no confusion. Morphisms and products of schemes 
are understood to be over $\Spec(k)$. For any $k$-scheme $X$,
we denote by $\cO(X)$ the $k$-algebra of global sections
of the structure sheaf $\cO_X$.  Given a field extension
$K$ of $k$, we denote the $K$-scheme $X \times \Spec(K)$ 
by $X_K$.

We identify a scheme $X$ with its \textit{functor of points}
that assigns to any scheme $S$ the set $X(S)$ of morphisms
$f : S \to X$. When $S$ is affine, i.e., $S = \Spec(R)$ for an
algebra $R$, we also use the notation $X(R)$ for $X(S)$.
In particular, we have the set $X(k)$ of $k$-\textit{rational}
points.

A \textit{variety} is a geometrically integral scheme 
of finite type. The function field of a variety $X$ will be
denoted by $k(X)$.

\section{Basic notions and results}
\label{sec:bnr}

\subsection{Group schemes}
\label{subsec:gs}

\begin{definition}\label{def:gs}
A \textit{group scheme} is a scheme $G$ equipped 
with morphisms $m : G \times G \to G$,
$i : G \to G$ and with a $k$-rational point $e$, 
which satisfy the following condition:

For any scheme $S$, the set $G(S)$ is a group with multiplication 
map $m(S)$, inverse map $i(S)$ and neutral element $e$.
\end{definition}

This condition is equivalent to the commutativity of the following
diagrams:
\[
\xymatrix{
G \times G \times G \ar[r]^-{m \times \id}\ar[d]_{\id \times m} 
& G \times G \ar[d]^m \\
G \times G \ar[r]^-m & G \\}
\]
(i.e., $m$ is associative),
\[
\xymatrix{
G \ar[r]^-{e \times \id} \ar[dr]_{\id} & 
G \times G \ar[d]^m & 
\ar[l]_-{\id \times e} \ar[dl]^{\id} G \\
& G \\
}
\]
(i.e., $e$ is the neutral element), and
\[
\xymatrix{
G \ar[r]^-{\id \times i} \ar[dr]_{e \circ f} & 
G \times G \ar[d]^m & 
\ar[l]_-{i \times \id} \ar[dl]^{e \circ f} G \\
& G \\
}
\]
(i.e., $i$ is the inverse map). Here $f : G \to \Spec(k)$ 
denotes the structure map.

We will write for simplicity $m(x,y) = xy$ and $i(x) = x^{-1}$ 
for any scheme $S$ and points $x,y \in G(S)$.

\begin{remarks}\label{rem:gf}
(i) For any $k$-group scheme $G$, the base change under a field
extension $K$ of $k$ yields a $K$-group scheme $G_K$. 

(ii) The assignment $S \mapsto G(S)$ defines a \textit{group functor}, 
i.e., a contravariant functor from the category of schemes to 
that of groups. In fact, the group schemes are exactly those group 
functors that are representable (by a scheme). 

(iii) Some natural group functors are not representable. For example, 
consider the functor that assigns to any scheme $S$ the group
$\Pic(S)$  of isomorphism classes of invertible sheaves on $S$,
and to any morphism of schemes $f : S' \to S$, the pull-back
map $f^* : \Pic(S) \to \Pic(S')$. This yields a commutative group 
functor that we still denote by $\Pic$. For any local ring $R$,
we have $\Pic(\Spec(R)) = 0$. If $\Pic$ is represented by a scheme 
$X$, then every morphism $\Spec(R) \to X$ is constant for $R$ 
local; hence every morphism $S \to X$ is locally constant. 
As a consequence, $\Pic(\bP^1) = \Hom(\bP^1,X) = 0$, 
a contradiction.
\end{remarks}

\begin{definition}\label{def:sgs}
Let $G$ be a group scheme. A \textit{subgroup scheme} 
of $G$ is a (locally closed) subscheme $H$ 
such that $H(S)$ is a subgroup of $G(S)$ for any 
scheme $S$. We say that $H$ is \textit{normal in} $G$, if $H(S)$
is a normal subgroup of $G(S)$ for any scheme $S$.
We then write $H \trianglelefteq G$.
\end{definition}

\begin{definition}\label{def:mor}
Let $G$, $H$ be group schemes. A morphism 
$f : G \to H$ is called a \textit{homomorphism} if 
$f(S) : G(S) \to H(S)$ is a group homomorphism for any 
scheme $S$. 

The \textit{kernel} of the homomorphism $f$ is the group functor 
$\Ker(f)$ such that $\Ker(f)(S) = \Ker(f(S) : G(S) \to H(S))$. 
It is represented by a closed normal subgroup scheme of $G$, 
the fiber of $f$ at the neutral element of $H$.
\end{definition}

\begin{definition}\label{def:ag}
An \textit{algebraic group} over $k$ is a $k$-group scheme 
of finite type.
\end{definition}

This notion of algebraic group is somewhat more general 
than the classical one. More specifically, the 
``algebraic groups defined over $k$'' in the sense of 
\cite{Bo,Sp} are the  \textit{geometrically reduced} $k$-group 
schemes of finite type. Yet both notions coincide in characteristic
$0$, as a consequence of the following result of Cartier:

\begin{theorem}\label{thm:car}
When $\car(k) = 0$, every algebraic group over $k$ is reduced.
\end{theorem}

\begin{proof}
See \cite[II.6.1.1]{DG} or \cite[VIB.1.6.1]{SGA3}. A self-contained
proof is given in \cite[p.~101]{Mum}.
\end{proof}

\begin{example}\label{ex:add}
The \textit{additive group} $\bG_a$ is the affine line $\bA^1$ 
equipped with the addition. More specifically, we have 
$m(x,y) = x + y$ and $i(x) = -x$ identically, and $e = 0$.

Consider a subgroup scheme $H \subseteq \bG_a$. If $H \neq \bG_a$,
then $H$ is the zero scheme $V(P)$ for some non-constant polynomial 
$P \in \cO(\bG_a) = k[x]$;
we may assume that $P$ has leading coefficient $1$. We claim 
that $P$ is an \textit{additive polynomial}, i.e.,
\[ P(x + y) = P(x) + P(y) \] 
in the polynomial ring $k[x,y]$.

To see this, note that $P(0) = 0$ as $0 \in H(k)$, and 
\[ P(x + y) \in (P(x), P(y)) \] 
(the ideal of $k[x,y]$ generated by $P(x)$ and $P(y)$), 
as the addition $\bG_a \times \bG_a \to \bG_a$ sends $H \times H$ 
to $H$. Thus, there exist $A(x,y)$, $B(x,y) \in k[x,y]$ such that 
\[ P(x + y) - P(x) - P(y) = A(x,y) \, P(x) + B(x,y) \, P(y). \]
Dividing $A(x,y)$ by $P(y)$, we may assume that 
$\deg_y A(x,y) < \deg(P)$ with an obvious notation. 
Since $\deg_y (P(x + y) - P(x) - P(y)) < \deg(P)$, it follows 
that $B = 0$. Likewise, we obtain $A = 0$; this yields the claim.

We now determine the additive polynomials. The derivative of any 
such polynomial $P$ satisfies $P'(x + y) = P'(x)$, hence $P'$ is
constant. When $\car(k) = 0$, we obtain $P(x) = a x$
for some $a \in k$, hence $H$ is just the (reduced) point $0$.
Alternatively, this follows from Theorem \ref{thm:car}, since $H(\bar{k})$ 
is a finite subgroup of $(\bar{k},+)$, and hence is trivial. 

When $\car(k) = p > 0$, we obtain $P(x) = a_0 x + P_1(x^p)$,
where $P_1$ is again an additive polynomial. By induction on
$\deg(P)$, it follows that 
\[
P(x) = a_0 x + a_1 x^p + \cdots + a_n x^{p^n}
\] 
for some positive integer $n$ and $a_0, \ldots, a_n \in k$. As a 
consequence, $\bG_a$ has many subgroup schemes in positive
characteristics; for example,  
\[ \alpha_{p^n} := V(x^{p^n}) \]
is a non-reduced subgroup scheme supported at $0$. 

Note finally that the additive polynomials are exactly the
endomorphisms of $\bG_a$, and their kernels yield all
subgroup schemes of that group scheme (in arbitrary characteristics).
\end{example}

\begin{example}\label{ex:mult}
The \textit{multiplicative group} $\bG_m$ is the punctured affine 
line $\bA^1 \setminus 0$ equipped with the multiplication:
we have $m(x,y) = x y$ and $i(x) = x^{-1}$ identically, and $e = 1$.

The subgroup schemes of $\bG_m$ turn out to be $\bG_m$ and 
the subschemes 
\[ \mu_n := V(x^n - 1) \] 
of $n$th roots of unity, where $n$ is a positive integer; 
these are the kernels of the endomorphisms $x \mapsto x^n$ 
of $\bG_m$. Moreover, $\mu_n$ is reduced if and only if 
$n$ is prime to $\car(k)$.  
\end{example}

\begin{example}\label{ex:GL}
Given a vector space $V$, the \textit{general linear group}
$\GL(V)$ is the group functor that assigns to any scheme $S$,
the automorphism group of the sheaf of $\cO_S$-modules 
$\cO_S \otimes_k V$. 
When $V$ is of finite dimension $n$, the choice of a basis 
identifies $V$ with $k^n$ and $\GL(V)(S)$ with $\GL_n(\cO(S))$, 
the group of invertible $n \times n$ matrices with coefficients
in the algebra $\cO(S)$. It follows that $\GL(V)$ is represented 
by an open affine subscheme of the affine scheme $\bA^{n^2}$
(associated with the linear space of $n \times n$ matrices),
the complement of the zero scheme of the determinant. 
This defines a group scheme $\GL_n$, which is smooth,
connected, affine and algebraic.
\end{example}

\begin{definition}\label{def:lin}
A group scheme is \textit{linear} if it is isomorphic to
a closed subgroup scheme of $\GL_n$ for some positive integer $n$.
\end{definition}

Clearly, every linear group scheme is algebraic and affine. 
The converse also holds, see Proposition \ref{prop:lin} below.

Some natural classes of group schemes arising from geometry,
such as automorphism group schemes and Picard schemes of
proper schemes, are generally not algebraic. Yet they turn out
to be locally of finite type; this motivates the following:

\begin{definition}\label{def:lag}
A \textit{locally algebraic group} over $k$ is a $k$-group 
scheme, locally of finite type.
\end{definition}

\begin{proposition}\label{prop:smooth}
The following conditions are equivalent for a locally algebraic 
group $G$ with neutral element $e$:

\begin{enumerate}

\item $G$ is smooth.

\item $G$ is geometrically reduced.

\item $G_{\bar{k}}$ is reduced at $e$.

\end{enumerate}

\end{proposition}

\begin{proof}
Clearly, (1)$\Rightarrow$(2)$\Rightarrow$(3). We now show 
that (3)$\Rightarrow$(1). For this, we may replace $G$ with 
$G_{\bar{k}}$ and hence assume that $k$ is algebraically closed.

Observe that for any $g \in G(k)$, the local ring $\cO_{G,g}$ 
is isomorphic to $\cO_{G,e}$ as the left multiplication
by $g$ in $G$ is an automorphism that sends $e$ to $g$. 
It follows that $\cO_{G,g}$ is reduced; hence every open
subscheme of finite type of $G$ is reduced as well. Since 
$G$ is locally of finite type, it must be reduced, too. Thus,
$G$ contains a smooth $k$-rational point $g$. By arguing
as above, we conclude that $G$ is smooth.
\end{proof}

\subsection{Actions of group schemes}
\label{subsec:ags}

\begin{definition}\label{def:action}
An \textit{action} of a group scheme $G$ on a scheme $X$ is
a morphism $a : G \times X \to X$ such that the map $a(S)$ yields 
an action of the group $G(S)$ on the set $X(S)$, for any scheme $S$.
\end{definition}

This condition is equivalent to the commutativity of the following
diagrams:
\[
\xymatrix{
G \times G \times X \ar[r]^-{m \times \id}\ar[d]_{\id \times a} 
& G \times X \ar[d]^a \\
G \times X \ar[r]^-a & X \\}
\]
(i.e., $a$ is ``associative''), and 
\[
\xymatrix{
X \ar[r]^-{e \times \id} \ar[dr]_{\id} & G \times X \ar[d]^a \\
& X \\
}
\]
(i.e., the neutral element acts via the identity).

We may view a $G$-action on $X$ as a homomorphism of group
functors 
\[ a : G \longrightarrow \Aut_X, \] 
where $\Aut_X$ denotes the \textit{automorphism group functor} 
that assigns to any scheme $S$, the group of automorphisms of 
the $S$-scheme $X \times S$. The $S$-points of $\Aut_X$ are 
those morphisms $f : X \times S \to X$ such that the map 
\[ f \times p_2 : X \times S \longrightarrow X \times S, 
\quad (x,s) \longmapsto (f(x,s),s) \] 
is an automorphism of $X \times S$; they may be viewed 
as families of automorphisms of $X$ parameterized by $S$.

\begin{definition}\label{def:Gs}
A scheme $X$ equipped with an action $a$ of $G$ will be called 
a \textit{$G$-scheme}; we then write for simplicity 
$a(g,x) = g \cdot x$ for any scheme $S$ and $g \in G(S)$, 
$x \in X(S)$.

The action is \textit{trivial} if $a$ is the second projection 
$p_2 : G \times X \to X$; equivalently, $g \cdot x = x$ identically.
\end{definition}

\begin{remark}\label{rem:action}
For an arbitrary action $a$, we have a commutative triangle
\[ 
\xymatrix{
G \times X \ar[r]^-{u} \ar[dr]_{a} & G \times X \ar[d]^{p_2} \\
& X, \\}
\]
where $u(g,x) := (g,a(g,x))$. Since $u$ is an automorphism
(with inverse the map $(g,x) \mapsto (g,a(g^{-1},x))$), it 
follows that the morphism $a$ shares many properties of the
scheme $G$. For example, $a$ is always faithfully flat; 
it is smooth if and only if $G$ is smooth.

In particular, the multiplication $m : G \times G \to G$
is faithfully flat.
\end{remark}

\begin{definition}\label{def:equi}
Let $X$, $Y$ be $G$-schemes with actions $a$, $b$.
A \textit{morphism of $G$-schemes} $\varphi : X \to Y$ is a morphism 
of schemes such that the following square commutes:
\[
\xymatrix{
G \times X \ar[r]^-a \ar[d]_{\id \times \varphi} 
& X \ar[d]^{\varphi} \\
G \times Y \ar[r]^-b & Y. \\}
\]
In other words, $\varphi(g \cdot x) = g \cdot \varphi(x)$
identically; we then say that $\varphi$ is \textit{$G$-equivariant}. 
When $Y$ is equipped with the trivial action of $G$, 
we say that $\varphi$ is \textit{$G$-invariant}.
\end{definition}

\begin{definition}\label{def:norcent}
Let $X$ be a $G$-scheme with action $a$, and $Y$ a closed subscheme of $X$. 

The \textit{normalizer} (resp. \textit{centralizer}) of $Y$ 
in $G$ is the group functor $\N_G(Y)$ (resp. $\C_G(Y)$) 
that associates with any scheme $S$, the set of those $g \in G(S)$  
which induce an automorphism of $Y \times S$ 
(resp. the identity of $Y \times S$). 

The \textit{kernel} of $a$ is the centralizer of $X$ in $G$,
or equivalently, the kernel of the corresponding homomorphism
of group functors.

The action $a$ is \textit{faithful} if its kernel is trivial;
equivalently, for any scheme $S$, every non-trivial element
of $G(S)$ acts non-trivially on $X \times S$.

The \textit{fixed point functor} of $X$ is the subfunctor 
$X^G$ that associates with any scheme $S$, the set of all
$x \in X(S)$ such that for any $S$-scheme $S'$ and any 
$g \in G(S')$, we have $g \cdot x = x$.
\end{definition}

\begin{theorem}\label{thm:norcent}
Let $G$ be a group scheme acting on a scheme $X$. 

\begin{enumerate}

\item The normalizer and centralizer of any closed subscheme 
$Y \subseteq X$ are represented by closed subgroup schemes of $G$.

\item The functor of fixed points is represented by 
a closed subscheme of $X$.

\end{enumerate}

\end{theorem}

\begin{proof}
See \cite[II.1.3.6]{DG} or \cite[VIB.6.2.4]{SGA3}.
\end{proof}

In particular, $\N_G(Y)$ is the largest subgroup scheme of $G$
that acts on $Y$, and $\C_G(Y)$ is the kernel of this action.
Moreover, $X^G$ is the largest subscheme of $X$ on which $G$ acts
trivially. We also say that $\N_G(Y)$ \textit{stabilizes} $Y$, 
and $\C_G(Y)$ \textit{fixes $Y$ pointwise}. 
 
When $Y$ is just a $k$-rational point $x$, we have 
$\N_G(Y) = \C_G(Y) =: \C_G(x)$. This is the \textit{stabilizer}
of $x$ in $G$, which is clearly represented by a closed
subgroup scheme of $G$: the fiber at $x$ of the 
\textit{orbit map}
\[ a_x : G \longrightarrow X, \quad g \longmapsto g \cdot x. \] 
We postpone the definition of orbits to \S \ref{subsec:hsq}, 
where homogeneous spaces are introduced; we now record classical 
properties of the orbit map:

\begin{proposition}\label{prop:orb}
Let $G$ be an algebraic group acting on a scheme of finite type
$X$ via $a$. 

\begin{enumerate}

\item The image of the orbit map $a_x$ is locally closed
for any closed point $x \in X$.

\item If $k$ is algebraically closed and $G$ is smooth, then 
there exists $x \in X(k)$ such that the image of $a_x$ is closed.

\end{enumerate}

\end{proposition}

\begin{proof}
(1) Consider the natural map $\pi : X_{\bar{k}} \to X$.
Since $\pi$ is faithfully flat and quasi-compact, it suffices 
to show that $\pi^{-1}(a_x(G))$ is locally closed (see e.g. 
\cite[IV.2.3.12]{EGA}). As $\pi^{-1}(a_x(G))$ is the image of 
the orbit map $(a_x)_{\bar{k}}$, we may assume $k$ algebraically closed.
Then $a_x(G)$ is constructible, and hence contains a dense open 
subset $U$ of its closure. The pull-back $a_x^{-1}(U)$ is a 
non-empty open subset of the underlying topological space of $G$; 
hence that space is covered by the translates $g a_x^{-1}(U)$, 
where $g \in G(k)$. It follows that $a_x(G)$ is covered by 
the translates $g U$, and hence is open in its closure.  

(2) Choose a closed $G$-stable subscheme $Y \subseteq X$ of 
minimal dimension and let $x \in Y(k)$. If $a_x(G)$ is not
closed, then $Z := \overline{a_x(G)} \setminus a_x(G)$ 
(equipped with its reduced subscheme structure) is 
a closed subscheme of $Y$, stable by $G(k)$. Since the 
normalizer of $Z$ is representable and $G(k)$ is dense in $G$, 
it follows that $Z$ is stable by $G$. But 
$\dim(Z) < \dim(a_x(G)) \leq \dim(Y)$, a contradiction.   
\end{proof}

\begin{example}\label{ex:actions}
Every group scheme $G$ acts on itself by left multiplication, via 
\[ \lambda : G \times G \longrightarrow G, 
\quad (x,y) \longmapsto xy. \]
It also acts by right multiplication, via 
\[ \rho : G \times G \longrightarrow G, 
\quad (x,y) \longmapsto y x^{-1} \]
and by conjugation, via 
\[ \Int : G \times G \longrightarrow G, 
\quad (x,y) \longmapsto x y x^{-1}. \]
The actions $\lambda$ and $\rho$ are both faithful. The kernel 
of $\Int$ is the \textit{center of $G$}.
\end{example}

\begin{definition}\label{def:sd}
Let $G$, $H$ be two group schemes and 
$a : G \times H \to H$ an action by group automorphisms, i.e., 
we have $g \cdot (h_1 h_2) = (g \cdot h_1) (g \cdot h_2)$ identically.  
The \textit{semi-direct product} $G \ltimes H$ is the 
scheme $G \times H$ equipped with the multiplication such that
\[ (g, h) \cdot (g',h') = (gg', (g'^{-1} \cdot h) h'), \]
the neutral element $e_G \times e_H$, and the inverse such that 
$(g,h)^{-1} = (g^{-1}, g \cdot h^{-1})$. 
\end{definition}

By using the Yoneda lemma, one may readily check that $G \ltimes H$ 
is a group scheme. Moreover, $H$ (identified with its image in 
$G \ltimes H$ under the closed immersion $h \mapsto (e_G,h)$) 
is a closed normal subgroup scheme, 
and $G$ (identified with its image under the closed immersion
$g \mapsto (g,e_H)$) is a closed subgroup scheme having a retraction
\[ r : G \ltimes H \longrightarrow G, \quad (g,h) \longmapsto g \]
with kernel $H$. The given action of $G$ on $H$ is identified
with the action by conjugation in $G \ltimes H$.

\begin{remarks}\label{rem:sd}
(i) With the above notation, $G$ is a normal subgroup scheme 
of $G \ltimes H$ if and only if $G$ acts trivially on $H$.

(ii) Conversely, consider a group scheme $G$ and two closed 
subgroup schemes $N$, $H$ of $G$ such that $H$ normalizes 
$N$ and the inclusion of $H$ in $G$ admits a retraction
$r : G \to H$ which is a homomorphism with kernel $N$. 
Form the semi-direct product $H \ltimes N$, 
where $H$ acts on $N$ by conjugation.
Then one may check that the multiplication map 
\[ H \ltimes N \longrightarrow G, 
\quad (x,y) \longmapsto xy \] 
is an isomorphism of group schemes, with inverse being 
the morphism 
\[ G \longrightarrow H \ltimes N, 
\quad z \longmapsto (r(z),r(z)^{-1}z). \]
\end{remarks}

\subsection{Linear representations}
\label{subsec:lr}

\begin{definition}\label{def:rep}
Let $G$ be a group scheme and $V$ a vector space. 
A \textit{linear representation $\rho$ of $G$ in $V$} 
is a homomorphism of group functors $\rho: G \to \GL(V)$. 
We then say that $V$ is a $G$-\textit{module}. 
\end{definition}

More specifically, $\rho$ assigns to any scheme $S$ and
any $g \in G(S)$, an automorphism $\rho(g)$ of the
sheaf of $\cO_S$-modules $\cO_S \otimes_k V$, functorially
on $S$. Note that $\rho(g)$ is uniquely determined
by its restriction to $V$ (identified with 
$1 \otimes_k V \subseteq \cO(S) \otimes_k V$, where $1$
denotes the unit element of the algebra $\cO(S)$), 
which yields a linear map $V \to \cO(S) \otimes_k V$.

A linear subspace $W \subseteq V$ is a $G$-\textit{submodule}
if each $\rho(g)$ normalizes $\cO_S \otimes_k W$.
More generally, the notions of quotients, exact sequences,
tensor operations of linear representations of abstract 
groups extend readily to the setting of group schemes.

\begin{examples}\label{ex:rep}
(i) When $V = k^n$ for some positive integer $n$, 
a linear representation of $G$ in $V$ is
a homomorphism of group schemes $\rho : G \to \GL_n$
or equivalently, a linear action of $G$ on the affine
space $\bA^n$.

(ii) Let $X$ be an affine $G$-scheme with action $a$.
For any scheme $S$ and $g \in G(S)$, we define an automorphism
$\rho(g)$ of the $\cO_S$-algebra $\cO_S \otimes_k \cO(X)$ by setting
\[ \rho(g)(f) := f \circ a(g^{-1}) \]
for any $f \in \cO(X)$. This yields a representation $\rho$ of $G$ 
in $\cO(X)$, which uniquely determines the action in view of 
the anti-equivalence of categories between affine schemes and algebras. 

For instance, if $G$ acts linearly on a finite-dimensional vector
space $V$, then $\cO(V) \cong \Sym(V^*)$ (the symmetric algebra
of the dual vector space) as $G$-modules.

(iii) More generally, given any $G$-scheme $X$, one may define 
a representation $\rho$ of $G$ in $\cO(X)$ as above. 
But in general, the $G$-action on $X$ is not uniquely determined 
by $\rho$. For instance, if $X$ is a proper $G$-variety, then 
$\cO(X) = k$ and hence $\rho$ is trivial.  
\end{examples}

\begin{lemma}\label{lem:aff}
Let $X$, $Y$ be quasi-compact schemes. Then the map
\[
\cO(X) \otimes_k \cO(Y) \longrightarrow \cO(X \times Y), 
\quad f \otimes g \longmapsto ((x,y) \mapsto f(x) \, g(y))
\]
is an isomorphism of algebras.
In particular, we have a canonical isomorphism
\[ 
\cO(X) \otimes_k R \stackrel{\cong}{\longrightarrow} \cO(X_R)
\] 
for any quasi-compact scheme $X$ and any algebra $R$.
\end{lemma}

\begin{proof}
The assertion is well-known when $X$ and $Y$ are affine.

When $X$ is affine and $Y$ is arbitrary, we may choose a finite
open covering $(V_i)_{1 \leq i \leq n}$ of $Y$; then the intersections
$V_i \cap V_j$ are affine as well. Also, we have an exact sequence
\[ 0 \longrightarrow \cO(Y) \longrightarrow \prod_i \cO(V_i)
\stackrel{d_Y}{\longrightarrow} \prod_{i,j} \cO(V_i \cap V_j), \]
where 
$d_Y((f_i)_i) := (f_i \vert_{V_i \cap V_j} - f_j \vert_{V_i \cap V_j})_{i,j}$.
Tensoring with $\cO(X)$ yields an exact sequence
\[ 0 \longrightarrow \cO(X) \otimes_k \cO(Y) \longrightarrow 
\prod_i \cO(X \times V_i)
\stackrel{d_{X, Y}}{\longrightarrow} 
\prod_{i,j} \cO(X \times (V_i \cap V_j)), \]
where $d_{X, Y}$ is defined similarly. Since the $X \times V_i$ 
form an open covering of $X \times Y$, the kernel of 
$d_{X, Y}$ is $\cO(X \times Y)$; this proves the assertion 
in this case.

In the general case, we choose a finite open affine covering 
$(U_i)_{1 \leq i \leq m}$ of $X$ and obtain an exact sequence
\[ 0 \longrightarrow \cO(X) \otimes_k \cO(Y) \longrightarrow 
\prod_i \cO(U_i \times Y) \longrightarrow
\prod_{i,j} \cO((U_i \cap U_j) \times Y), \]
by using the above step. The assertion follows similarly.
\end{proof}

The quasi-compactness assumption in the above lemma is 
a mild finiteness condition, which is satisfied e.g. for
affine or noetherian schemes.

\begin{proposition}\label{prop:rep}
Let $G$ be an algebraic group and $X$ a $G$-scheme
of finite type. Then the $G$-module $\cO(X)$ is the union 
of its finite-dimensional submodules. 
\end{proposition}

\begin{proof}
The action map $a : G \times X \to X$ yields a homomorphism
of algebras $a^\# : \cO(X) \to \cO(G \times X)$. 
In view of Lemma \ref{lem:aff}, we may view $a^\#$ 
as a homomorphism $\cO(X) \to \cO(G) \otimes_k \cO(X)$. 
Choose a basis $(\varphi_i)$ of the vector space $\cO(G)$. 
Then for any $f \in \cO(X)$, there exists a family 
$(f_i)$ of elements of $\cO(X)$ such that $f_i \neq 0$ 
for only finitely many $i$'s, and
\[ a^\#(f) = \sum_i \varphi_i \otimes f_i. \]
Thus, we have identically
\[ \rho(g)(f) = \sum_i \varphi_i(g^{-1}) \, f_i. \]
Applying this to the action of $G$ on itself by left 
multiplication, we obtain the existence of families 
$(\gamma_{ij})_j$, $(\psi_{ij})_j$ of elements of $\cO(G)$ such 
that $\gamma_{ij} \neq 0$ for only finitely many $j$'s, and
\[ \varphi_i(h^{-1} g^{-1}) = 
\sum_j \gamma_{ij}(g^{-1}) \, \psi_{ij}(h^{-1}) \]
identically on $G \times G$. It follows that
\[ \rho(g) \rho(h) (f) = \sum_{i,j} \gamma_{ij}(g^{-1}) \,
\psi_{ij}(h^{-1}) \, f_i. \]
As a consequence, the span of the $f_i$'s in $\cO(G)$
is a finite-dimensional $G$-submodule, which contains 
$f = \sum_i \varphi_i(e) \, f_i$.
\end{proof}

\begin{proposition}\label{prop:emb}
Let $G$ be an algebraic group and $X$ an affine $G$-scheme 
of finite type. Then there exists a finite-dimensional 
$G$-module $V$ and a closed $G$-equivariant immersion 
$\iota : X \to V$.
\end{proposition}

\begin{proof}
We may choose finitely many generators $f_1, \ldots,f_n$
of the algebra $\cO(X)$. By Proposition \ref{prop:rep},
each $f_i$ is contained in some finite-dimensional 
$G$-submodule $W_i \subseteq \cO(X)$. Thus, 
$W := W_1 + \cdots + W_n$ is a finite-dimensional 
$G$-submodule of $\cO(X)$, which generates that algebra. 
This defines a surjective homomorphism
of algebras $\Sym(W) \to \cO(X)$, equivariant for the natural
action of $G$ on $\Sym(W)$. In turn, this yields the desired
closed equivariant immersion.
\end{proof}

Examples of linear representations arise from the action
of the stabilizer of a $k$-rational point on its infinitesimal
neighborhoods, which we now introduce.

\begin{example}\label{ex:inf}
Let $G$ be an algebraic group acting on a scheme $X$ via $a$
and let $Y \subseteq X$ be a closed subscheme.
For any non-negative integer $n$, consider the 
\textit{$n$th infinitesimal neighborhood $Y_{(n)}$ of $Y$ in $X$}; 
this is the closed subscheme of $X$ with ideal sheaf $\cI_Y^{n+1}$, 
where $\cI_Y \subseteq \cO_X$ denotes the ideal sheaf of $Y$. 
The subschemes $Y_{(n)}$ form an increasing sequence, starting with 
$Y_{(0)} = Y$. 

Next, assume that $G$ stabilizes $Y$. Then 
$a^{-1}(Y) = p_2^{-1}(Y)$, and hence
\[ a^{-1}(\cI_Y) \cO_{G \times X} = p_2^{-1}(\cI_Y) \cO_{G \times X}. \]
It follows that
\[ a^{-1}(\cI_Y^{n+1}) \cO_{G \times X} = 
p_2^{-1}(\cI_Y^{n+1}) \cO_{G \times X}. \]
Thus, $a^{-1}(Y_{(n)}) = p_2^{-1}(Y_{(n)})$, i.e., $G$ stabilizes 
$Y_{(n)}$ as well.

As a consequence, given a (say) locally noetherian $G$-scheme $X$ 
equipped with a $k$-rational point $x = \Spec(\cO_{X,x}/\fm_x)$, 
the algebraic group $\C_G(x)$ acts on each 
infinitesimal neighborhood $x_{(n)} = \Spec(\cO_{X,x}/\fm_x^{n+1})$, 
which is a finite scheme supported at $x$. This yields a linear 
representation $\rho_n$ of $G$ on $\cO_{X,x}/\fm_x^{n+1}$ 
by algebra automorphisms. In particular, $\C_G(x)$ acts linearly 
on $\fm_x/\fm_x^2$ and hence on the Zariski tangent space, 
$T_x(X) = (\fm_x/\fm_x^2)^*$.
\end{example}

Applying the above construction to the action of $G$ on itself 
by conjugation, which fixes the point $e$, we obtain a linear 
representation of $G$ in $\fg := T_e(G)$, called the
\textit{adjoint representation} and denoted by
\[ \Ad : G \longrightarrow \GL(\fg). \]
This yields in turn a linear map
\[ \ad := d\Ad_e : \fg \longrightarrow \End(\fg) \]
(where the right-hand side denotes the space of endomorphisms 
of the vector space $\fg$), and hence a bilinear map 
\[ [\, , \, ] : \fg \times \fg \longrightarrow \fg, \quad
(x,y) \longmapsto \ad(x)(y). \]
One readily checks that $[x,x] = 0$ identically; also, $[\, , \, ]$
satisfies the Jacobi identity (see e.g. \cite[II.4.4.5]{DG}). 
Thus, $(\fg,[\, , \, ])$ is a Lie algebra, called the 
\textit{Lie algebra of $G$}; we denote it by $\Lie(G)$.

Denote by $T_G$ the tangent sheaf of $G$, i.e., the sheaf of
derivations of $\cO_G$. By \cite[II.4.4.6]{DG}, we may also 
view $\Lie(G)$ as the Lie algebra $H^0(G,T_G)^G = \Der^G(\cO_G)$ 
consisting of those global derivations of $\cO_G$ 
that are invariant under the $G$-action via right 
multiplication; this induces an isomorphism
\[ T_G \cong \cO_G \otimes_k \Lie(G). \] 
We have $\dim(G) \leq \dim \Lie(G)$ with equality if and only
if $G$ is smooth, as follows from Proposition \ref{prop:smooth}.
Also, every homomorphism of algebraic groups $f : G \to H$ 
differentiates to a homomorphism of Lie algebras 
\[ \Lie(f) := df_{e_G} : \Lie(G) \longrightarrow \Lie(H). \]  
More generally, every action $a$ of $G$ on a scheme $X$ yields
a homomorphism of Lie algebras 
\[
\Lie(a) : \Lie(G) \longrightarrow H^0(X,T_X) = \Der(\cO_X) \] 
(see \cite[II.4.4]{DG}).

When $\car(k) = p > 0$, the $p$th power of any derivation is 
a derivation; this equips $\Lie(G) = \Der^G(\cO_G)$ with an
additional structure of $p$-\textit{Lie algebra}, also called
\textit{restricted Lie algebra} (see \cite[II.7.3]{DG}).
This structure is preserved by the above homomorphisms.

\subsection{The neutral component}
\label{subsec:tnc}

Recall that a scheme $X$ is \'etale (over $\Spec(k)$) 
if and only if its underlying topological space is discrete and 
the local rings of $X$ are finite separable extensions of $k$
(see e.g. \cite[I.4.6.1]{DG}).
In particular, every \'etale scheme is locally of finite type. 
Also, $X$ is \'etale if and only if the $k_s$-scheme
$X_{k_s}$ is constant; moreover, the assignment $X \mapsto X(k_s)$
yields an equivalence from the category of \'etale schemes
(and morphisms of schemes) to that of discrete topological
spaces equipped with a continuous action of the Galois group 
$\Gamma$ (and $\Gamma$-equivariant maps); see 
\cite[I.4.6.2, I.4.6.4]{DG}.

Next, let $X$ be a scheme, locally of finite type. 
By \cite[I.4.6.5]{DG}, there exists an \'etale scheme 
$\pi_0(X)$ and a morphism 
\[ \gamma = \gamma_X : X \longrightarrow \pi_0(X) \] 
such that every morphism of schemes $f: X \to Y$, 
where $Y$ is \'etale, factors uniquely through $\gamma$. 
Moreover, $\gamma$ is faithfully flat and its fibers are exactly 
the connected components of $X$. The formation of the
\textit{scheme of connected components} $\pi_0(X)$ commutes 
with field extensions in view of \cite[I.4.6.7]{DG}.

As a consequence, given a morphism of schemes
$f : X \to Y$, where $X$ and $Y$ are locally of finite type,
we obtain a commutative diagram 
\[
\xymatrix{
X \ar[r]^-f \ar[d]_{\gamma_X} 
& Y \ar[d]^{\gamma_Y} \\
\pi_0(X) \ar[r]^-{\pi_0(f)} & \pi_0(Y), \\}
\]
where $\pi_0(f)$ is uniquely determined. Applying this 
construction to the two projections $p_1 : X \times Y \to X$, 
$p_2: X \times Y \to Y$, we obtain a canonical morphism 
$\pi_0(X \times Y) \to \pi_0(X) \times \pi_0(Y)$,
which is in fact an isomorphism (see \cite[I.4.6.10]{DG}).
In particular, the formation of the scheme of connected components
commutes with finite products.

It follows easily that for any locally algebraic group scheme $G$, 
there is a unique group scheme structure on $\pi_0(G)$ such 
that $\gamma$ is a homomorphism. Moreover, given an action $a$ of $G$ 
on a scheme $X$, locally of finite type, we have a compatible 
action $\pi_0(a)$ of $\pi_0(G)$ on $\pi_0(X)$.

\begin{theorem}\label{thm:neut}
Let $G$ be a locally algebraic group and denote by $G^0$ 
the connected component of $e$ in $G$.

\begin{enumerate}

\item $G^0$ is the kernel of $\gamma : G \to \pi_0(G)$.

\item The formation of $G^0$ commutes with field extensions.

\item $G^0$ is a geometrically irreducible algebraic group.

\item The connected components of $G$ are irreducible, of finite
type and of the same dimension. 

\end{enumerate}

\end{theorem}

\begin{proof}
(1) This holds as the fibers of $\gamma$ are the connected components
of $G$.

(2) This follows from the fact that the formation of $\gamma$ 
commutes with field extensions.

(3) Consider first the case of an algebraically closed field $k$. Then 
the reduced neutral component $G^0_\red$ is smooth by Proposition 
\ref{prop:smooth}, and hence locally irreducible. Since $G^0_\red$ 
is connected, it is irreducible. 

Returning to an arbitrary ground field, $G^0$ is geometrically
irreducible by (2) and the above step. We now show that $G^0$
is of finite type. Choose a non-empty open subscheme of finite 
type $U \subseteq G^0$; then $U$ is dense in $G^0$. Consider 
the multiplication map of $G^0$, and its pull-back
\[ n : U \times U \longrightarrow G^0. \]
We claim that $n$ is faithfully flat.

Indeed, $n$ is flat by Remark \ref{rem:action}. To show that
$n$ is surjective, let $g \in G^0(K)$ for some field extension
$K$ of $k$. Then $U_K \cap g \, i( U_K)$ is non-empty, since
$G^0_K$ is irreducible. Thus, there exists a field extension
$L$ of $k$ and $x,y \in U(L)$ such that $g = x y^{-1}$.
This yields the claim. 

By that claim and the quasi-compactness of $U \times U$,
we see that $G^0$ is quasi-compact as well. But $G^0$  is also 
locally of finite type; hence it is of finite type.

(4) Let $X \subseteq G$ be a connected component. Since $G$
is locally of finite type, we may choose a closed point $x \in X$;
then the residue field $\kappa(x)$ is a finite extension of $k$.
Thus, we may choose a field extension $K$ of $\kappa(x)$,
which is finite and stable under $\Aut_k(\overline{\kappa(x)})$.
The structure map $\pi : X_K \to X$ is finite and faithfully flat, 
hence open and closed; moreover, every point $x'$ of $\pi^{-1}(x)$ 
is $K$-rational (as $\kappa(x')$ is a quotient field of 
$K \otimes_k \kappa(x)$). Thus, the fiber of $\gamma_K$
at $x'$ is the translate $x' G^0_K$ (since 
$x'^{-1} \gamma_K^{-1} \gamma_K(x')$ is a connected component
of $G_K$ and contains $e$). As a consequence, 
$\pi(x' G^0_K)$ is irreducible, open and closed in $G$, and
contains $\pi(x') = x$; so $\pi(x' G^0_K) = X$. This shows
that $X$ is irreducible of dimension $\dim(G^0)$. To check that
$X$ is of finite type, observe that 
$X_K = \bigcup_{x' \in \pi^{-1}(x)} \, x' G^0_K$ is of finite
type, and apply descent theory (see \cite[IV.2.7.1]{EGA}).

\end{proof}

With the notation and assumptions of the above theorem,
$G^0$ is called the \textit{neutral component} of $G$.
Note that $G$ is equidimensional of dimension $\dim(G^0)$.

\begin{remarks}\label{rem:neut}
(i) Let $G$ be a locally algebraic group acting on a scheme
$X$, locally of finite type. If $k$ is separably closed, then
every connected component of $X$ is stable by $G^0$.

(ii) A locally algebraic group $G$ is algebraic if and only if 
$\pi_0(G)$ is finite.

(iii) By \cite[II.5.1.8]{DG}, the category of \'etale group
schemes is equivalent to that of discrete topological groups 
equipped with a continuous action of $\Gamma$ by group automorphisms, 
via the assignment $G \mapsto G(k_s)$. Under this equivalence,
the finite \'etale group schemes correspond to the (abstract) 
finite groups equipped with a $\Gamma$-action by group automorphisms.
\end{remarks}

These results reduce the structure of locally algebraic groups 
to that of algebraic groups; we will concentrate on the latter 
in the sequel.

\subsection{Reduced subschemes}
\label{subsec:rs}

Recall that every scheme $X$ has a largest reduced subscheme $X_\red$; 
moreover, $X_\red$ is closed in $X$ and has the same underlying
topological space. Every morphism of schemes 
$f : X \to Y$ sends $X_\red$ to $Y_\red$. 

\begin{proposition}\label{prop:norm}
Let $G$ be a smooth algebraic group acting on a scheme of finite 
type $X$.

\begin{enumerate}

\item $X_\red$ is stable by $G$.

\item Let $\eta : \tilde{X} \to X_\red$ denote the normalization.
Then there is a unique action of $G$ on $\tilde{X}$
such that $\eta$ is equivariant.

\item When $k$ is separably closed, every irreducible component
of $X_\red$ is stable by $G^0$.

\end{enumerate}

\end{proposition}

\begin{proof}
(1) As $G$ is geometrically reduced, $G \times X_\red$ is reduced
by \cite[IV.6.8.5]{EGA}. Thus, $G \times X_\red = (G \times X)_\red$
is sent to $X_\red$ by $a$.

(2) Likewise, as $G$ is geometrically normal, $G \times \tilde{X}$
is normal by \cite[IV.6.8.5]{EGA} again. So the map
$\id \times \eta : G \times \tilde{X} \to G \times X$ is the 
normalization. This yields a morphism 
$\tilde{a} : G \times \tilde{X} \to \tilde{X}$ such that 
the square
\begin{equation}\label{eqn:norm}
\xymatrix{
G \times \tilde{X} \ar[r]^-{\tilde{a}} \ar[d]_{\id \times \eta} 
& \tilde{X} \ar[d]^{\eta} \\
G \times X_\red \ar[r]^-{a} & X_\red, \\}
\end{equation}
commutes, where $a$ denotes the $G$-action.
Since $\eta$ induces an isomorphism on a dense open subscheme 
of $\tilde{X}$, we have $\tilde{a}(e,\tilde{x})= \tilde{x}$
identically on $\tilde{X}$. Likewise,
$\tilde{a}(g,\tilde{a}(h,\tilde{x})) = \tilde{a}(gh,\tilde{x})$
identically on $G \times G \times \tilde{X}$, i.e., 
$\tilde{a}$ is an action.

(3) Let $Y$ be an irreducible component of $X_\red$. Then the
normalization $\tilde{Y}$ is a connected component of $\tilde{X}$,
and hence is stable by $G^0$ (Remark \ref{rem:neut} (i)). Using 
the surjectivity of the normalization map $\tilde{Y} \to Y$ 
and the commutative square (\ref{eqn:norm}), it follows that
$Y$ is stable by $G^0$.
\end{proof}

When the field $k$ is perfect, the product of any two 
reduced schemes is reduced (see \cite[I.2.4.13]{DG}). 
It follows that the natural map 
$(X \times Y)_\red \to X_\red \times Y_\red$ is an isomorphism;
in particular, the formation of $X_\red$ commutes with field extensions.
This implies easily the following:

\begin{proposition}\label{prop:red}
Let $G$ be a group scheme over a perfect field $k$. 

\begin{enumerate}

\item Any action of $G$ on a scheme $X$ restricts to
an action of the closed subgroup scheme $G_\red$ on $X_\red$.

\item If $G$ is locally algebraic, then $G_\red$ is the 
largest smooth subgroup scheme of $G$.

\end{enumerate}

\end{proposition}

Note that $G_\red$ is not necessarily normal in $G$, as shown
by the following:

\begin{example}\label{ex:nonnorm}
Consider the $\bG_m$-action on $\bG_a$ by multiplication. 
If $\car(k) = p$, then every subgroup scheme
$\alpha_{p^n} = V(x^{p^n}) \subset \bG_a$ is normalized
by this action (since $x^{p^n}$ is homogeneous), but not centralized
(since $\bG_m$ acts non-trivially on 
$\cO(\alpha_{p^n}) = k[x]/(x^{p^n})$). Thus, we may form the 
corresponding semi-direct product 
$G := \bG_m \ltimes \alpha_{p^n}$. Then $G$ is an algebraic group;
moreover, $G_\red = \bG_m$ is not normal in $G$ by Remark 
\ref{rem:sd} (i).

To obtain a similar example with $G$ finite, just replace 
$\bG_m$ with its subgroup scheme $\mu_\ell$ of $\ell$-th
roots of unity, where $\ell$ is prime to $p$.
\end{example}
 
We now obtain a structure result for finite group schemes: 

\begin{proposition}\label{prop:finite}
Let $G$ be a finite group scheme over a perfect field $k$. 
Then the multiplication map induces an isomorphism
$G_\red \ltimes G^0 \stackrel{\cong}{\longrightarrow} G$.
\end{proposition} 

\begin{proof}
Consider, more generally, a finite scheme $X$. 
We claim that the morphism
$\gamma : X \to \pi_0(X)$ restricts to an isomorphism
$X_\red \cong \pi_0(X)$. To check this, we may assume that
$X$ is irreducible; then $X = \Spec(R)$ for some local 
artinian $k$-algebra $R$ with residue field $K$ being a 
finite extension of $k$. Since $k$ is perfect, $K$ lifts 
uniquely to a subfield of $R$, which is clearly 
the largest subfield of that algebra. Then $\gamma_X$ 
is the associated morphism $\Spec(R) \to \Spec(K)$; 
this yields our claim.

Returning to our finite group scheme $G$, we obtain 
an isomorphism of group schemes 
$G_\red \stackrel{\cong}{\to} \pi_0(G)$ via $\gamma$. 
This yields in turn a retraction of $G$ to $G_\red$ 
with kernel $G^0$. So the desired statement follows 
from Remark \ref{rem:sd} (ii). 
\end{proof}
 
With the notation and assumptions of the above proposition,
$G_\red$ is a finite \'etale group scheme, which corresponds
to the finite group $G(\bar{k})$ equipped with the action
of the Galois group $\Gamma$. Also, $G^0$ is finite and
its underlying topological space is just the point $e$; 
such a group scheme is called \textit{infinitesimal}.
Examples of infinitesimal group schemes include $\alpha_{p^n}$ 
and $\mu_{p^n}$ in characteristic $p > 0$. When $\car(k) = 0$,
every infinitesimal group scheme is trivial by Theorem 
\ref{thm:car}.

Proposition \ref{prop:finite} can be extended to the setting
of algebraic groups over perfect fields, see Corollary
\ref{cor:red}. But it fails over any imperfect field,
as shown by the following example of a finite group
scheme $G$ such that $G_\red$ is not a subgroup scheme:

\begin{example}\label{ex:red}
Let $k$ be an imperfect field, i.e., $\car(k) = p > 0$ 
and $k \neq k^p$. Choose $a \in k \setminus k^p$ and 
consider the finite subgroup scheme $G \subset \bG_a$ 
defined as the kernel of the additive polynomial 
$x^{p^2} - a x^p$. Then $G_\red = V(x(x^{p(p-1)} - a))$ 
is smooth at $0$ but not everywhere, since 
$x^{p(p-1)} - a = (x^{p-1} - a^{1/p})^p$ over $k_i$.
So $G_\red$ admits no group scheme structure in view of 
Proposition \ref{prop:smooth}.
\end{example}

\subsection{Torsors}
\label{subsec:t}

\begin{definition}\label{def:tors}
Let $X$ be a scheme equipped with an action $a$ of a group scheme 
$G$, and $f : X \to Y$ a $G$-invariant morphism of schemes. 

We say that $f$ is a \textit{$G$-torsor over $Y$}
(or a \textit{principal $G$-bundle over $Y$})
if it satisfies the following conditions:

\begin{enumerate}

\item $f$ is faithfully flat and quasi-compact.

\item The square
\begin{equation}\label{eqn:tors}
\xymatrix{
G \times X \ar[r]^-a \ar[d]_{p_2} 
& X \ar[d]^{f} \\
X \ar[r]^-f & Y \\}
\end{equation}
is cartesian.

\end{enumerate}

\end{definition}

\begin{remarks}\label{rem:tors}
(i) The condition (2) may be rephrased as follows: for any
scheme $S$ and any points $x,y \in X(S)$, we have $f(x) = f(y)$
if and only if there exists $g \in G(S)$ such that $y = g \cdot x$;
moreover, such a $g$ is unique. This is the scheme-theoretic version 
of the notion of principal bundle in topology.

(ii)  Consider a group scheme $G$ and a scheme $Y$. Let
$G$ act on $G \times Y$ via left multiplication
on itself. Then the projection $p_2: G \times Y \to Y$ is
a $G$-torsor, called the \textit{trivial $G$-torsor over $Y$}. 

(iii) One easily checks that a $G$-torsor $f: X \to Y$ is trivial 
if and only if $f$ has a section. In particular, a $G$-torsor 
$X$ over $\Spec(k)$ is trivial if and only if $X$ has 
a $k$-rational point. When $G$ is algebraic, this holds of course 
if $k$ is algebraically closed, but generally not over 
an arbitrary field $k$. Assume for instance that $k$ 
contains some element $t$ which is not a square, and consider 
the scheme $X := V(x^2 - t) \subset \bA^1$. 
Then $X$ is normalized by the action of $\mu_2$ on $\bA^1$ 
via multiplication; this yields a non-trivial $\mu_2$-torsor 
over $\Spec(k)$. 

(iv) For any $G$-torsor $f : X \to Y$, the topology of $Y$
is the quotient of the topology of $X$ by the equivalence
relation defined by $f$ (see \cite[IV.2.3.12]{EGA}). As a
consequence, the assignment $Z \mapsto f^{-1}(Z)$ yields
a bijection from the open (resp.~closed) subschemes of 
$Y$ to the open (resp.~closed) $G$-stable subschemes of~$X$. 
\end{remarks}

\begin{definition}\label{def:cat}
Let $G$ be a group scheme acting on a scheme $X$. 
A morphism of schemes $f : X \to Y$ is a 
\textit{categorical quotient} of $X$ by $G$, if $f$ is 
$G$-invariant and every $G$-invariant morphism of schemes 
$\varphi : X \to Z$ factors uniquely through $f$.
\end{definition}

In view of its universal property, a categorical 
quotient is unique up to unique isomorphism.

\begin{proposition}\label{prop:cat}
Let $G$ be an algebraic group, and $f : X \to Y$ be a $G$-torsor. 
Then $f$ is a categorical quotient by $G$.
\end{proposition}

\begin{proof}
Consider a $G$-invariant morphism $\varphi : X \to Z$.
Then $\varphi^{-1}(U)$ is an open $G$-stable subscheme
for any open subscheme $U$ of $Z$. Thus,
$f$ restricts to a $G$-torsor $f_U : \varphi^{-1}(U) \to V$
for some open subscheme $V = V(U)$ of $Y$.
To show that $\varphi$ factors uniquely through $f$, 
it suffices to show that $\varphi_U : \varphi^{-1}(U) \to U$
factors uniquely through $f_U$ for any affine $U$. Thus,
we may assume that $Z$ is affine. Then $\varphi$ corresponds
to a $G$-invariant homomorphism $\cO(Z) \to \cO(X)$, i.e.,
to a homomorphism $\cO(Z) \to \cO(X)^G$ (the subalgebra of
$G$-invariants in $\cO(X)$). So it suffices to check that
the map 
\[ f^\# : \cO_Y \longrightarrow f_*(\cO_X)^G \] 
is an isomorphism.

Since $f$ is faithfully flat, it suffices in turn to show
that the natural map 
\[ \cO_X = f^*(\cO_Y) \to f^*(f_*(\cO_X)^G) \]
is an isomorphism. We have canonical isomorphisms 
\[ f^*(f_*(\cO_X)) \cong p_{2*}(a^*(\cO_X)) 
\cong p_{2*}(\cO_{G \times X}) \cong \cO(G) \otimes_k \cO_X,\]
where the first isomorphism follows from the cartesian
square (\ref{eqn:tors}) and the faithful flatness of $f$, 
and the third isomorphism follows from Lemma \ref{lem:aff}. 
Moreover, the composition of these isomorphisms identifies 
the $G$-action on $f^*(f_*(\cO_X))$ with that on 
$\cO(G) \otimes_k \cO_X$ via left multiplication on $\cO(G)$. 
Thus, taking $G$-invariants yields the desired isomorphism.
\end{proof}

\begin{proposition}\label{prop:des}
Let $f: X \to Y$ be a $G$-torsor. 

\begin{enumerate}

\item The morphism $f$ is finite (resp. affine, proper, 
of finite presentation) if and only if so is the scheme $G$. 

\item When $Y$ is of finite type, $f$ is smooth if and only 
if $G$ is smooth.

\end{enumerate}

\end{proposition}

\begin{proof}
(1) This follows from the cartesian diagram (\ref{eqn:tors}) 
together with descent theory (see \cite[IV.2.7.1]{EGA}).

Likewise, (2) follows from \cite[IV.6.8.3]{EGA}. 
\end{proof}
  
\begin{remarks}\label{rem:des}
(i) As a consequence of the above proposition, 
every torsor $f : X \to Y$ under an algebraic group $G$ 
is of finite presentation. In particular,
$f$ is surjective on $\bar{k}$-rational points, i.e., 
the induced map $X(\bar{k}) \to Y(\bar{k})$ is surjective.
But $f$ is generally not surjective on $S$-points for an
arbitrary scheme $S$ (already for $S = \Spec(k)$). Still,
$f$ satisfies the following weaker version of surjectivity:

\textit{For any scheme $S$ and any point $y \in Y(S)$, 
there exists a faithfully flat morphism of finite presentation 
$\varphi : S' \to S$ and a point $x \in X(S')$ such that 
$f(x) = y$.}

Indeed, viewing $y$ as a morphism $S \to Y$, we may take 
$S' :=  X \times_Y S$, $\varphi := p_2$ and $x := p_1$. 

(ii) Consider a $G$-scheme $X$, a $G$-invariant morphism of 
schemes $f : X \to Y$ and a faithfully flat quasi-compact morphism
of schemes $v : Y' \to Y$. Form the cartesian square
\[ 
\xymatrix{
X' \ar[r]^-{f'} \ar[d]_u 
& Y' \ar[d]^v \\
X \ar[r]^-f & Y. \\}
\]
Then there is a unique action of $G$ on $X'$ such that $u$
is equivariant and $f'$ is invariant. Moreover, $f$ is a
$G$-torsor if and only if $f'$ is a $G$-torsor. Indeed,
this follows again from descent theory, more specifically
from \cite[IV.2.6.4]{EGA} for the condition (2), and 
\cite[IV.2.7.1]{EGA} for (3).

(iii) In the above setting, $f$ is a $G$-torsor if and only 
if the base change $f_K$ is a $G_K$-torsor for some field 
extension $K$ of $k$. 

(iv) Consider two $G$-torsors $f : X \to Y$, $f' : X' \to Y$
and a $G$-equivariant morphism $\varphi : X \to X'$ of schemes
over $Y$. Then $\varphi$ is an isomorphism: to check this, 
one may reduce by descent to the case where $f$ and $f'$ are 
trivial. Then $\varphi$ is identified with an endomorphism 
of the trivial torsor. But every such endomorphism is of the 
form $(g,y) \mapsto (g \psi(y),y)$ for a unique morphism 
$\psi: Y \to G$, and hence is an automorphism with inverse 
$(g,y) \mapsto (g \psi(y)^{-1},y)$.
\end{remarks}

\begin{example}\label{ex:neut}
Let $G$ be an algebraic group. Then 
$\gamma : G \to \pi_0(G)$ is a $G^0$-torsor.

Indeed, recall from \S \ref{subsec:tnc} that 
the formation of $\gamma$ commutes with field extensions. 
By Remark \ref{rem:des} (iii), we may thus assume 
$k$ algebraically closed. Then the finite \'etale 
scheme $\pi_0(G)$ just consists of finitely many 
$k$-rational points, say $x_1,\ldots,x_n$, and the fiber 
$F_i$ of $\gamma$ at $x_i$ contains a $k$-rational point, 
say $g_i$. Recall that $F_i$ is a connected component of $G$; 
thus, the translate $g_i^{-1} F_i$ is a connected component 
of $G$ through $e$, and hence equals $G^0$. It follows that 
$G$ is the disjoint union of the translates $g_i G^0$, which 
are the fibers of $\gamma$; this yields our assertion.
\end{example}

\subsection{Homogeneous spaces and quotients}
\label{subsec:hsq}

\begin{proposition}\label{prop:hom}
Let $f : G \to H$ be a homomorphism of algebraic groups.

\begin{enumerate}

\item The image $f(G)$ is closed in $H$.

\item $f$ is a closed immersion if and only if its kernel
is trivial.

\end{enumerate}

\end{proposition}

\begin{proof}
As in the proof of Proposition \ref{prop:orb},
we may assume that $k$ is algebraically closed. 

(1) Consider the action $a$ of $G$ on $H$ given by
$g \cdot h := f(g) h$. By Proposition \ref{prop:orb} again, 
there exists $h \in H(k)$ such that the image of the orbit 
map $a_h$ is closed. But $a_h(G) = a_e(G) h$ and hence 
$a_e(G) = f(G)$ is closed.

(2) Clearly, $\Ker(f)$ is trivial if $f$ is a closed immersion.
Conversely, if $\Ker(f)$ is trivial then the fiber of $f$ at
any point $x \in X$ consists of that point; in particular,
$f$ is quasi-finite. By Zariski's Main Theorem (see 
\cite[IV.8.12.6]{EGA}), $f$ factors as an immersion followed 
by a finite morphism. As a consequence, there exists 
a dense open subscheme $U$ of $f(G)$ such that the restriction 
$f^{-1}(U) \to U$ is finite. Since the translates of $U_{\bar{k}}$ 
by $G(\bar{k})$ cover $f(G_{\bar{k}})$, it follows that $f_{\bar{k}}$ 
is finite; hence $f$ is finite as well. 
Choose an open affine subscheme $V$ of $f(G)$; then so is 
$f^{-1}(V)$, and $\cO(f^{-1}(V))$ is a finite module over 
$\cO(V)$ via $f^\#$. Moreover, the natural map
\[ \cO(V)/\fm \longrightarrow 
\cO(f^{-1}(V)/\fm \cO(f^{-1}(V)) = \cO(f^{-1}(\Spec \cO(V)/\fm)) \] 
is an isomorphism for any maximal ideal $\fm$ of $\cO(V)$.
By Nakayama's lemma, it follows that $f^\#$ is surjective;
this yields the assertion. 
\end{proof}

As a consequence of the above proposition, 
\textit{every subgroup scheme of an algebraic group is closed}.

We now come to an important existence result:

\begin{theorem}\label{thm:hom}
Let $G$ be an algebraic group and $H \subseteq G$ a subgroup scheme. 

\begin{enumerate}

\item There exists a $G$-scheme $G/H$ equipped with a 
$G$-equivariant morphism 
\[ q : G \longrightarrow G/H, \] 
which is an $H$-torsor for the action of $H$ on $G$ by
right multiplication. 

\item The scheme $G/H$ is of finite type. It is smooth 
if $G$ is smooth.

\item If $H$ is normal in $G$, then $G/H$ has a unique structure 
of algebraic group such that $q$ is a homomorphism.

\end{enumerate}

\end{theorem}

\begin{proof}
See \cite[VIA.3.2]{SGA3}.
\end{proof}

\begin{remarks}\label{rem:quot}
(i) With the notation and assumptions of the above theorem, 
$q$ is the categorical quotient of $G$ by $H$, in view of
Proposition \ref{prop:cat}. In particular, $q$ is unique
up to unique isomorphism; it is called the 
\textit{quotient morphism}. The \textit{homogeneous space}
$G/H$ is equipped with a $k$-rational point $x := q(e)$, the
\textit{base point}. The stabilizer $\C_G(x)$ equals $H$,
since it is the fiber of $q$ at $x$.

(ii) By the universal property of categorical quotients,
the homomorphism of algebras $q^\# : \cO(G/H)  \to \cO(G)^H$ 
is an isomorphism.

(iii) The morphism $q$ is faithfully flat and lies
in a cartesian diagram
\[
\xymatrix{
G \times H \ar[r]^-n \ar[d]_{p_1} 
& G \ar[d]^-q \\
G \ar[r]^-q & G/H, \\}
\]
where $n$ denotes the restriction of the multiplication
$m : G \times G \to G$. Also, $q$ is of finite presentation 
in view of Proposition \ref{prop:des}. 

(iv) Since $q$ is flat and $G$, $H$ are equidimensional,
we see that $G/H$ is equidimensional of dimension 
$\dim(G) - \dim(H)$.

(v) We have $(G/H)(\bar{k}) = G(\bar{k})/H(\bar{k})$
as follows e.g. from Remark \ref{rem:des} (i).
In particular, if $k$ is perfect (so that $G_\red$ is 
a subgroup scheme of $G$), then the scheme $G/G_\red$ 
has a unique $\bar{k}$-rational point. Since that scheme
is of finite type, it is finite and local; its base point
is its unique $k$-rational point.
\end{remarks}

Next, we obtain two further factorization properties 
of quotient morphisms:

\begin{proposition}\label{prop:fact}
Let $f : G \to H$ be a homomorphism of algebraic groups, 
$N := \Ker(f)$ and $q : G \to G/N$ the quotient homomorphism.
Then there is a unique homomorphism $\iota : G/N \to H$ such
that the triangle 
\[
\xymatrix{
G \ar[r]^-{f} \ar[d]_{q} & H \\ 
G/N \ar[ur]_{\iota}  \\}
\]
commutes. Moreover, $\iota$ is an isomorphism onto a subgroup 
scheme of $H$.
\end{proposition}

\begin{proof}
Clearly, $f$ is $N$-invariant; thus, it factors through
a unique morphism $\iota : G/N \to H$ by Theorem \ref{thm:hom}.
We check that $\iota$ is a homomorphism: let $S$ be a scheme
and $x,y \in (G/N)(S)$. By Remark \ref{rem:des}, 
there exist morphisms of schemes $\varphi : T \to S$, 
$\psi : U \to S$ and points $x_T \in G(T)$, $y_U \in G(U)$ 
such that $q(x_T) = x$, $q(y_U) = y$. Using the fibered product
$S' := T \times_S U$, we thus obtain a morphism 
$f : S' \to S$ and points $x',y' \in G(S')$
such that $q(x') = x$, $q(y') = y$; then $q(x' y') = x y$. 
Since $f(x' y') = f(x') f(y')$, we have 
$\iota(x y) = \iota(x) \iota(y)$. 
One may check likewise that $\Ker(\iota)$ is trivial. Thus,
$\iota$ is a closed immersion; hence its image is a subgroup scheme
in view of Proposition \ref{prop:hom}.
\end{proof}

\begin{proposition}\label{prop:orbit}
Let $G$ be an algebraic group, $X$ a $G$-scheme of finite
type and $x \in X(k)$. Then the orbit map $a_x : G \to X$,
$g \mapsto g \cdot x$ factors through a unique immersion 
$j_x : G/\C_G(x) \to X$.
\end{proposition}

\begin{proof}
See \cite[III.3.5.2]{DG} or \cite[V.10.1.2]{SGA3}.
\end{proof}

With the above notation and assumptions, we may define
the \textit{orbit} of $x$ as the locally closed subscheme 
of $X$ corresponding to the immersion $j_x$.

\subsection{Exact sequences, isomorphism theorems}
\label{subsec:esit}

\begin{definition}\label{def:ex}
Let $j : N \to G$ and $q: G \to Q$ be homomorphisms
of group schemes. We have an \textit{exact sequence}
\begin{equation}\label{eqn:ex} 
1 \longrightarrow N \stackrel{j}{\longrightarrow} G
\stackrel{q}{\longrightarrow} Q \longrightarrow 1 
\end{equation}
if the following conditions hold:

\begin{enumerate}

\item $j$ induces an isomorphism of $N$ with $\Ker(q)$.

\item For any scheme $S$ and any $y \in Q(S)$, there exists
a faithfully flat morphism $f: S' \to S$ of finite presentation
and $x \in G(S')$ such that $q(x) = y$.

\end{enumerate}

Then $G$ is called an \textit{extension} of $Q$ by $N$.

We say that $q$ is an \textit{isogeny} if $N$ is finite.
\end{definition}

\begin{remarks}\label{rem:ex}
(i) The condition (1) holds if and only if the sequence of groups
\[ 1 \longrightarrow N(S) \stackrel{j(S)}{\longrightarrow} G(S)
\stackrel{q(S)}{\longrightarrow} Q(S) \]
is exact for any scheme $S$.

(ii) The condition (2) holds whenever $q$ is faithfully flat of 
finite presentation, as already noted in Remark \ref{rem:des}(i).

(iii) As for exact sequences of abstract groups, one may define
the push-forward of the exact sequence (\ref{eqn:ex}) under 
any homomorphism $N \to N'$, and the pull-back under 
any homomorphism $Q' \to Q$. Also, exactness is preserved under 
field extensions.
\end{remarks} 

Next, consider an algebraic group $G$ and a normal subgroup
scheme $N$; then we have an exact sequence 
\begin{equation}\label{eqn:quot}
1 \longrightarrow N \longrightarrow G 
\stackrel{q}\longrightarrow G/N \longrightarrow 1 
\end{equation}
by Theorem \ref{thm:hom} and the above remarks. Conversely,
given an exact sequence (\ref{eqn:ex}) of algebraic groups,
$j$ is a closed immersion and $q$ factors through a 
closed immersion $\iota : G/N \to Q$ by Proposition 
\ref{prop:fact}. Since $q$ is surjective, $\iota$ is 
an isomorphism; this identifies the exact sequences 
(\ref{eqn:ex}) and (\ref{eqn:quot}). 

As another consequence of Proposition \ref{prop:fact},
\textit{the category of commutative algebraic groups is abelian}.
Moreover, the above notion of exact sequence coincides 
with the categorical notion. In this setting, the set 
of isomorphism classes of extensions of $Q$ by $N$
has a natural structure of commutative group, that we 
denote by $\Ext^1(Q,N)$.

We now extend some classical isomorphism theorems for
abstract groups to the setting of group schemes, in a series
of propositions: 

\begin{proposition}\label{prop:sub}
Let $G$ be an algebraic group and $N \trianglelefteq G$ 
a normal subgroup scheme with quotient $q : G \to G/N$. 
Then the assignment $H \mapsto H/N$ yields a bijective 
correspondence between the subgroup schemes of $G$ 
containing $N$ and the subgroup schemes of $G/N$,
with inverse the pull-back. Under this correspondence, 
the normal subgroup schemes of $G$ containing $N$ 
correspond to the normal subgroup schemes of $G/N$.
\end{proposition}

\begin{proof}
See \cite[VIA.5.3.1]{SGA3}.
\end{proof}

\begin{proposition}\label{prop:quot}
Let $G$ be an algebraic group and $N \subseteq H \subseteq G$ 
subgroup schemes with quotient maps $q_N : G \to G/N$, 
$q_H : G \to G/H$.

\begin{enumerate}

\item There exists a unique morphism $f : G/N \to G/H$ 
such that the triangle
\[
\xymatrix{
G \ar[r]^-{q_H} \ar[d]_{q_N} & G/H \\ 
G/N \ar[ur]_f  \\}
\]
commutes. Moreover, $f$ is $G$-equivariant and faithfully flat
of finite presentation. The fiber of $f$ at the base point of $G/H$ 
is the homogeneous space $H/N$.

\item If $N$ is normal in $H$, then the action of $H$ on $G$ 
by right multiplication factors through an action of $H/N$ 
on $G/N$ that centralizes the action of $G$. Moreover,
$f$ is an $H/N$-torsor.

\item If $H$ and $N$ are normal in $G$, then we have 
an exact sequence
\[ 1 \longrightarrow H/N \longrightarrow G/N 
\stackrel{f}{\longrightarrow} G/H \longrightarrow 1.\]
\end{enumerate}

\end{proposition}

\begin{proof}
(1) The existence of $f$ follows from the fact that $q_N$ is
a categorical quotient. To show that $f$ is equivariant,
let $S$ be a scheme, $g \in G(S)$ and $y \in (G/N)(S)$.
Then there exists a morphism $S' \to S$ and $y' \in G(S')$
such that $q_N(y') = y$. So 
\[ f(g \cdot y) = f(g \cdot q_N (y') = (f \circ q_N)(gy') = q_H(gy')
= g \cdot q_H(y') = g \cdot y. \]
One checks similarly that the fiber of $f$ at the base point
$x$ equals $H/N$.

Next, note that the multiplication map $n : G \times H \to H$
yields a morphism $r : G \times H/N \to G/N$. We claim that
the square
\begin{equation}\label{eqn:hom} 
\xymatrix{
G \times H/N \ar[r]^-{r} \ar[d]_{p_1} 
& G/N \ar[d]^f \\
G \ar[r]^-{q_H} & G/H \\}
\end{equation}
is cartesian. The commutativity of this square follows readily 
from the equivariance of the involved morphisms. Let $S$
be a scheme and $g \in G(S)$, $y \in (G/N)(S)$. Then 
$q_H(g) = f(y)$ if and only if 
$f(g^{-1} \cdot y) = q_H(e) = f(x)$, i.e., 
$g^{-1} y \in (H/N)(S)$. It follows that the map
$G \times H/N \to G \times_{G/H}G/N$ is bijective
on $S$-points; this yields the claim.

Since $q_H$ and $p_1$ are faithfully flat of finite
presentation, the same holds for $f$ in view of the 
cartesian square (\ref{eqn:hom}).

(2) The existence of the action $G/N \times H/N \to G/N$
follows similarly from the universal property of the
quotient $G \times H \to G/N \times H/N$. One may check
by lifting points as in the proof of (1) that this action 
centralizes the $G$-action. Finally, $f$ is a $G$-torsor
in view of the cartesian square (\ref{eqn:hom}) again.

(3) This follows readily from (1) together with Proposition
\ref{prop:fact} (or argue by lifting points to check 
that $f$ is a homomorphism).
\end{proof}

\begin{proposition}\label{prop:prod}
Let $G$ be an algebraic group, $H \subseteq G$ a subgroup
scheme and $N \trianglelefteq G$ a normal subgroup scheme. 
Consider the semi-direct product $H \ltimes N$, where 
$H$ acts on $N$ by conjugation. 

\begin{enumerate}

\item The map
\[ f : H \ltimes N \longrightarrow G, \quad (x,y) \longmapsto xy \]
is a homomorphism with kernel $H \cap N$ identified with a subgroup
scheme of $H \ltimes N$ via $x \mapsto (x^{-1},x)$. 

\item The image $H \cdot N$ of $f$ is the smallest subgroup scheme
of $G$ containing $H$ and $N$.

\item The natural maps $H/H \cap N \to H \cdot N/N$ 
and $N/H \cap N \to H \cdot N/H$ are isomorphisms.

\item If $H$ is normal in $G$, then $H \cdot N$ is normal in $G$ as well.

\end{enumerate}

\end{proposition}

\begin{proof}
The assertions (1) and (2) are easily checked.

(3) We have a commutative diagram
\[
\xymatrix{
H \ar[r] \ar[d] & H \ltimes N/N \ar[d] \\
H/H\cap N \ar[r] & H \cdot N/N, \\}
\] 
where the top horizontal arrow is an isomorphism and the vertical 
arrows are $H \cap N$-torsors. This yields the first isomorphism
by using Proposition \ref{prop:cat}. The second isomorphism is
obtained similarly.

(4) This may be checked as in the proof of Proposition \ref{prop:fact}.
\end{proof}

We also record a useful observation:

\begin{lemma}\label{lem:prod}
Keep the notation and assumptions of the above proposition.
If $G = H \cdot N$, then $G(\bar{k}) = H(\bar{k}) \, N(\bar{k})$.
The converse holds when $G/N$ is smooth.
\end{lemma}

\begin{proof}
The first assertion follows e.g. from Remark \ref{rem:des} (i).

For the converse, consider the quotient homomorphism 
$q : G \to G/N$: it restricts to a homomorphism $H \to G/N$ 
with kernel $H \cap N$, and hence factors through a closed 
immersion $i : H/H \cap N \to G/N$ by Proposition 
\ref{prop:fact}. Since $G(\bar{k}) = H(\bar{k}) N(\bar{k})$, 
we see that $i$ is surjective on $\bar{k}$-rational points. As
$G/N$ is smooth, $i$ must be an isomorphism.
Thus, $H \cdot N/N = G/N$. By Proposition \ref{prop:sub}, 
we conclude that $H \cdot N = G$.
\end{proof}

We may now obtain the promised generalization of the structure
of finite group schemes over a perfect field (Proposition
\ref{prop:finite}): 

\begin{corollary}\label{cor:red}
Let $G$ be an algebraic group over a perfect field $k$.

\begin{enumerate}

\item $G = G_\red \cdot G^0$.

\item $G_\red \cap G^0 = G^0_\red$ is the smallest subgroup
scheme $H$ of $G$ such that $G/H$ is finite.

\end{enumerate}

\end{corollary}

\begin{proof}
(1) This follows from Lemma \ref{lem:prod}, since $G/G^0 \cong \pi_0(G)$ 
is smooth and $G(\bar{k}) = G_\red(\bar{k})$.

(2) Let $H \subseteq G$ be a subgroup scheme. Since $G/H$ 
is of finite type, the finiteness of $G/H$ is equivalent to
the finiteness of $(G/H)(\bar{k}) = G(\bar{k})/H(\bar{k}) 
= G(\bar{k})/H_\red(\bar{k})$. Thus, $G/H_\red$ is finite
if and only if so is $G/H$. Likewise, using the finiteness 
of $H/H^0$, one may check that $G/H$ is finite if and only
if so is $G/H^0_\red$. Under these conditions, the homogeneous
space $G^0_\red/H^0_\red$ is finite as well; since it is also
smooth and connected, it follows that $G^0_\red = H^0_\red$,
i.e., $G^0_\red \subseteq H$. 

To complete the proof, it suffices to check that $G/G^0_\red$ is 
finite, or equivalently that $G(\bar{k})/G^0(\bar{k})$ is finite.
But this follows from the finiteness of $G/G^0$.
\end{proof}

\begin{definition}\label{def:split}
An exact sequence of group schemes (\ref{eqn:ex}) is called
\textit{split} if $q : G \to Q$ has a section which is
a homomorphism.
\end{definition}

Any such section $s$ yields an endomorphism $r := s \circ q$
of the group scheme $G$ with kernel $N$; moreover, $r$
may be viewed as a retraction of $G$ to the image of $s$, 
isomorphic to $H$. By Remark \ref{rem:sd} (ii), this identifies 
(\ref{eqn:ex}) with the exact sequence
\[ 1 \longrightarrow N \stackrel{i}{\longrightarrow} 
H \ltimes N \stackrel{r}{\longrightarrow} H \longrightarrow 1. \]

\subsection{The relative Frobenius morphism}
\label{subsec:trfm}

Throughout this subsection, we assume that the ground field $k$ 
has characteristic $p > 0$.

Let $X$ be a $k$-scheme and $n$ a positive integer. The
\textit{$n$th absolute Frobenius morphism of $X$} is the
endomorphism 
\[ F^n_X : X \longrightarrow X \] 
which is the identity on the underlying topological space 
and such that the homomorphism of sheaves of algebras
$(F^n_X )^\# : \cO_X \to (F^n_X )_*(\cO_X) = \cO_X$
is the $p^n$th power map, $f \mapsto f^{p^n}$.

Clearly, every morphism of $k$-schemes $f : X \to Y$ lies
in a commutative square
\[
\xymatrix{
X \ar[r]^-{f} \ar[d]_{F^n_X} & Y \ar[d]^{F^n_Y} \\ 
X \ar[r]^-{f} & Y.  \\}
\]
Note that $F^n_X$ is generally not a morphism of $k$-schemes,
since the $p^n$th power map is generally not $k$-linear.
To address this, define a $k$-scheme $X^{(n)}$ by the cartesian square
\[
\xymatrix{
X^{(n)} \ar[r] \ar[d] & X \ar[d]^{\pi} \\ 
\Spec(k) \ar[r]^-{F^n_k} & \Spec(k),  \\}
\]
where $\pi$ denotes the structure map and $F^n_k := F^n_{\Spec(k)}$ 
corresponds to the $p^n$th power map of $k$. Then 
$F^n_X$ factors through a unique morphism of $k$-schemes
\[ F^n_{X/k} : X \longrightarrow X^{(n)}, \]
the \textit{$n$th relative Frobenius morphism}. Equivalently, 
the above cartesian square extends to a commutative diagram
\[
\xymatrix{X \ar[d]_{F^n_{X/k}} \ar[dr]^{F^n_X}\\
X^{(n)} \ar[r] \ar[d] & X \ar[d]^{\pi} \\ 
\Spec(k) \ar[r]^-{F^n_k} & \Spec(k).  \\}
\]

The underlying topological space of $X^{(n)}$ is $X$
again, and the structure sheaf is given by
\[ \cO_{X^{(n)}}(U) = \cO_X(U) \otimes_{F^n} k \]
for any open subset $U \subseteq X$, where the right-hand side 
denotes the tensor product of $\cO_X(U)$ and $k$ over $k$
acting on $\cO_X(U)$ via scalar multiplication, and on $k$ 
via the $p^n$th power map. Thus, we have in 
$\cO_X(U) \otimes_{F^n} k$
\[ t f \otimes u = f \otimes t^{p^n} u \]
for any $f \in \cO_X(U)$ and $t,u \in k$. The $k$-algebra
structure on $\cO_X(U) \otimes_{F^n} k$ is defined by 
\[ t (f \otimes u) = f \otimes tu \]
for any such $f$, $t$ and $u$. The morphism
$F^n_{X/k}$ is again the identity on the underlying topological
spaces; the associated homomorphism of sheaves of algebras
is the map 
\begin{equation}\label{eqn:trfm} 
(F^n_{X/k})^\# : \cO_X(U) \otimes_{F^n} k \longrightarrow \cO_X(U), 
\quad f \otimes t \longmapsto t f^{p^n}. 
\end{equation} 

Using this description, one readily checks that 
\textit{the formation of the $n$th relative Frobenius morphism 
commutes with field extensions}. Moreover, for any positive integers 
$m$, $n$, we have an isomorphism of schemes 
\[ (X^{(m)})^{(n)} \cong X^{(m+n)} \]
that identifies the composition $F^n_{X^{(m)}/k} \circ F^m_{X/k}$
with $F^{m+n}_{X/k}$. In particular, $F^n_{X/k}$ may be seen as
the $n$th iterate of the relative Frobenius morphism $F_{X/k}$.

Also, note that \textit{the formation of $F^n_{X/k}$ is compatible
with closed subschemes and commutes with finite products}. 
Specifically, any morphism of $k$-schemes $f: X \to Y$ induces 
a morphism of $k$-schemes $f^{(n)} : X^{(n)} \to Y^{(n)}$ such that 
the square 
\[
\xymatrix{
X \ar[r]^-f \ar[d]_{F^n_{X/k}} & Y \ar[d]^{F^n_{Y/k}} \\ 
X^{(n)} \ar[r]^-{f^{(n)}} & Y^{(n)}  \\}
\]
commutes. If $f$ is a closed immersion, then so is $f^{(n)}$.
Moreover, for any two schemes $X$, $Y$, the map
\[ p_1^{(n)} \times p_2^{(n)} : 
(X \times Y)^{(n)} \longrightarrow
X^{(n)} \times Y^{(n)} \]
is an isomorphism (where $p_1: X \times Y \to X$,
$p_2: X \times Y \to Y$ denote the projections), and the triangle
\[
\xymatrix{
X \times Y \ar[r]^-{F^n_{X \times Y/k}} \ar[dr]_{F^n_{X/k} \times F^n_{Y/k}} 
& (X \times Y)^{(n)} \ar[d]^{\cong} \\
& X^{(n)} \times Y^{(n)} \\
}
\]
commutes. 

We now record some geometric properties of the relative Frobenius 
morphism:

\begin{lemma}\label{lem:frob}
Let $X$ be a scheme of finite type and $n$ a positive integer.

\begin{enumerate}

\item The morphism $F^n_{X/k}$ is finite and purely inseparable. 

\item The scheme-theoretic image of $F^n_{X/k}$ is geometrically 
reduced for $n \gg 0$.

\end{enumerate}

\end{lemma}

\begin{proof}
(1) Since $F^n_{X/k}$ is the identity on the underlying topological
spaces, we may assume that $X$ is affine. Let $R := \cO(X)$, then 
the image of the homomorphism 
$(F^n_{X/k})^\# : R \otimes_{F^n} k \to R$ 
is the $k$-subalgebra $k R^{p^n}$ generated by the $p^n$th powers.
Thus, $F^n_{X/k}$ is integral, and hence finite since $R$ is of finite 
type. Also, $F^n_{X/k}$ is clearly purely inseparable.

(2) Let $I \subset R$ denote the ideal consisting of nilpotent elements. 
Since the algebra $R$ is of finite type, there exists a positive
integer $n_0$ such that $f^n = 0$ for all $f \in I$ and all
$n \geq n_0$. Choose $n_1$ such that $p^{n_1} \geq n_0$, then 
$(F^n_{X/k})^\#$ sends $I$ to $0$ for any $n \geq n_1$. Thus,
the image of $F^n_{X/k}$ is reduced for $n \gg 0$. Since 
the formation of $F^n_{X/k}$ commutes with field extensions, 
this completes the proof.
\end{proof}

\begin{proposition}\label{prop:frob}
Let $G$ be a $k$-group scheme.

\begin{enumerate}

\item There is a unique structure of $k$-group scheme on $G^{(n)}$
such that $F^n_{G/k}$ is a homomorphism.

\item If $G$ is algebraic, then $\Ker(F^n_{G/k})$ is infinitesimal.
Moreover, $G/\Ker(F^n_{G/k})$ is smooth for $n \gg 0$.

\end{enumerate}

\end{proposition}

\begin{proof}
(1) This follows from the fact that the formation of the relative Frobenius 
morphism commutes with finite products.

(2) This is a consequence of the above lemma together with
Proposition \ref{prop:smooth}.
\end{proof}

\medskip

\noindent
\textit{Notes and references}.

Most of the notions and results presented in this section 
can be found in \cite{DG} and \cite{SGA3} in a much greater
generality. We provide some specific references:

Proposition \ref{prop:smooth} is taken from 
\cite[VIA.1.3.1]{SGA3}; Proposition \ref{prop:orb} follows 
from results in \cite[II.5.3]{DG}; Lemma \ref{lem:aff} 
is a special case of \cite[I.2.2.6]{DG}; Theorem \ref{thm:neut}
follows from \cite[II.5.1.1, II.5.1.8]{DG}; Proposition 
\ref{prop:finite} holds more generally for locally algebraic 
groups, see \cite[II.2.2.4]{DG}; Example \ref{ex:red} 
is in \cite[VIA.1.3.2]{SGA3}.

Our definition of torsors is somewhat ad hoc: what we call 
$G$-torsors over $Y$ should be called $G_Y$-torsors, 
where $G_Y$ denotes the group scheme $p_2 : G \times Y \to Y$
(see \cite[III.4.1]{DG} for general notions and results on torsors).

Proposition \ref{prop:cat} is a special case of a result of
Mumford, see \cite[Prop.~0.1]{MFK}; Proposition \ref{prop:hom} 
is a consequence of \cite[II.5.5.1]{DG}; Proposition 
\ref{prop:fact} is a special case of \cite[VIA.5.4.1]{SGA3}.

Theorem \ref{thm:hom} (on the existence of homogeneous 
spaces) is a deep result, since no direct construction of these
spaces is known in this generality. In the setting of affine 
algebraic groups, homogeneous spaces may be constructed by a method
of Chevalley; this is developed in \cite[III.3.5]{DG}.

Propositions \ref{prop:quot} and \ref{prop:prod} are closely 
related to results in \cite[VIA.5.3]{SGA3}. We have provided 
additional details to be used later.

Proposition \ref{prop:frob} (2) holds more generally 
for locally algebraic groups, see \cite[VII.8.3]{SGA3}.

Many interesting extensions of algebraic groups are not split, 
but quite a few of them turn out to be \textit{quasi-split}, 
i.e., split after pull-back by some isogeny. For example, the
extension 
\[ 1 \longrightarrow G^0 \longrightarrow G \longrightarrow 
\pi_0(G) \longrightarrow 1 \] 
is quasi-split for any algebraic group $G$ (see \cite[Lem.~5.11]{BS} 
when $G$ is smooth and $k$ is algebraically closed of characteristic
$0$; the general case follows from \cite[Thm.~1.1]{Br4}). 
Further instances of quasi-split extensions will be obtained in
Theorems \ref{thm:qc}, \ref{thm:cag} and \ref{thm:cagc} below. 
On the other hand, the group $G$ of upper triangular unipotent 
$3 \times 3$ matrices lies in an extension
\[ 1 \longrightarrow \bG_a \longrightarrow G \longrightarrow 
\bG_a^2 \longrightarrow 1, \]
which is not quasi-split. It would be interesting to determine
those classes of algebraic groups that yield quasi-split extensions.

\section{Proof of Theorem \ref{thm:affi}}
\label{sec:ptaffi}

\subsection{Affine algebraic groups}
\label{subsec:aag}

In this subsection, we obtain several criteria for an algebraic 
group to be affine, which will be used throughout the sequel. 
We begin with a classical result:

\begin{proposition}\label{prop:lin}
Every affine algebraic group is linear.
\end{proposition}

\begin{proof}
Let $G$ be an affine algebraic group. By Proposition
\ref{prop:emb}, there exist a finite-dimensional $G$-module
$V$ and a closed $G$-equivariant immersion $\iota : G \to V$,
where $G$ acts on itself by left multiplication. Since the latter
action is faithful, the $G$-action on $V$ is faithful as well. 
In other words, the corresponding homomorphism $\rho: G \to \GL(V)$ 
has a trivial kernel. By Proposition \ref{prop:hom}, it follows 
that $\rho$ is a closed immersion.
\end{proof}

Next, we relate the affineness of algebraic groups with
that of subgroup schemes and quotients:

\begin{proposition}\label{prop:tower}
Let $H$ be a subgroup scheme of an algebraic group $G$.

\begin{enumerate}

\item If $H$ and $G/H$ are both affine, then $G$ is affine as well.

\item If $G$ is affine, then $H$ is affine. If in addition
$H \trianglelefteq G$, then $G/H$ is affine.

\end{enumerate}

\end{proposition}

\begin{proof}
(1) Since $H$ is affine, the quotient morphism $q : G \to G/H$ 
is affine as well, in view of Proposition \ref{prop:des}
and Theorem \ref{thm:hom} (3). This yields the statement.

(2) The first assertion follows from the closedness of $H$ in $G$ 
(Proposition \ref{prop:hom}). The second assertion is proved 
in \cite[III.3.7.3]{DG}, see also \cite[VIB.11.7]{SGA3}.
\end{proof}

\begin{remark}\label{rem:tower}
With the notation and assumptions of the above proposition,
$G$ is smooth (resp.~proper, finite) if $H$ and $G/H$ are both
smooth (resp.~proper, finite), as follows from the same argument. 
Also, $G$ is connected if $H$ and $G/H$ are both connected;
since all these schemes have $k$-rational points, this is
equivalent to geometric connectedness.
\end{remark}

The above proposition yields that every algebraic
group has an ``affine radical'':

\begin{lemma}\label{lem:ar}
Let $G$ be an algebraic group.

\begin{enumerate}

\item $G$ has a largest smooth connected normal affine 
subgroup scheme, $L(G)$.

\item $L(G/L(G))$ is trivial.

\item The formation of $L(G)$ commutes with separable 
algebraic field extensions.

\end{enumerate}

\end{lemma}

\begin{proof}
(1) Let $L_1$, $L_2$ be two smooth connected normal affine 
subgroup schemes of $G$. Then the product 
$L_1 \cdot L_2 \subseteq G$ is a normal subgroup scheme
by Proposition \ref{prop:prod}. Since $L_1 \cdot L_2$
is a quotient of $L_1 \ltimes L_2$, it is smooth and
connected. Also, by using the isomorphism  
$L_1 \cdot L_2/L_1 \cong L_2/L_1 \cap L_2$ together
with Proposition \ref{prop:tower}, we see that 
$L_1 \cdot L_2$ is affine. 

Next, take $L_1$ as above and of maximal dimension.
Then $\dim(L_1 \cdot L_2/L_1) = 0$ by Proposition
\ref{prop:fact}. Since $L_1 \cdot L_2/L_1$ is smooth
and connected, it must be trivial. It follows
that $L_2 \subseteq L_1$; this proves the assertion.

(2) Denote by $M \subseteq G$ the pull-back of 
$L(G/L(G))$ under the quotient map $G \to G/L(G)$.
By Proposition \ref{prop:sub}, $M$ is a normal subgroup
scheme of $G$ containing $L(G)$. Moreover, $M$ is affine,
smooth and connected, since so are $L(G)$ and $M/L(G)$.
Thus, $M = L(G)$; this yields the assertion by Proposition 
\ref{prop:sub} again.

(3) This follows from a classical argument of Galois descent,
see \cite[V.22]{Se1}. More specifically, it suffices to check
that the formation of $L(G)$ commutes with Galois extensions. 
Let $K$ be such an extension of $k$, and $\cG$ the Galois group.
Then $\cG$ acts on $G_K = G \times \Spec(K)$ via its action
on $K$. Let $L' : = L(G_K)$; then for any $\gamma \in \cG$,
the image $\gamma(L')$ is also a smooth connected affine normal
$K$-subgroup scheme of $G_K$. Thus, $\gamma(L') \subseteq L'$. 
Since this also holds for $\gamma^{-1}$, we obtain $\gamma(L') = L'$. 
As $G_K$ is covered by $\cG$-stable affine open subschemes,
it follows (by arguing as in \cite[V.20]{Se1}) 
that there exists a unique subscheme $M \subseteq G$
such that $L' = M_K$. Then $M$ is again a smooth connected
affine normal subgroup scheme of $G$, and hence $M \subseteq L(G)$.
On the other hand, we clearly have $L(G)_K \subseteq L'$;
we conclude that $M = L(G)$.
\end{proof}

\begin{remark}\label{rem:aag}
In fact, the formation of $L$ commutes with separable field
extensions that are not necessarily algebraic. This can be shown 
by adapting the proof of \cite[1.1.9]{CGP}, which asserts 
that the formation of the unipotent radical commutes with all
separable field extensions. That proof involves methods of group 
schemes over rings, which go beyond the scope of this text. 
\end{remark}

Our final criterion for affineness is of geometric origin:

\begin{proposition}\label{prop:act}
Let $a : G \times X \to X$ be an action of an algebraic group 
on an irreducible locally noetherian scheme and let $x \in X(k)$. 
Then the quotient group scheme $\C_G(x)/\Ker(a)$ is affine.
\end{proposition}

\begin{proof}
We may replace $G$ with $\C_G(x)$, and hence assume that 
$G$ fixes $x$. Consider the $n$th infinitesimal 
neighborhoods, $x_{(n)} := \Spec(\cO_{X,x}/\fm^{n+1}_x)$, 
where $n$ runs over the positive integers; these form 
an increasing sequence of finite subschemes of $X$ supported 
at $x$. As seen in Example \ref{ex:inf}, 
each $x_{(n)}$ is stabilized by $G$; 
this yields a linear representation $\rho_n$ of $G$ in 
$\cO_{X,x}/\fm^{n+1}_x =:V_n$, a finite-dimensional vector space.
Denote by $N_n$ the kernel of $\rho_n$; then $N_n$ contains 
$\Ker(a)$. As $\rho_n$ is a quotient of $\rho_{n+1}$, 
we have $N_{n+1} \subseteq N_n$. Since $G$ is of finite type,
it follows that there exists $n_0$ such that 
$N_n = N_{n_0} =: N$ for all $n \geq n_0$. 
Then $N$ acts trivially on each subscheme $x_{(n)}$. 
As $X$ is locally noetherian and irreducible, 
the union of these subschemes is dense in $X$; it follows 
that $N$ acts trivially on $X$, by using the representability 
of the fixed point functor $X^G$ (Theorem 
\ref{thm:norcent}). Thus, $N = \Ker(a)$. So 
$\rho_{n_0} : G \to \GL(V_{n_0})$ factors through a closed 
immersion $j: G/\Ker(a) \to \GL(V_{n_0})$ by Proposition 
\ref{prop:hom}.
\end{proof}

\begin{corollary}\label{cor:cent}
Let $G$ be a connected algebraic group and $Z$ its center.
Then $G/Z$ is affine.
\end{corollary}

\begin{proof}
Consider the action of $G$ on itself by inner automorphisms.
Then the kernel of this action is $Z$ and the neutral 
element is fixed. So the assertion follows from Proposition 
\ref{prop:act}.  
\end{proof}

The connectedness assumption in the above corollary cannot
be removed in view of Example \ref{ex:noncon} below.

\subsection{The affinization theorem}
\label{subsec:tat}

Every scheme $X$ is equipped with a morphism to an affine scheme,
namely, the canonical morphism
\[ \varphi = \varphi_X : X \to \Spec \cO(X). \]
The restriction of $\varphi_X$ to any affine open subscheme
$U \subseteq X$ is the morphism $U \to \Spec \cO(X)$ associated
with the restriction homomorphism $\cO(X) \to \cO(U)$.
Moreover, $\varphi$ satisfies the following universal
property: every morphism $f : X \to Y$, where $Y$ is an affine
scheme, factors uniquely through $\varphi$. We say that
$\varphi$ is the \textit{affinization morphism} of $X$, 
and denote $\Spec \cO(X)$ by $\Aff(X)$. When $X$ is of 
finite type, $\Aff(X)$ is not necessarily of finite type;
equivalently, the algebra $\cO(X)$ is not necessarily finitely
generated (even when $X$ is a quasi-projective variety, 
see Example \ref{ex:ell} below).

Also, every morphism of schemes $f: X \to Y$ lies in 
a commutative diagram
\[
\xymatrix{
X \ar[r]^-f \ar[d]_{\varphi_X} 
& Y \ar[d]^{\varphi_Y} \\
\Aff(X) \ar[r]^-{\Aff(f)} & \Aff(Y), \\}
\]
where $\Aff(f)$ is the morphism of affine schemes associated
with the ring homomorphism $f^\# : \cO(Y) \to \cO(X)$. 

For quasi-compact schemes, the formation of the affinization 
morphism commutes with field extensions and finite products, 
as a consequence of Lemma \ref{lem:aff}.
It follows that for any algebraic group $G$, there is a
canonical group scheme structure on $\Aff(G)$ such that
$\varphi_G$ is a homomorphism. Moreover, given an action 
$a$ of $G$ on a quasi-compact scheme $X$, the map
$\Aff(a)$ is an action of $\Aff(G)$ on $\Aff(X)$, compatibly with $a$.

With these observations at hand, we may make an important step 
in the proof of Theorem \ref{thm:affi}:

\begin{theorem}\label{thm:aff}
Let $G$ be an algebraic group, $\varphi : G \to \Aff(G)$
its affinization morphism and $H := \Ker(\varphi)$. 
Then $H$ is the smallest normal subgroup scheme of $G$ 
such that $G/H$ is affine. Moreover, $\cO(H) = k$ and 
$\Aff(G) = G/H$. In particular, $\cO(G) = \cO(G/H)$; 
thus, the algebra $\cO(G)$ is finitely generated.
\end{theorem} 

\begin{proof}
Consider a normal subgroup scheme $N$ of $G$ such that 
$G/N$ is affine. Then we have a commutative diagram
of homomorphisms
\[
\xymatrix{
G \ar[r]^-q \ar[d]_{\varphi_G} 
& G/N \ar[d]^{\varphi_{G/N}} \\
\Aff(G) \ar[r]^-{\Aff(q)} & \Aff(G/N), \\}
\]
where $q$ is the quotient morphism and $\varphi_{G/N}$ 
is an isomorphism. Since $H$ is the fiber of $\varphi_G$ 
at the neutral element $e_G$, it follows that $H \subseteq N$.  

We now claim that $H$ is the kernel of the action of $G$
on $\cO(G)$ via left multiplication. Denote by $K$ the 
latter kernel; we check that $H(R) = K(R)$ for any algebra 
$R$. Note that $H(R)$ consists of those $x \in G(R)$ such 
that $f(x) = f(e)$ for all $f \in \cO(G)$ (since 
$\cO(G \times \Spec(R)) = \cO(G) \otimes_k R$). Also, $K(R)$
consists of those $x \in G(R)$ such that $f(xy) = f(y)$
for all $f \in \cO(G \times \Spec(R'))$ and $y \in G(R')$, 
where $R'$ runs over all $R$-algebras. In particular, 
$f(x) = f(e)$ for all $f \in \cO(G)$, and hence 
$K(R) \subseteq H(R)$. 

To show the opposite inclusion, choose a basis 
$(\varphi_i)_{i \in I}$ of the $k$-vector space $\cO(G)$; 
then the $R'$-module 
$\cO(G \times \Spec(R')) = \cO(G) \otimes_k R'$
is free with basis $(\varphi_i)_{i \in I}$. Thus, for any 
$f \in \cO(G) \otimes_k R'$, there exists a unique family 
$(\psi_i = \psi_i(f))_{i \in I}$ in $\cO(G) \otimes_k R'$ 
such that $f(xy) = \sum_i \psi_i(x) \, \varphi_i(y)$ identically.
So the equalities $f(xy) = f(y)$ for all $y \in G(R')$
are equivalent to the equalities
\[ \sum_i (\psi_i(x) - \psi_i(e)) \, \varphi_i(y) = 0 \]
for all such $y$. 
Since the latter equalities are satisfied for any $x \in H(R)$,
this yields the inclusion $H(R) \subseteq K(R)$,
and completes the proof of the claim.

By Proposition \ref{prop:rep},
there exists an increasing family of finite-dimensional
$G$-submodules $(V_i)_{i \in I}$ of $\cO(G)$ such that 
$\cO(G) = \bigcup_i V_i$. Denoting by $K_i$ the kernel
of the corresponding homomorphism $G \to \GL(V_i)$, 
we see that $H = K$ is the decreasing intersection of the $K_i$.
Since the topological space underlying $G$ is noetherian and 
each $K_i$ is closed in $G$, there exists $i \in I$ such that 
$H = K_i$. It follows that $G/H$ is affine. We have proved 
that $H$ is the smallest normal subgroup scheme of $G$ having 
an affine quotient.

The affinization morphism $\varphi_G$ factors through
a unique morphism of affine schemes $\iota : G/H \to \Aff(G)$. 
The associated homomorphism
\[ \iota^{\#}: \cO(\Aff(G)) = \cO(G) \to \cO(G/H) = \cO(G)^H \]
is an isomorphism; thus, so is $\iota$. This shows that
$\Aff(G) = G/H$.

Next, consider the kernel $N$ of the affinization
morphism $\varphi_H$. Then $N \trianglelefteq H$ and 
the quotient group $H/N$ is affine. Since $G/H$ is
affine as well, it follows by Proposition \ref{prop:tower}
that the homogeneous space $G/N$ is affine. 
Thus, the quotient morphism $G \to G/N$
factors through a unique morphism $\Aff(G) \to G/N$. 
Taking fibers at $e$ yields that $H \subseteq N$; thus, $H = N$. 
Hence the action on $H$ on itself via left multiplication yields 
a trivial action on $\cO(H)$. As $\cO(H)^H = k$, we conclude that
$\cO(H) = k$.
\end{proof}

\begin{corollary}\label{cor:aff}
Let $G$ be an algebraic group acting faithfully on an 
affine scheme $X$. Then $G$ is affine.
\end{corollary}

\begin{proof}
The action of $G$ on $X$ factors through an action of 
$\Aff(G)$ on $\Aff(X) = X$. Thus, the subgroup scheme $H$ 
of Theorem \ref{thm:aff} acts trivially on $X$. Hence $H$ 
is trivial; this yields the assertion.
\end{proof}

\begin{example}\label{ex:ell}
Let $E$ be an elliptic curve equipped with an invertible 
sheaf $\cL$ such that $\deg(\cL) = 0$ and $\cL$ has 
infinite order in $\Pic(E)$. (Such a pair $(E,\cL)$ exists
unless $k$ is algebraic over a finite field,
as follows from \cite{ST}; see also \cite{Ul}). 
Choose an invertible sheaf $\cM$ on $E$ such that 
$\deg(\cM) > 0$. Denote by $L,M$ the line bundles on $E$ 
associated with $\cL,\cM$ and consider their direct sum,
\[ \pi : X := L \oplus M \longrightarrow E. \]
Then $X$ is a quasi-projective variety and
\[ \pi_*(\cO_X) \cong  \bigoplus_{\ell,m} 
\cL^{\otimes \ell} \otimes_{\cO_E}  \cM^{\otimes m}, \]
where the sum runs over all pairs of non-negative integers.
Thus,
\[ \cO(X) \cong \bigoplus_{\ell,m} 
H^0(E,\cL^{\otimes \ell} \otimes_{\cO_E}  \cM^{\otimes m}). \]
In particular, the algebra $\cO(X)$ is equipped with a
bi-grading. If this algebra is finitely generated, then
the pairs $(\ell,m)$ such that $\cO(X)_{\ell,m} \neq 0$
form a finitely generated monoid under componentwise addition;
as a consequence, the convex cone $C \subset \bR^2$ 
generated by these pairs is closed. But we have 
$\cO(X)_{\ell,0} = 0$ for any $\ell \geq 1$,
since $\cL^{\otimes \ell}$ is non-trivial and has degree $0$.
Also, given any positive rational number $t$, we have
$\cO(X)_{n,tn} \neq 0$ for any positive integer $n$
such that $tn$ is integer, since 
$\deg(\cL^{\otimes n} \otimes_{\cO_E} \cM^{\otimes tn}) > 0$. 
Thus, $C$ is not closed, a contradiction. We conclude that
\textit{the algebra $\cO(X)$ is not finitely generated}.
\end{example}

\subsection{Anti-affine algebraic groups}
\label{subsec:aaag}

\begin{definition}\label{def:anti}
An algebraic group $G$ over $k$ is \textit{anti-affine} 
if $\cO(G) = k$.
\end{definition}

By Lemma \ref{lem:aff}, $G$ is anti-affine if and only if 
$G_K$ is anti-affine for some field extension $K$ of $k$.

\begin{lemma}\label{lem:anti}
Every anti-affine algebraic group is smooth and connected.
\end{lemma}

\begin{proof}
Let $G$ be an algebraic group. Recall that the group
of connected components $\pi_0(G) \cong G/G^0$ is finite and
\'etale. Also, $\cO(\pi_0(G)) \cong \cO(G)^{G^0}$ by Remark 
\ref{rem:quot} (ii). If $G$ is anti-affine, then it follows
that $\cO(\pi_0(G)) = k$. Thus, $\pi_0(G)$ is trivial, 
i.e., $G$ is connected.

To show that $G$ is smooth, we may assume that $k$ is
algebraically closed. Then $G_\red$ is a smooth subgroup scheme
of $G$; moreover, the homogeneous space $G/G_\red$ is finite
by Remark \ref{rem:quot}(v). As above, it follows that $G = G_\red$.   
\end{proof}

We now obtain a generalization of a classical rigidity lemma 
(see \cite[p.~43]{Mum}): 

\begin{lemma}\label{lem:rig}
Let $X$, $Y$, $Z$ be schemes such that $X$ is quasi-compact, 
$\cO(X) = k$ and $Y$ is locally noetherian and irreducible.
Let $f : X \times Y \to Z$ be a morphism. 
Assume that there exist $k$-rational points $x_0 \in X$, 
$y_0 \in Y$ such that $f(x,y_0) = f(x_0,y_0)$ identically. 
Then $f(x,y) = f(x_0,y)$ identically.
\end{lemma}

\begin{proof}
Let $z_0:=f(x_0,y_0)$; this is a $k$-rational point of $Z$. 
As in Example \ref{ex:inf}, consider 
the $n$th infinitesimal neighborhoods of this point,
\[ z_{0,(n)} := \Spec(\cO_{Z,z_0}/\fm^{n+1}_{z_0}), \] 
where $n$ runs over the positive integers. These form 
an increasing sequence of finite subschemes of $Z$ 
supported at $z_0$, and one checks as in the above example that 
$X \times y_{0,(n)}$ is contained in the fiber of $f$ at 
$z_{0,(n)}$, where $y_{0,(n)} := \Spec(\cO_{Y,y_0}/\fm^{n+1}_{y_0})$.
In other words, $f$ restricts to a morphism
$f_n : X \times y_{0,(n)} \to z_{0,(n)}$. Consider the associated 
homomorphism of algebras 
$f_n^\# : \cO(z_{0,(n)}) \to \cO(X \times y_{0,(n)})$.
By Lemma \ref{lem:aff} and the assumptions on $X$, we have 
$\cO(X \times y_{0,(n)}) = \cO(X) \otimes_k \cO(y_{0,(n)})
= \cO(y_{0,(n)})$. Since $z_{0,(n)}$ is affine, it follows that 
$f_n$ factors through a morphism $g_n : y_{0,(n)} \to Z$,
i.e., $f_n(x,y) = g_n(y)$ identically. In particular,
$f(x,y) = f(x_0,y)$ on $X \times y_{0,(n)}$.

Next, consider the largest closed subscheme 
$W \subseteq X \times Y$ on which $f(x,y) = f(x_0,y)$,
i.e., $W$ is the pull-back of the diagonal in $Z \times Z$
under the morphism $(x,y) \mapsto (f(x,y),f(x_0,y))$. 
Then $W$ contains $X \times y_{0,(n)}$ for all $n$. 
Since $Y$ is locally noetherian and irreducible, the union 
of the $y_{0,(n)}$ is dense in $Y$. It follows that the union 
of the $X \times y_{0,(n)}$ is dense in $X \times Y$; we conclude 
that $W = X \times Y$.
\end{proof}

\begin{proposition}\label{prop:rig}
Let $H$ be an anti-affine algebraic group, $G$ 
an algebraic group and $f: H \to G$ 
a morphism of schemes such that $f(e_H) = e_G$.
Then $f$ is a homomorphism and factors through
the center of $G^0$.
\end{proposition}

\begin{proof}
Since $H$ is connected by Lemma \ref{lem:anti},
we see that $f$ factors through $G^0$. Thus, we
may assume that $G$ is connected.

Consider the morphism
\[ \varphi : H \times H \longrightarrow G, \quad
(x,y) \longmapsto f(xy) f(y)^{-1} f(x)^{-1}. \]
Then $\varphi(x,e_H) = e_G = \varphi(e_H,e_H)$ 
identically; also, $H$ is irreducible
in view of Lemma \ref{lem:anti}. Thus, the rigidity 
lemma applies, and yields 
$\varphi(x,y) = \varphi(e_H,y) = e_G$
identically. This shows that $f$ is a homomorphism.

The assertion that $f$ factors through the center of 
$G$ is proved similarly by considering the morphism
\[ \psi : H \times G \longrightarrow G, \quad
(x,y) \longmapsto f(x) y f(x)^{-1} y^{-1}. \]
\end{proof}

In particular, every anti-affine group $G$ is commutative.
Also, note that $G/H$ is anti-affine for any subgroup scheme
$H \subseteq G$ (since $\cO(G/H) = \cO(G)^H$).

We may now complete the proof of Theorem \ref{thm:affi}
with the following:

\begin{proposition}\label{prop:antmax}
Let $G$ be an algebraic group and $H$ the kernel of 
the affinization morphism of $G$.

\begin{enumerate}

\item $H$ is contained in the center of $G^0$.

\item $H$ is the largest anti-affine subgroup of $G$.

\end{enumerate}

\end{proposition}

\begin{proof}
(1) By Theorem \ref{thm:aff}, $H$ is anti-affine.
So the assertion follows from Lemma \ref{lem:anti} and 
Proposition \ref{prop:rig}, or alternatively, from Corollary
\ref{cor:cent}. 

(2) Consider another anti-affine subgroup $N \subseteq G$.
Then the quotient group $N/N \cap H$ is anti-affine,
and also affine (since $N/N\cap H$ is isomorphic to 
a subgroup of $G/H$, and the latter is affine). As
a consequence, $N/N \cap H$ is trivial, that is,
$N$ is contained in $H$.
\end{proof}

We will denote the largest anti-affine subgroup of an 
algebraic group $G$ by $G_\ant$. 

For later use, we record the following observations:

\begin{lemma}\label{lem:quotient}
Let $G$ be an algebraic group and $N \trianglelefteq G$ 
a normal subgroup scheme. Then the quotient map 
$G \to G/N$ yields an isomorphism
\[ G_{\ant}/G_{\ant} \cap N 
\stackrel{\cong}{\longrightarrow} (G/N)_{\ant} .\]
\end{lemma}

\begin{proof}
By Proposition \ref{prop:fact}, 
we have a closed immersion of algebraic groups
$G_{\ant}/G_{\ant} \cap N \to G/N$; moreover, $G_{\ant}/G_{\ant} \cap N$
is anti-affine. So we obtain a closed immersion of commutative
algebraic groups $j : G_{\ant}/G_{\ant} \cap N \to (G/N)_{\ant}$. 
Denote by $C$ the cokernel of $j$; then $C$ is anti-affine
as a quotient of $(G/N)_{\ant}$. Also, $C$ is a subgroup of 
$(G/N)/(G_{\ant}/G_{\ant} \cap N)$, which is a quotient group
of $G/G_{\ant}$. Since the latter group is affine, it follows 
that $C$ is affine as well, by using Proposition \ref{prop:tower}. 
Thus, $C$ is trivial, i.e., $j$ is an isomorphism.
\end{proof}

\begin{lemma}\label{lem:prop}
The following conditions are equivalent for an algebraic 
group~$G$:

\begin{enumerate}

\item $G$ is proper.

\item $G_\ant$ is an abelian variety and $G/G_\ant$ is finite.

\end{enumerate}

Under these conditions, we have $G_\ant = G^0_\red$; in particular,
$G^0_\red$ is a smooth connected algebraic group and its 
formation commutes with field extensions.
\end{lemma}

\begin{proof}
(1)$\Rightarrow$(2) As $G_\ant$ is smooth, connected and proper,
it is an abelian variety. Also, $G/G_\ant$ is proper and affine,
hence finite.

(2)$\Rightarrow$(1) This follows from Remark \ref{rem:tower}.

For the final assumption, note that the quotient group scheme
$G^0/G_\ant$ is finite and connected, hence infinitesimal. 
So the algebra $\cO(G^0/G_\ant)$ is local with residue field $k$
(via evaluation at $e$). It follows that $(G^0/G_\ant)_\red = e$
and hence that $G^0_\red \subseteq G_\ant$; this yields the assertion.
\end{proof}

\medskip

\noindent
\textit{Notes and references}.

Some of the main results of this section originate in Rosenlicht's 
article \cite{Ro1}. More specifically, Corollary \ref{cor:cent} 
is a scheme-theoretic version of \cite[Thm.~13]{Ro1}, and Theorem
\ref{thm:aff}, of \cite[Cor.~3, p.~431]{Ro1}. 

Also, Theorem \ref{thm:aff}, Lemma \ref{lem:anti} and
Proposition \ref{prop:rig} are variants of results from
\cite[III.3.8]{DG}.

The rigidity lemma \ref{lem:rig} is a version of 
\cite[Thm.~1.7]{SS}.

\section{Proof of Theorem \ref{thm:che}}
\label{sec:ptche}

\subsection{The Albanese morphism}
\label{subsec:av}

Throughout this subsection, $A$ denotes an abelian variety, i.e.,
a smooth connected proper algebraic group. Then $A$ is commutative 
by Corollary \ref{cor:cent}. Thus, we will denote the group law 
additively; in particular, the neutral element will be denoted by $0$.
Also, the variety $A$ is projective (see \cite[p.~62]{Mum}).

\begin{lemma}\label{lem:rat}
Every morphism $f: \bP^1 \to A$ is constant.
\end{lemma}

\begin{proof}
We may assume that $k$ is algebraically closed. 
Suppose that $f$ is non-constant and denote by $C \subseteq A$ 
its image, with normalization $\eta : \tilde{C} \to C$.
Then $f$ factors through a surjective morphism
$\bP^1 \to \tilde{C}$. By L\"uroth's theorem, it follows that 
$\tilde{C} \cong \bP^1$. Thus, $f$ factors through the normalization 
$\eta : \bP^1 \to C$ and hence it suffices to show 
that $\eta$ is constant. In other words, we may assume 
that $f$ is birational to its image. Then the differential 
\[ df : T_{\bP^1} \longrightarrow f^*(T_A) \]
is non-zero at the generic point of $\bP^1$ and hence is
injective. Since the tangent sheaf $T_A$ is trivial and 
$T_ {\bP^1} \cong \cO_{\bP^1}(2)$, we obtain an injective map 
$\cO_{\bP^1} \to  \cO_{\bP^1}(-2)^{\oplus n}$, where $n := \dim(A)$. 
This yields a contradiction, since $H^0(\bP^1,\cO_{\bP^1}) = k$
while  $H^0(\bP^1,\cO_{\bP^1}(-2)) = 0$.
\end{proof}

\begin{theorem}\label{thm:weil}
Let $X$ be a smooth variety and $f : X  \dasharrow A$ 
a rational map. Then $f$ is a morphism.
\end{theorem}

\begin{proof}
Again, we may assume $k$ algebraically closed.
View $f$ as a morphism $U \to A$, where $U \subseteq X$
is a non-empty open subvariety. Denote by $Y \subseteq X \times A$ 
the closure of the graph of $f$, with projections
$p_1 : Y \to X$, $p_2 : Y \to A$. Then $p_1$ is proper (since so 
is $A$) and birational (since it restricts to an isomorphism 
over $U$). Assume that $p_1$ is not an isomorphism. Then 
$p_1$ contracts some rational curve in $Y$, i.e., there
exists a non-constant morphism $g : \bP^1 \to Y$ such that
$p_1 \circ g$ is constant (see \cite[Prop.~1.43]{De}). 
It follows that $p_2 \circ g : \bP^1 \to A$ is non-constant; 
but this contradicts Lemma \ref{lem:rat}.
\end{proof}

\begin{lemma}\label{lem:alb}
Let $X$, $Y$ be varieties equipped with $k$-rational
points $x_0$, $y_0$ and let $f : X \times Y \to A$
be a morphism. Then we have identically
\[ f(x,y) - f(x_0,y) - f(x,y_0) + f(x_0,y_0) = 0. \]
\end{lemma}

\begin{proof}
By a result of Nagata (see \cite{Na1,Na2}, and \cite{Lu}
for a modern proof), we may choose a compactification of $X$, 
i.e., an open immersion $X \to \bar{X}$, where $\bar{X}$ 
is a proper variety. Replacing $\bar{X}$ with its normalization, 
we may further assume that $\bar{X}$ is normal. Also, we may 
assume $k$ algebraically closed from the start.

Denote by $U$ the smooth locus of $\bar{X}$ and by $V$ the 
smooth locus of $Y$. By Theorem \ref{thm:weil}, the rational 
map $f: \bar{X} \times Y \dasharrow A$ yields a morphism 
$g : U \times V \to A$. Also, note that the complement 
$\bar{X} \setminus U$ has codimension at least $2$, 
since $\bar{X}$ is normal and $k = \bar{k}$. It follows that 
$\cO(U) = \cO(\bar{X}) = k$. Using again the assumption that 
$k = \bar{k}$, we may choose points $x_1 \in U(k)$, $y_1 \in V(k)$. 
Consider the morphism
\[ \varphi : U \times V \longrightarrow A, \quad 
(x,y) \longmapsto g(x,y) - g(x_1,y) - g(x,y_1) + g(x_1,y_1). \]
Then $\varphi(x,y_1) = 0$ identically, and hence $\varphi = 0$ 
by rigidity (Lemma \ref{lem:rig}). It follows that
$f(x,y) - f(x_1,y) - f(x,y_1) + f(x_1,y_1) = 0$ identically
on $X \times Y$. This readily yields the desired equation. 
\end{proof}

\begin{proposition}\label{prop:alba}
Let $G$ be a smooth connected algebraic group.

\begin{enumerate}

\item Let $f: G \to A$ be a morphism to an abelian variety
sending $e$ to $0$. Then $f$ is a homomorphism.

\item There exists a smallest normal subgroup scheme 
$N$ of $G$ such that the quotient $G/N$ is an abelian variety.
Moreover, $N$ is connected.

\item For any abelian variety $A$, every morphism $f : G \to A$ 
sending $e$ to $0$ factors uniquely through the quotient
homomorphism $\alpha : G \to G/N$.

\item The formation of $N$ commutes with separable
algebraic field extensions.

\end{enumerate}

\end{proposition}
  
\begin{proof}
(1) This follows from Lemma \ref{lem:alb} applied
to the map $G \times G \to A$, $(x,y) \mapsto f(xy)$
and to $x_0 = y_0 = e$.

(2) Consider two normal subgroup schemes $N_1$, $N_2$ of $G$ 
such that $G/N_1$, $G/N_2$ are abelian varieties. Then 
$N_1 \cap N_2$ is normal in $G$ and $G/N_1 \cap N_2$ is proper,
since the natural map 
\[ G/N_1 \cap N_2 \longrightarrow G/N_1 \times G/N_2 \]
is a closed immersion of algebraic groups
(Proposition \ref{prop:hom}). Since $G/N_1 \cap N_2$
is smooth and connected, it is an abelian variety as well. It follows 
that there exists a smallest such subgroup scheme, say $N$.

We claim that the neutral component $N^0$ is normal in $G$.
To check this, we may assume $k$ algebraically closed.
Then $G(k)$ is dense in $G$ and normalizes $N^0$; this yields
the claim. 
 
The natural homomorphism $G/N^0 \to G/N$ is finite, 
since it is a torsor under the finite group $N/N^0$ (Propositions 
\ref{prop:des} and \ref{prop:quot}). 
As a consequence, $G/N^0$ is proper; hence 
$N  = N^0$ by the minimality assumption. Thus, $N$ is connected. 

(3) By (1), $f$ is a homomorphism; denote its kernel by $H$.
Then $G/H$ is an abelian variety in view of Proposition 
\ref{prop:fact}. Thus, $H$ contains $N$, and the assertion 
follows from Proposition \ref{prop:quot}.

(4) This is checked by a standard argument of Galois descent,
as in the proof of Lemma \ref{lem:ar}.
\end{proof}

\begin{remark}\label{rem:av}
In fact, the formation of $N$ commutes with all separable
field extensions. This may be proved as sketched in Remark 
\ref{rem:aag}.
\end{remark}

With the notation and assumptions of Proposition \ref{prop:alba},
we say that the quotient morphism 
\[ \alpha : G \longrightarrow G/N \] 
is the \textit{Albanese homomorphism of $G$}, and 
$G/N =: \Alb(G)$ the \textit{Albanese variety}. 

Actually, Proposition \ref{prop:alba} (3) extends to any 
\textit{pointed variety}, i.e., a variety equipped with 
a $k$-rational point:

\begin{theorem}\label{thm:alba}
Let $(X,x)$ be a pointed variety. 
Then there exists an abelian variety $A$ and a morphism 
$\alpha = \alpha_X : X \to A$ sending $x$ to $0$, such that 
for any abelian variety $B$, every morphism $X \to B$ sending 
$x$ to $0$ factors uniquely through~$\alpha$. 
\end{theorem}

\begin{proof}
See \cite[Thm.~5]{Se2} when $k$ is algebraically closed;
the general case is obtained in \cite[Thm.~A1]{Wit}. 
\end{proof}

The morphism $\alpha : X \to A$ in the above theorem is uniquely 
determined by $(X,x)$; it is called the \textit{Albanese morphism},
and $A$ is again the \textit{Albanese variety}, denoted by $\Alb(X)$.
Combining that theorem with Theorem \ref{thm:weil}
and Lemma \ref{lem:alb}, we obtain readily:

\begin{corollary}\label{cor:alba}

\begin{enumerate}

\item For any smooth pointed variety $(X,x)$ and 
any open subvariety $U \subseteq X$ containing $x$,
we have a commutative square
\[
\xymatrix{
U \ar[r]^-{\alpha_U} \ar[d]_{j} & \Alb(U) \ar[d]^{\Alb(j)} \\ 
X \ar[r]^-{\alpha_X} & \Alb(X),  \\}
\]
where $j: U \to X$ denotes the inclusion and
$\Alb(j)$ is an isomorphism.

\item For any pointed varieties $(X,x)$, $(Y,y)$, 
we have $\alpha_{X \times Y} = \alpha_X \times \alpha_Y$. 
In particular, the natural map 
\[ \Alb(X \times Y) \longrightarrow \Alb(X) \times \Alb(Y) \] 
is an isomorphism.

\item Any action $a$ of a smooth connected algebraic 
group $G$ on a pointed variety $X$ yields a unique homomorphism 
$f = f_a : \Alb(G) \to \Alb(X)$
such that the square
\[
\xymatrix{
G \times X \ar[r]^-{a} \ar[d]_{\alpha_G \times \alpha_X} & X \ar[d]^{\alpha_X} \\ 
\Alb(G) \times \Alb(X) \ar[r]^-{f + \id} & \Alb(X)  \\}
\]
commutes.

\end{enumerate}

\end{corollary}

The formation of the Albanese morphism commutes with separable
algebraic field extensions, by Galois descent again. But it does 
not commute with arbitrary field extensions, as shown by Example
\ref{ex:ray} below.

\subsection{Abelian torsors}
\label{subsec:at}

\begin{lemma}\label{lem:mult}
Let $A$ an abelian variety and $n$ a non-zero integer. 
Then the multiplication map
\[ n_A : A \longrightarrow A, \quad x \longmapsto n x \]
is an isogeny.
\end{lemma}

\begin{proof}
See \cite[p.~62]{Mum}.
\end{proof}

\begin{example}\label{ex:noncon}
With the above notation, consider the semi-direct product 
$G := \bZ/2 \ltimes A$, where $\bZ/2$ (viewed as a constant
group scheme) acts on $A$ via $x \mapsto \pm x$. One may check
that the center $Z$ of $G$ is the kernel of the multiplication
map $2_A$ and hence is finite; thus, $G/Z$ is not affine. So 
Corollary \ref{cor:cent} does not extend to disconnected 
algebraic groups. 

Also, note that $G$ is not contained in any connected algebraic 
group, as follows from Proposition \ref{prop:rig}.
\end{example}

\begin{lemma}\label{lem:finite}
Let $G$ be a smooth connected commutative algebraic group, 
with group law denoted additively, and let $f : X \to \Spec(k)$ 
be a $G$-torsor. Then there exists a positive integer $n$ 
and a morphism $\varphi : X \to G$ such that 
\[ \varphi(g \cdot x) = \varphi(x) + n g \] 
identically on $G \times X$. 
\end{lemma}

\begin{proof}
If $X$ has a $k$-rational point $x$, then the orbit map 
$a_x : G  \to X$,  $g \mapsto g \cdot x$
is a $G$-equivariant isomorphism, where $G$ acts on itself by
translation. So we may just take $\varphi = a_x^{-1}$ and $n = 1$
(i.e., $\varphi$ is $G$-equivariant).

In the general case, since $X$ is a smooth variety, it has a $K$-rational 
point $x_1$ for some finite Galois extension of fields $K/k$. Denote by 
$\cG$ the corresponding Galois group and by $x_1,\ldots, x_n$ 
the distinct conjugates of $x_1$ under $\cG$. Then the first step
yields $G_K$-equivariant isomorphisms
$g_i : X_K \to G_K$ for $i = 1,\ldots,n$, such that 
$x = g_i(x) \cdot x_i$ identically. Consider the morphism
\[ \phi : X_K \longrightarrow G_K, \quad 
x \longmapsto g_1(x) + \cdots + g_n(x). \]
Then $\phi$ is equivariant under $\cG$, since that group permutes
the $x_i$'s and hence the $g_i$'s. Also, we have 
$\phi(g \cdot x) = \phi(x) + n g$ identically. So $\phi$ descends 
to the desired morphism $X \to G$.
\end{proof}

\begin{proposition}\label{prop:cov}
Let $A$ be an abelian variety and $f : X \to Y$ an $A$-torsor, 
where $X,Y$ are smooth varieties. Then there exists a positive integer 
$n$ and a morphism $\varphi : X \to A$ such that 
$\varphi(a \cdot x) = \varphi(x) + n a$ identically on $A \times X$. 
\end{proposition}

\begin{proof}
Let $\eta_Y : \Spec k(Y) \to Y$ be the generic point.
Then the base change $X \times_Y \Spec k(Y) \to \Spec k(Y)$
is an $A_{k(Y)}$-torsor. Using Lemma \ref{lem:finite}, we obtain a map 
$\psi : X \times_Y \Spec k(Y) \to A_{k(Y)}$ satisfying the required 
covariance property. Composing $\psi$ with the natural maps 
$\eta_X \times f : \Spec k(X) \to X \times_Y \Spec k(Y)$ 
and $\pi : A_{k(Y)} \to A$ yields a map $\Spec k(X) \to A$, 
which may be viewed as a rational map $X  \dasharrow A$ and hence 
(by Theorem \ref{thm:weil}) as a morphism $\varphi : X \to A$. 
Clearly, $\varphi$ satisfies the same covariance property as $\psi$.
\end{proof}

\begin{theorem}\label{thm:qc}
Let $G$ be a smooth connected algebraic group and $A \subseteq G$ 
an abelian subvariety. Then $A \subseteq Z(G)$ and there exists 
a connected normal subgroup scheme $H \subseteq G$ such that
$G = A \cdot H$ and $A \cap H$ is finite. If $k$ is perfect, 
then we may take $H$ smooth.
\end{theorem}

\begin{proof}
By Proposition \ref{prop:rig}, $A$ is contained in the center of $G$.
The quotient map $q : G \to G/A$ is an $A$-torsor; also, $G/A$ is smooth,
since so is $G$. Thus, Proposition \ref{prop:cov} yields a map
$\varphi : G \to A$ such that $\varphi(a g) = \varphi(g) + n a$ identically,
for some integer $n > 0$. Composing $\varphi$ with a translation of $A$,
we may assume that $\varphi(e_G) = 0$. Then $\varphi$ is a homomorphism
in view of Proposition \ref{prop:alba}; its restriction
to $A$ is the multiplication $n_A$.

We claim that $G = A \cdot \Ker(\varphi)$. Since $G$ is smooth,
it suffices by Lemma \ref{lem:prod} to show the equality 
$G(\bar{k}) = A(\bar{k}) \Ker(\varphi)(\bar{k})$.
Let $g \in G(\bar{k})$; by Lemma \ref{lem:mult}, 
there exists $a \in A(\bar{k})$ such that $\varphi(g) = n a$.  
Thus, $\varphi(a^{-1}g)=0$; this yields the claim.

Let $H := \Ker(\varphi)^0$. Then $A \cap H$ is finite, 
since it is contained in $A \cap \Ker(\varphi) = \Ker(n_A)$. 
We now show that $G = A \cdot H$. The homogeneous space 
$G/A \cdot H$ is smooth and connected, since so is $G$.
On the other hand, $G/A \cdot H = A \cdot \Ker(\varphi)/A \cdot H$ 
is homogeneous under $\Ker(\varphi)/H = \pi_0(\Ker(\varphi))$, 
and hence is finite. Thus, $G/A \cdot H$ is trivial; this yields 
the assertion. Since $H$ centralizes $A$, it follows that 
$H$ is normal in $G$.

Finally, if $k$ is perfect, then we may replace $H$ with $H_\red$
in the above argument.
\end{proof}

\begin{corollary}\label{cor:poin}
Let $A$ be an abelian variety and $B \subseteq A$ an abelian 
subvariety. Then there exists an abelian subvariety 
$C \subseteq A$ such that $A = B + C$ and $B \cap C$ is finite.
\end{corollary}

\begin{proof}
By the above theorem, there exists a connected subgroup
scheme $H$ of $A$ such that $A = B + H$ and $B \cap H$ 
is finite. Now replace $H$ with $H_\red$, which is an abelian
subvariety by Lemma \ref{lem:prop}.
\end{proof}

The following example displays several specific features of 
algebraic groups over imperfect fields. It is based on Weil 
restriction; we refer to Appendix A of \cite{CGP} for 
the definition and main properties of this notion.

\begin{example}\label{ex:ray}
Let $k$ be an imperfect field. Choose a non-trivial finite
purely inseparable extension $K$ of $k$. Let $A'$ be 
a non-trivial abelian variety over $K$; then the Weil 
restriction $\R_{K/k}(A') =: G$ is a smooth connected commutative
algebraic group over $k$, of dimension $[K:k] \, \dim(A')$.
Moreover, there is an exact sequence of algebraic groups over $K$
\begin{equation}\label{eqn:ray} 
1 \longrightarrow U' \longrightarrow G_K 
\stackrel{q}{\longrightarrow} A' \longrightarrow 1, 
\end{equation}
where $U'$ is non-trivial, smooth, connected and unipotent
(see \cite[A.5.11]{CGP}). It follows readily that 
$q$ is the Albanese homomorphism of $G_K$. 

Let $H$ be a smooth connected affine algebraic group over $k$.
We claim that every homomorphism $f : H \to G$ is constant.
Indeed, the morphisms of $k$-schemes $f: H \to G$ correspond
bijectively to the morphisms of $K$-schemes 
$f' : H_K \to A'$, via the assignment 
$f \mapsto f' := q \circ f_K$ (see \cite[A.5.7]{CGP}). 
Since $q$ is a homomorphism, this bijection sends homomorphisms 
to homomorphisms. As every homomorphism $H_K \to A'$
is constant, this proves the claim.
 
Next, consider the Albanese homomorphism 
$\alpha : G \to \Alb(G)$. If $\alpha_K = q$, then 
$\Ker(\alpha)_K = \Ker(\alpha_K) = U'$; 
as a consequence, $\Ker(\alpha)$ is smooth, connected
and affine. By the claim, it follows that $\Ker(\alpha)$ is 
trivial, i.e., $G$ is an abelian variety. Then so is $G_K$,
but this contradicts the non-triviality of $U'$.
So \textit{the formation of the Albanese morphism does not 
commute with arbitrary field extensions}.

Note that $U' = L(G_K)$ (the largest smooth connected
normal affine subgroup scheme of $G_K$, introduced in
Lemma \ref{lem:ar}). Also, $L(G)$ is trivial
by the claim. Thus, \textit{the formation of $L(G)$
does not commute with arbitrary field extensions}.

Likewise, $G$ \textit{is not an extension of an abelian 
variety by a smooth connected affine algebraic group}.
We will see in Theorem \ref{thm:chev} that every smooth
connected algebraic group over a \textit{perfect} field
lies in a unique such extension.

For later use, we analyze the structure of $G$ in more 
detail. We claim that 
\textit{there is a unique exact sequence
\[ 1 \longrightarrow A \longrightarrow G 
\longrightarrow U \longrightarrow 1, \]
where $A$ is an abelian variety and $U$ is unipotent.}
This is equivalent to the assertion that $G_\ant$ 
is an abelian variety, and $G/G_\ant$ is unipotent. 
So it suffices to show the corresponding assertion 
for $G_K$. We may choose a positive integer $n$
such that $p^n_{U'} = 0$. Thus, the extension (\ref{eqn:ray})
is trivialized by push-out via $p^n_{U'}$. It follows
that (\ref{eqn:ray}) is also trivialized by pull-back 
via $p^n_{A'}$, and hence $G_K \cong (U' \times A')/F$
for some finite group scheme $F$ (isomorphic to the kernel
of $p^n_{A'}$). So the image of $A'$ in $G_K$ 
is an abelian variety with unipotent quotient; this
proves the claim.  

Next, we claim that \textit{there exists a connected
unipotent subgroup scheme $V \subseteq G$ such that
$G = A \cdot V$ and $A \cap V$ is finite.} This 
follows from Theorem \ref{thm:qc}, or directly by
taking for $V$ the neutral component of $\Ker(p^m_G)$, 
where $m$ is chosen so that $p^m_U = 0$. Yet 
\textit{there exists no smooth connected subgroup scheme 
$H \subseteq G$ such that $G = A \cdot H$
and $A \cap H$ is finite}. 
Otherwise, the quotient homomorphism $G \to U$ restricts 
to a homomorphism $H \to U$ with finite kernel. Thus, $H$ is
affine; also, $G$ is an extension of the abelian variety
$A/A \cap H$ by $H$. This yields a contradiction.
\end{example}

\subsection{Completion of the proof of Theorem \ref{thm:che}} 
\label{subsec:compl}

The final ingredient in Rosenlicht's proof of the Chevalley structure
theorem is the following:

\begin{lemma}\label{lem:exi}
Every non-proper algebraic group over an algebraically closed
field contains an affine subgroup scheme of positive dimension.
\end{lemma}

We now give a brief outline of the proof of this lemma, 
which is presented in detail in 
\cite[Sec.~2.3]{BSU}; see also \cite[Sec.~4]{Mi}. 
Let $G$ be a non-proper algebraic group over $k = \bar{k}$.
Then $G^0_\red$ is non-proper as well, and hence we may 
assume that $G$ is smooth and connected. By \cite{Na1, Na2}, 
there exists a compactification $X$ of $G$, i.e., $X$ is a proper 
variety containing $G$ as an open subvariety; then the boundary 
$X \setminus G$ is non-empty. The action of $G$ by left 
multiplication on itself induces a faithful rational action 
\[ a : G \times X \dasharrow X. \]
One shows (this is the main step of the proof) that there exists 
a proper variety $X'$ and a birational morphism 
$\varphi : X' \to X$ such that the induced birational action 
$a' : G \times X' \dasharrow X'$ normalizes some irreducible 
divisor $D \subset X'$, i.e., $a'$ induces a rational action 
$G \times D \dasharrow D$. Then one considers the ``orbit map''
$a'_x$ for a general point $x \in D$, and the corresponding
``stabilizer'' $C_G(x)$ (these have to be defined 
appropriately). By adapting the argument of Proposition
\ref{prop:act}, one shows that $\C_G(x)$ is affine; it has
positive dimension, since 
$\dim(G/C_G(x)) \leq \dim(D) = \dim(G) -1$. So $\C_G(x)$ is
the desired subgroup scheme.

\medskip

We now show how to derive Theorem \ref{thm:che} from Lemma
\ref{lem:exi}, under the assumptions that $k$ is perfect
and $G$ is smooth and connected. We then have the following 
more precise result, which is a version of Chevalley's structure 
theorem:

\begin{theorem}\label{thm:chev}
Let $G$ be a smooth connected algebraic group over a perfect field
$k$, and $L = L(G)$ the largest smooth connected affine
normal subgroup scheme. 

\begin{enumerate}

\item $L$ is the kernel of the Albanese homomorphism of $G$.

\item The formation of $L$ commutes with field extensions.

\end{enumerate}

\end{theorem}

\begin{proof}
(1) Recall that the existence of $L(G)$ has been obtained in 
Lemma \ref{lem:ar}, as well as the triviality of $L(G/L(G))$.
We may thus replace $G$ with $G/L$ and assume that $L$ is trivial.
Also, since the formations of $L(G)$ and of the Albanese homomorphism
commute with algebraic field extensions (see Lemma
\ref{lem:ar} again, and Proposition \ref{prop:alba}), we may assume
that $k$ is algebraically closed.

Consider the center $Z$ of $G$. If $Z$ is proper, then 
its reduced neutral component $Z^0_\red$ is an abelian variety, 
say $A$. Since $G/Z$ is affine (Corollary \ref{cor:cent}) 
and $Z/A$ is finite, $G/A$ is affine as well by Proposition 
\ref{prop:tower}. Also, by Theorem \ref{thm:qc}, there exists 
a smooth connected normal subgroup scheme $H \trianglelefteq G$ 
such that $G = A \cdot H$ and $A \cap H$ is finite. Then 
$H/A \cap H \cong G/A$ is affine, and hence $H$ is affine 
by Proposition \ref{prop:tower} again. It follows that 
$H$ is trivial, and $G = A$ is an abelian variety.

On the other hand, if $Z$ is not proper, then it contains an 
affine subgroup scheme $N$ of positive dimension by Lemma
\ref{lem:exi}. Thus, $N^0_\red$ is a non-trivial smooth connected 
central subgroup scheme of $G$. But this contradicts the
assumption that $L(G)$ is trivial.

(2) Let $K$ be a field extension of $k$. Then the exact sequence
\[ 1 \longrightarrow L \longrightarrow G 
\longrightarrow A \longrightarrow 1 \] 
yields an exact sequence of smooth connected algebraic groups over $K$
\[ 1 \longrightarrow L_K \longrightarrow G_K
\longrightarrow A_K \longrightarrow 1, \]
where $L_K$ is affine and $A_K$ is an abelian variety. 
It follows readily that $L_K$ equals $L(G_K)$ and is 
the kernel of the Albanese homomorphism of $G_K$.
\end{proof}

\begin{remark}
With the notation and assumptions of the above theorem,
every smooth connected affine subgroup scheme of $G$ 
(not necessarily normal) has a trivial image in the abelian
variety $A$. Thus, $L$ is the largest smooth connected affine 
subgroup scheme of $G$.
\end{remark}

We now return to an arbitrary field and obtain the following:

\begin{theorem}\label{thm:cheva}
Let $G$ be a smooth algebraic group and denote by $N$ the kernel 
of the Albanese homomorphism of $G^0$. Then $N$ is the smallest 
normal subgroup scheme of $G$ such that the quotient is proper.
Moreover, $N$ is affine and connected.
\end{theorem}

\begin{proof}
By Proposition \ref{prop:alba}, $N$ is connected. Also, as
$G^0/N$ is proper and $G/G^0$ is finite, $G/N$ is proper.

We now show that $N$ is normal in $G$. For this, we may 
assume that $k$ is separably closed, since the formation 
of $N$ commutes with separable algebraic field extensions 
(Proposition \ref{prop:alba} again). Then $G(k)$ is dense 
in $G$ by smoothness, and normalizes $N$ by the uniqueness 
of the Albanese homomorphism. Thus, $G$ normalizes $N$.

Next, we show that $N$ is contained in every normal 
subgroup scheme $H \trianglelefteq G$ such that $G/H$ is
proper. Indeed, one sees as above that $H^0 \trianglelefteq G$
and $G/H^0$ is proper as well; hence $G^0/H^0$ is an abelian 
variety. Thus, $H^0 \supseteq N$ as desired.

Finally, we show that $N$ is affine. For this, we may
assume that $G$ is connected. By Theorem \ref{thm:chev}, 
there exists an exact sequence of 
algebraic groups over the perfect closure $k_i$,
\[ 1 \longrightarrow L_i \longrightarrow G_{k_i} 
\longrightarrow A_i \longrightarrow 1, \]
where $L_i$ is smooth, connected and affine, and $A_i$ is an
abelian variety. This exact sequence is defined over some
subfield $K \subseteq k_i$, finite over $k$. In other words,
there exists a finite purely inseparable field extension $K$  
of $k$ and a smooth connected affine normal subgroup scheme
$L' \trianglelefteq G_K$ such that $G_K/L'$ is an
abelian variety.

By Lemma \ref{lem:ray} below, we may choose a subgroup scheme
$L \subseteq G$ such that $L_K \supseteq L'$ and $L_K/L'$
is finite. As a consequence, $L_K \trianglelefteq G_K$
and hence $L \trianglelefteq G$. Also, $L_K$ is affine 
and hence $L$ is affine. Finally, 
$(G/L)_K \cong G_K/L_K \cong (G_K/L')/(L_K/L')$ 
is an abelian variety, and hence $G/L$ is an abelian variety.
Thus, $N \subseteq L$ is affine.
\end{proof}

\begin{lemma}\label{lem:ray}
Let $G$ be an algebraic group over $k$. Let $K$ be a finite 
purely inseparable field extension of $k$, and 
$H' \subseteq G_K$ a $K$-subgroup scheme. Then there exists 
a $k$-subgroup scheme $H \subseteq G$ such that 
$H' \subseteq H_K$ and $H_K/H'$ is finite. 
\end{lemma}

\begin{proof}
We may choose a positive integer $n$ such that 
$K^{p^n} \subseteq k$. Consider the $n$th relative Frobenius 
homomorphism
\[ F^n := F^n_{G_K/K} : G_K \longrightarrow G_K^{(n)} \]
as in \S \ref{subsec:trfm}. Denote by $H'_n$ the pull-back
under $F^n$ of the subgroup scheme $H'^{(n)}$ of $G_K^{(n)}$. 
Then $H' \subseteq H'_n$
and the quotient $H'_n/H'$ is finite, since $F^n$ is the
identity on the underlying topological spaces and remains
so over $\bar{k}$. Denote by $\cI' \subset \cO_{G_K}$ 
the sheaf of ideals of $H'$; then the sheaf of ideals 
$\cI'_n$ of $H'_n$ is generated by the $p^n$th powers of 
local sections of $\cI'$, as follows from (\ref{eqn:trfm}).
Since $K^{p^n} \subseteq k$, every such power lies in 
$\cO_G$. Thus, $\cI'_n = \cI_K$ for a unique sheaf of
ideals $\cI \subset \cO_G$. The corresponding closed 
$k$-subscheme $H \subseteq G$ satisfies $H_K = H'_n$,
and hence is the desired subgroup scheme.
\end{proof}

\begin{remark}\label{rem:ext}
Let $G$ be a smooth connected algebraic group over $k$
with Albanese homomorphism $\alpha : G \to \Alb(G)$ 
and consider a field extension $K$ of $k$. 
Then the homomorphism $\alpha_K : G_K \to \Alb(G)_K$ 
factors through a unique homomorphism
\[ f : \Alb(G_K) \longrightarrow \Alb(G)_K. \]
Since $\alpha$ is surjective, $f$ is surjective as well.
Moreover, $\Ker(f)$ is infinitesimal: indeed, denoting 
by $\alpha' : G_K \to \Alb(G_K)$ the Albanese homomorphism,
we have $\Ker(\alpha') \subseteq \Ker(\alpha_K)$
and $\Ker(f) \cong  \Ker(\alpha_K)/\Ker(\alpha')$.
In particular, $\Ker(f)$ is affine and connected. 
But $\Ker(f)$ is also a subgroup scheme of the abelian 
variety $\Alb(G_K)$; hence it must be finite and local.

In loose terms, the formation of the Albanese homomorphism 
commutes with field extensions up to purely inseparable 
isogenies.
\end{remark}

Since every algebraic group is an extension of a smooth
algebraic group by an infinitesimal one (Proposition
\ref{prop:frob}), Theorem \ref{thm:cheva} implies readily:

\begin{corollary}\label{cor:ray}
Every connected algebraic group $G$ is an extension of
an abelian variety by a connected affine algebraic group.
\end{corollary}

Yet $G$ may contain no smallest connected affine subgroup
scheme with quotient an abelian variety, as shown by 
the following:

\begin{example}\label{ex:nonu}
Assume that $\car(k) = p > 0$ and let $A$ be a non-trivial abelian 
variety. Then $A$ contains two non-trivial infinitesimal
subgroup schemes $I$, $J$ such that $J \subsetneq I$: 
we may take $I = \Ker(F_{A/k})$ and $J = \Ker(F^2_{A/k})$,
where $F^n_{A/k}$ denotes the $n$th relative Frobenius morphism
as in \S \ref{subsec:trfm}. Thus, 
\[ G := (A \times I)/\diag(J) \] 
is a connected commutative algebraic group, where 
$\diag(x) := (x,x)$. Consider the infinitesimal subgroup schemes 
$N_1 := (J \times I)/\diag(J)$ and $N_2 := \diag(I)/\diag(J)$ 
of $G$. Then $G/N_1 \cong A/J$ and $G/N_2 \cong A$ are abelian 
varieties. Also, $N_1 \cap N_2$ is trivial, and $G$ is not 
an abelian variety. 
\end{example}

To complete the proof of Theorem \ref{thm:che}, it remains
to treat the general case, where $G$ is an arbitrary algebraic 
group over an arbitrary field. We then have to prove: 

\begin{lemma}\label{lem:chevalley}
Any algebraic group $G$ has a smallest normal subgroup 
scheme $N$ such that $G/N$ is proper. Moreover, $N$ is affine 
and connected.
\end{lemma}

\begin{proof}
The existence of $N$ is obtained as in the proof of Proposition 
\ref{prop:alba}. 

We now claim that there exists a connected affine normal subgroup 
$H \trianglelefteq G$ such that $G/H$ is proper. For this, we may 
reduce to the case where $G$ is smooth by using Proposition 
\ref{prop:frob} if $\car(k) > 0$. Then we may take for $H$ 
the kernel of the Albanese homomorphism of $G^0$ (Theorem
\ref{thm:cheva}); this proves the claim.

Given $H$ as in that claim, $N$ is a normal subgroup scheme 
of $H$, and hence is affine. Moreover, $H/N$ is affine, connected,
and proper since $G/N$ is proper. Thus, $H/N$ is infinitesimal;
it follows that $N$ is connected.
\end{proof}

\medskip

\noindent
\textit{Notes and references}.

Theorem \ref{thm:weil} is called 
\textit{Weil's extension theorem}. The proof presented here is 
taken from that of \cite[Cor.~1.44]{De}.

Many results of this section are due to Rosenlicht. 
More specifically, Lemma \ref{lem:finite} is a version of
\cite[Thm.~14]{Ro1}, and Theorem \ref{thm:qc}, of 
\cite[Cor.~, p.~434]{Ro1}; Lemma \ref{lem:exi} is 
\cite[Lem.~1, p.~437]{Ro1}.

Corollary \ref{cor:poin} is called \textit{Poincar\'e's complete
reducibility theorem}; it is proved e.g. in \cite[p.~173]{Mum}
over an algebraically closed field, and in 
\cite[Cor.~3.20]{Co2} over an arbitrary field.
It implies that every abelian variety is isogenous to a product
of simple abelian varieties, and these are uniquely determined
up to isogeny and reordering. 
 
Example \ref{ex:ray} develops a construction of Raynaud, 
see \cite[XVII.C.5]{SGA3}. The proof of Lemma \ref{lem:ray}
is taken from \cite[9.2 Thm.~1]{BLR}. Corollary \ref{cor:ray} 
is due to Raynaud, see \cite[IX.2.7]{Ra}.

\section{Some further developments}
\label{sec:sfd}

\subsection{The Rosenlicht decomposition}
\label{subsec:trd}

Throughout this section, $G$ denotes a smooth connected 
algebraic group. By Theorem \ref{thm:che}, $G$ has a smallest 
connected normal affine subgroup scheme $G_\aff$ with quotient 
being an abelian variety; also, recall that $G_\aff$ is the kernel 
of the Albanese homomorphism $\alpha : G \to \Alb(G)$. 
On the other hand, by Theorem \ref{thm:affi}, every algebraic group
$H$ has a largest anti-affine subgroup scheme that we will denote 
by $H_\ant$; moreover, $H_\ant$ is smooth, connected, and contained
in the center of $H^0$. Also, $H_\ant$ is the smallest normal subgroup 
scheme of $H$ having an affine quotient. 

We will analyze the structure of $G$ in terms of those of $G_\aff$ and 
$G_\ant$. Note that $(G_\aff)_\ant$ is trivial (since $G_\aff$
is affine), but $(G_\ant)_\aff$ may have positive dimension as
we will see in \S \ref{subsec:saag}.

\begin{theorem}\label{thm:ros}
Keep the above notation and assumptions.

\begin{enumerate}

\item $G = G_\aff \cdot G_\ant$.

\item $G_\aff \cap G_\ant$ contains $(G_\ant)_\aff$.

\item The quotient $(G_\aff \cap G_\ant)/(G_\ant)_\aff$ is finite.

\end{enumerate}

\end{theorem}

\begin{proof}
(1) By Proposition \ref{prop:prod}, $G_\aff \cdot G_\ant$ is a
normal subgroup scheme of $G$. Moreover, the quotient
$G \to G/G_\aff \cdot G_\ant$ factors through a homomorphism
$G/G_\aff \to G/G_\aff \cdot G_\ant$
and also through a homomorphism 
$G/G_\ant \to G/G_\aff \cdot G_\ant$, in view of Proposition
\ref{prop:quot}. In particular, $G/G_\aff \cdot G_\ant$
is a quotient of an abelian variety, and hence is
an abelian variety as well. Also, $G/G_\aff \cdot G_\ant$ 
is a quotient of an affine algebraic group, and hence 
is affine as well (Proposition \ref{prop:tower}). Thus, 
$G/G_\aff \cdot G_\ant$ is trivial; this proves the assertion.

(2) Proposition \ref{prop:prod} yields an isomorphism
$G_\ant/G_\ant \cap G_\aff \cong G/G_\aff$. In particular,
$G_\ant/G_\ant \cap G_\aff$ is an abelian variety.
Since $G_\ant \cap G_\aff$ is affine, the assertion follows
from Theorem \ref{thm:che}.

(3) Consider the quotient $\bar{G} := G/(G_\ant)_\aff$. 
Then $\bar{G} = \bar{G}_\aff \cdot \bar{G}_\ant$, where
$\bar{G}_\aff := G_\aff/(G_\ant)_\aff$ and 
$\bar{G}_\ant := G_\ant/(G_\ant)_\aff$. Moreover,
$\bar{G}_\aff$ is affine (as a quotient of $G_\aff$)
and $\bar{G}_\ant$ is an abelian variety.
Thus, $\bar{G}_\ant \cap \bar{G}_\ant$ is finite. 
This yields the assertion in view of Proposition 
\ref{prop:sub}.
\end{proof}

\begin{remark}\label{rem:ros}
The above theorem is equivalent to the assertion that 
\textit{the multiplication map of $G$ induces an isogeny}
\[ (G_\aff \times G_\ant)/(G_\ant)_\aff \longrightarrow G, \]
where $(G_\ant)_\aff$ is viewed as a subgroup scheme of
$G_\aff \times G_\ant$ via $x \mapsto (x,x^{-1})$. 

Theorem \ref{thm:ros} may also be reformulated 
in terms of the two Albanese varieties 
$\Alb(G) = G/G_\aff$ and $\Alb(G_\ant) = G_\ant/(G_\ant)_\aff$.  
Namely, the inclusion $\iota : G_\ant \to G$ yields a homomorphism
$\Alb(\iota) : \Alb(G_\ant) \to \Alb(G)$ with kernel isomorphic to
$(G_\aff \cap G_\ant)/(G_\ant)_\aff$. The assertion (1) is equivalent
to the surjectivity of $\Alb(\iota)$, and (3) amounts to the finiteness
of its kernel. So Theorem \ref{thm:ros} just means that
$\Alb(\iota)$ \textit{is an isogeny}.
\end{remark}

\begin{proposition}\label{prop:small}
With the above notation and assumptions, $G_\ant$ is 
the smallest subgroup scheme $H \subseteq G$ such that 
$G = H \cdot G_\aff$.
\end{proposition}

\begin{proof}
Let $H \subseteq G$ be a subgroup scheme. Then 
$H_\ant \cdot G_\aff \subseteq G$ is a normal subgroup scheme, 
since $H_\ant$ is central in $G$; moreover, 
$G/H_\ant \cdot G_\aff$ is a quotient of $G/G_\aff$, hence
an abelian variety. If $G = H \cdot G_\aff$, then 
\[ G/H_\ant \cdot G_\aff \cong H/H \cap (H_\ant \cdot G_\aff). \]
Moreover, $H \cap (H_\ant \cdot G_\aff)$ is a normal subgroup
scheme of $H$ containing $H_\ant$; thus, the quotient 
$H/H \cap (H_\ant \cdot G_\aff)$ is affine. So $G/H_\ant \cdot G_\aff$
is trivial, i.e., $G = H_\ant \cdot G_\aff$. 
By Proposition \ref{prop:quot}, it follows that 
$G/H_\ant \cong G_\aff/G_\aff \cap H_\ant$.
The right-hand side is the quotient of an affine algebraic 
group by a normal subgroup scheme, and hence is affine. 
By Theorem \ref{thm:affi}, it follows that $H_\ant \supseteq G_\ant$.
\end{proof}

\begin{remark}\label{rem:ns}
By the above proposition, the extension
\[ 1 \longrightarrow G_\aff \longrightarrow G \longrightarrow 
\Alb(G) \longrightarrow 1 \]
is split if and only if $G_\aff \cap G_\ant$ is trivial.
But this fails in general; in fact, $G_\aff \cap G_\ant$
is generally of positive dimension, and hence the above
extension does not split after pull-back by any isogeny
(see Remark \ref{rem:nonsplit} below for specific examples). 
\end{remark}

We now present two applications of Theorem \ref{thm:ros}; 
first, to the derived subgroup $\cD(G)$. Recall from 
\cite[II.5.4.8]{DG} (see also \cite[VIB.7.8]{SGA3}) that $\cD(G)$ 
is the subgroup functor of $G$ that assigns to any scheme $S$, 
the set of those $g \in G(S)$ such that $g$ lies in 
the commutator subgroup of $G(S')$ for some scheme $S'$, 
faithfully flat and of finite presentation over $S$. 
Moreover, the group functor $\cD(G)$ is represented by 
a smooth connected subgroup scheme of $G$, and 
$\cD(G)(\bar{k})$ is the commutator subgroup of $G(\bar{k})$.

\begin{corollary}\label{cor:ros}
With the above notation and assumptions, we have 
$\cD(G) = \cD(G_\aff)$. In particular, $\cD(G)$ is affine.
Also, $G$ is commutative if and only if $G_\aff$ is 
commutative.
\end{corollary}

Our second application of Theorem \ref{thm:ros} characterizes
Lie algebras of algebraic groups in characteristic $0$:

\begin{corollary}\label{cor:lie}
Assume that $\car(k) = 0$ and consider a finite-dimensional
Lie algebra $\fg$ over $k$, with center $\fz$. Then the
following conditions are equivalent:

\begin{enumerate}

\item $\fg = \Lie(G)$ for some algebraic group $G$.

\item $\fg = \Lie(L)$ for some \textit{linear} algebraic group $L$.

\item $\fg/\fz$ (viewed as a Lie subalgebra of $\Lie(\GL(\fg))$ 
via the adjoint representation) is the Lie algebra of some algebraic
subgroup of $\GL(\fg)$. 

\end{enumerate}

\end{corollary}

\begin{proof}
Since (2)$\Leftrightarrow$(3) follows from \cite[V.5.3]{Ch1}
and (2)$\Rightarrow$(1) is obvious, it suffices to show
that (1)$\Rightarrow$(2). 

Let $G$ be an algebraic group such that $\fg = \Lie(G)$. 
We may assume that $G$ is connected. Then 
$G = G_\aff \cdot G_\ant = G_\aff \cdot Z$ and hence
$\fg = \fg_\aff + \fz$ with an obvious notation. Thus, we have
a decomposition of Lie algebras $\fg = \fg_\aff \oplus \fa$ 
for some linear subspace $\fa \subseteq \fz$, viewed as an
abelian Lie algebra. Let $n := \dim(\fa)$, then 
$L := G_\aff \times \bG_a^n$ is a connected linear algebraic
group with Lie algebra isomorphic to $\fg$.
\end{proof}

The Lie algebras satisfying the condition (2) above are called
\textit{algebraic}.

\subsection{Equivariant compactification of homogeneous spaces}
\label{subsec:echs}

\begin{definition}\label{def:equ}
Let $G$ be an algebraic group and $H \subseteq G$ a subgroup scheme.
An \textit{equivariant compactification} of the homogeneous space
$G/H$ is a proper $G$-scheme $X$ equipped with an open equivariant
immersion $G/H \to X$ with schematically dense image.
\end{definition}

Equivalently, $X$ is a $G$-scheme equipped with a base point
$x \in X(k)$ such that the $G$-orbit of $x$ is schematically 
dense in $X$ and the stabilizer $\C_G(x)$ equals $H$.

\begin{theorem}\label{thm:equiv}
Let $G$ be an algebraic group and $H \subseteq G$ a subgroup scheme.
Then $G/H$ has an equivariant compactification by a projective 
scheme.
\end{theorem}
 
\begin{proof}
If $G$ is affine, then $H$ is the stabilizer of a line $\ell$ in 
some finite-dimensional $G$-module $V$ (see e.g. \cite[II.2.3.5]{DG}). 
Then the closure of the $G$-orbit of $\ell$ 
in the projective space of lines of $V$ yields the desired 
projective equivariant compactification $X$, in view of Proposition 
\ref{prop:orbit}. Note that $X$ is equipped with an ample
$G$-linearized invertible sheaf in the sense of 
\cite[I.3 Def.~1.6]{MFK}.

If $G$ is proper, then the homogeneous space $G/H$ is proper
as well. So it suffices to check that $G/H$ is projective.
For this, we may assume $k$ algebraically closed by using
\cite[IV.9.1.5]{EGA}. Note that 
$H \subseteq H  \cdot G^0 \subseteq G$ and the quotient scheme
$G/H \cdot G^0$ is finite and \'etale, hence constant. 
So the scheme $G/H$ is a finite disjoint union of copies of 
$H \cdot G^0/G^0 \cong G^0/G^0 \cap H$. Thus, we may assume
that $G$ is connected. If $\car(k) = 0$, then $G$, and hence
$G/H$, are abelian varieties and thus projective. 
If $\car(k) > 0$, then Proposition \ref{prop:frob} yields
that $G/\Ker(F^n_{G/k})$ is an abelian variety for $n \gg 0$;
thus, $G/H \cdot \Ker(F^n_{G/k})$ is an abelian variety as
well, and hence is projective. Moreover, the natural map
$f : G/H \to G/H \cdot \Ker(F^n_{G/k})$ is the quotient by
the action of the infinitesimal group scheme $\Ker(F^n_{G/k})$,
and hence is finite. It follows that $G/H$ is projective.

For an arbitrary algebraic group $G$, Theorem \ref{thm:che} 
yields an affine normal subgroup scheme $N \subseteq G$ such that 
$G/N$ is proper. Then $H \cdot N$ is a subgroup scheme of $G$ and 
$G/H \cdot N \cong (G/N)/(H \cdot N/N)$ is proper as well, 
hence projective. 

We now claim that it suffices to show the existence of a projective 
$H \cdot N$-equivariant compactification $Y$ of $H \cdot N/H$ 
having an $H \cdot N$-linearized ample line bundle. Indeed, 
by \cite[Prop.~7.1]{MFK} applied to the projection 
$p_1: G \times Y \to G$ and to the $H \cdot N$-torsor 
$q : G \to G/H \cdot N$, there exists a unique cartesian square
\[
\xymatrix{
G \times Y\ar[r]^-{p_1} \ar[d]_{r} & G \ar[d]^{q} \\ 
X \ar[r]^-{f} & G/H \cdot N,  \\}
\]
where $X$ is a $G$-scheme, $r$ and $f$ are $G$-equivariant,
and $r$ is an $H \cdot N$ torsor for the action defined by
$x  \cdot (g,y) := (g x^{-1}, xy)$, where $x \in H \cdot N$,
$g \in G$ and $y \in Y$. Moreover, $f$ is projective, and hence
so is $X$. (We may view $f : X \to G/H \cdot N$ as the homogeneous
fiber bundle with fiber $Y$ associated with the principal bundle
$q : G \to G/H \cdot N$ and the $H \cdot N$-scheme $Y$). Also, 
as in the proof of Proposition \ref{prop:quot}, we have 
a cartesian square 
\[ 
\xymatrix{
G \times H \cdot N/H \ar[r]^-{p_1} \ar[d]_{n} & G \ar[d]^{q} \\ 
G/H \ar[r]^-{f} & G/H \cdot N,  \\}
\]
where $n$ is obtained from the multiplication map 
$G \times H \cdot N \to G$ (indeed, this square is commutative 
and the horizontal arrows are $H \cdot N/H$-torsors). Since $Y$ 
is an equivariant compactification of $H \cdot N/H$, it follows 
that $X$ is the desired equivariant compactification of~$G/H$.
This proves the claim.

In view of the first step of the proof (the case of an affine
group $G$), it suffices in turn to check that $H \cdot N$ acts 
on $H \cdot N/H$ via an affine quotient group.

By Lemma \ref{lem:quotient}, $(H \cdot N)_{\ant}$
is a quotient of $(H \ltimes N)_{\ant}$. The latter is the fiber
at the neutral element of the affinization morphism
$H \ltimes N \to \Spec \,\cO(H \ltimes N)$. Also, recall that
$H \ltimes N \cong H \times N$ as schemes, $N$ is affine,
and the affinization morphism commutes with finite products; 
thus, $(H \ltimes N)_{\ant} = H_{\ant}$. As a consequence, 
we have $(H \cdot N)_{\ant} = H_{\ant}$. Since
$H \cdot N/H \cong (H \cdot N/H_\ant)/(H/H_\ant)$,
and $H \cdot N/H_\ant = H \cdot N/(H \cdot N)_\ant$ is affine, 
this completes the proof.
\end{proof}

\begin{remarks}\label{rem:qp}
(i) With the notation and assumptions of the above theorem,
the homogeneous space $G/H$ is quasi-projective. In particular, 
every algebraic group is quasi-projective.

(ii) Assume in addition that $G$ is smooth. Then $G/H$ has 
an equivariant compactification by a \textit{normal} 
projective scheme, as follows from Proposition \ref{prop:norm}.

(iii) If $\car(k) = 0$ then every homogeneous space has 
a \textit{smooth} projective equivariant compactification.
This follows indeed from the existence of an equivariant resolution 
of singularities; see \cite[Prop.~3.9.1, Thm.~3.36]{Ko2}. 

(iv) Over any imperfect field $k$, there exist smooth connected
algebraic groups $G$ having no smooth compactification
(equivariant or not).
For example, choose $a \in k \setminus k^p$, where $p := \car(k)$,
and consider the subgroup scheme $G \subseteq \bG_a^2$
defined as the kernel of the homomorphism
\[ \bG_a^2 \longrightarrow \bG_a, 
\quad (x,y) \longmapsto y^p - x - a x^p. \]
Then $G_{k_i}$ is the kernel of the homomorphism
$(x,y) \mapsto (y - a^{1/p} x)^p -x$; thus,
$G_{k_i} \cong (\bG_a)_{k_i}$ via the map 
$(x,y) \mapsto y - a^{1/p} x$. As a consequence, $G$ is smooth,
connected and unipotent. Moreover, $G$ is equipped with
a compactification $X$, the zero subscheme of 
$y^p - xz^{p-1} - a x^p$ in $\bP^2$; the complement 
$X \setminus G$ consists of a unique $k_i$-rational point
$P$ with homogeneous coordinates $(1,a^{1/p},0)$. One may
check that $P$ is a regular, non-smooth point of $X$.
Since $G$ is a regular curve, it follows that $X$ is its 
unique regular compactification. 
\end{remarks}

\subsection{Commutative algebraic groups}
\label{subsec:cag}

This section gives a brief survey of some structure results 
for the groups of the title. We will see how to build them from
two classes: the (commutative) unipotent groups and the groups
of multiplicative type, i.e., those algebraic groups $M$ such 
that $M_{\bar{k}}$ is isomorphic to a subgroup scheme of some
torus $\bG_{m,\bar{k}}^n$. Both classes are stable under 
taking subgroup schemes, quotients, commutative group extensions, 
and extensions of the base field (see 
\cite[IV.2.2.3, IV.2.2.6]{DG} for the unipotent groups,
and \cite[IV.1.2.4, IV.1.4.5]{DG} for those of multiplicative
type). Also, every homomorphism between groups of different 
classes is constant (see \cite[IV.2.2.4]{DG}).

The structure of commutative unipotent algebraic groups is 
very simple if $\car(k) = 0$: every such group $G$ is isomorphic 
to its Lie algebra via the exponential map (see \cite[IV.2.4]{DG}).
In particular, $G$ is a \textit{vector group}, i.e., the additive
group of a finite-dimensional vector space, uniquely determined
up to isomorphism of vector spaces. In characteristic $p > 0$,
the commutative unipotent groups are much more involved 
(see \cite[V.1]{DG} for structure results over a perfect field);
we just recall that for any such group $G$, the multiplication
map $p^n_G$ is zero for $n \gg 0$.

Next, we consider algebraic groups of multiplicative type. 
Any such group $M$ is uniquely determined by its
\textit{character group} $X^*(M)$, consisting of the homomorphisms 
of $k_s$-group schemes $G_{k_s} \to (\bG_m)_{k_s}$. Also, $X^*(M)$ 
is a finitely generated abelian group equipped with an action
of the Galois group $\Gamma$. Moreover, the assignment
$M \mapsto X^*(M)$ yields an anti-equivalence from the category
of algebraic groups of multiplicative type (and homomorphisms)
to that of finitely generated abelian groups equipped with a 
$\Gamma$-action (and $\Gamma$-equivariant homomorphisms).
Also, $M$ is  a torus (resp. smooth) if and only if the group 
$X^*(M)$ is torsion-free (resp.~$p$-torsion free if 
$\car(k) = p > 0$). 

It follows readily that every algebraic group of 
multiplicative type $M$ has a largest subtorus, namely, 
$M^0_\red$; its character group is the quotient of $X^*(M)$ by
the torsion subgroup. Moreover, the formation of the largest 
subtorus commutes with arbitrary field extensions (see 
\cite[IV.1.3]{DG} for these results).

\begin{theorem}\label{thm:cag}
Let $G$ be a commutative affine algebraic group.

\begin{enumerate}

\item $G$ has a largest subgroup scheme of multiplicative 
type $M$, and the quotient $G/M$ is unipotent.

\item When $k$ is perfect, $G$ also has a largest unipotent
subgroup scheme $U$, and $G = M \times U$.

\item Returning to an arbitrary field $k$, there exists 
a subgroup scheme $H \subseteq G$ such that $G = M \cdot H$ 
and $M \cap H$ is finite.

\item Any exact sequence of algebraic groups
\[ 1 \longrightarrow G_1 \longrightarrow G 
\longrightarrow G_2 \longrightarrow 1 \]
induces exact sequences
\[ 1 \to M_1 \to M \to M_2 \to 1, \quad 
1 \to U_1 \to U \to U_2 \to 1 \]
with an obvious notation.

\end{enumerate}

\end{theorem}

\begin{proof}
The assertions (1) and (2) are proved in \cite[IV.3.1.1]{DG} and 
\cite[XVII.7.2.1]{SGA3}.

(3) We argue as in the proof of Theorem \ref{thm:cheva}.
By (2), we have $G_{k_i} = M_{k_i} \times U_i$ for a unique
unipotent subgroup scheme $U_i \subseteq G_{k_i}$. Thus, 
$U_i$ is defined over some subfield $K \subseteq k_i$, 
finite over $k$, and $G_K = M_K \times U'$
with an obvious notation. By Lemma \ref{lem:ray}, there
exists a subgroup scheme $H \subseteq G$ such that 
$H_K \supseteq U'$ and $H_K/U'$ is finite. Then
$G_K = M_K \cdot H_K$ and $M_K \cap H_K$ is
finite. This yields the assertion.

(4) We have $M_1 = M \cap G_1$ by construction. It follows
that the quotient map $q : G \to G_2$ induces a closed 
immersion of group schemes
$\iota : M/M_1 \to M_2$. Let $N$ be the scheme-theoretic
image of $\iota$; then $N$ is of multiplicative type and 
$G_2/N$ is a quotient of $G/M$, hence is unipotent. 
Thus, $N = M_2$; this yields the first exact sequence, 
and in turn the second one.
\end{proof}

The assertion (2) of the above proposition fails over any imperfect 
field, as shown by the following variant of Example \ref{ex:ray}: 

\begin{example}\label{ex:cag}
Let $k$ be an imperfect field and choose a non-trivial finite 
purely inseparable field extension $K$ of $k$. Then
$G := \R_{K/k}(\bG_{m,K})$ is a smooth connected commutative
algebraic group of dimension $[K:k]$. Moreover, there is 
a canonical exact sequence
\begin{equation}\label{eqn:cag}
1 \longrightarrow U' \longrightarrow G_K 
\stackrel{q}{\longrightarrow} \bG_{m,K} \longrightarrow 1, 
\end{equation}
where $U'$ is a non-trivial smooth connected unipotent group
(see \cite[A.5.11]{CGP}). Also, we have a closed immersion
of algebraic groups
\[ j : \bG_{m,k} \longrightarrow \R_{K/k}(\bG_{m,K}) = G \]
in view of \cite[A.5.7]{CGP}. Thus,
$j_K : \bG_{m,K} \to G_K$ a closed immersion of algebraic 
groups as well, which yields a splitting of the exact sequence 
(\ref{eqn:cag}). Moreover, we have an exact sequence
\begin{equation}\label{eqn:ns} 
1 \longrightarrow \bG_{m,k} \stackrel{j}{\longrightarrow} 
G \longrightarrow U \longrightarrow 1, 
\end{equation}
where $U$ is smooth, connected and unipotent (since 
$U_K$ is isomorphic to $U'$). 

We claim that (\ref{eqn:ns}) is not split. Indeed,
it suffices to show that every homomorphism $f : H \to G$ is
constant, where $H$ is a smooth connected unipotent group.
But as in Example \ref{ex:ray}, these homomorphisms
correspond bijectively to the homomorphisms 
$f' : H_K \to \bG_{m,K}$ via the assignment 
$f \mapsto f' := q \circ f$. Moreover, every such homomorphism
is constant; this completes the proof of the claim.
\end{example}

Next, consider a smooth connected commutative algebraic group 
(not necessarily affine) over a perfect field $k$. 
By combining Theorems \ref{thm:chev} and \ref{thm:cag}, 
we obtain an exact sequence
\begin{equation}\label{eqn:com}  
1 \longrightarrow T \times U \longrightarrow G 
\stackrel{q}{\longrightarrow} A \longrightarrow 1,
\end{equation}
where $T \subseteq G$ is the largest subtorus, and $U \subseteq G$ 
the largest smooth connected unipotent subgroup scheme. 
This yields in turn two exact sequences
\[  1 \longrightarrow T \longrightarrow G/U \longrightarrow A 
\longrightarrow 1, \quad 
1 \longrightarrow U \longrightarrow G/T \longrightarrow A 
\longrightarrow 1 \]
and a homomorphism 
\begin{equation}\label{eqn:pf} 
f : G \longrightarrow G/U \times_A G/T 
\end{equation}
which is readily seen to be an isomorphism.

Thus, the structure of $G$ is reduced to those of $G/U$
and $G/T$. The former will be described in the next section.
We now describe the latter under the assumption that 
$\car(k) = 0$; we then consider extensions
\begin{equation}\label{eqn:vec}
1 \longrightarrow U \longrightarrow G \longrightarrow A
\longrightarrow 1,
\end{equation}
where $U$ is unipotent and $A$ is a prescribed abelian variety.
As $U$ is a vector group, such an extension is called
a \textit{vector extension} of $A$.  

When $U \cong \bG_a$, every extension (\ref{eqn:vec})
yields a $\bG_a$-torsor over $A$; this defines a map
\[ \Ext^1(A,\bG_a) \to H^1(A,\cO_A) \]
which is in fact an isomorphism (see e.g. 
\cite[III.17]{Oo}). For an arbitrary
vector group $U$, we obtain an isomorphism
\[ \Ext^1(A,U) \cong H^1(A,\cO_A) \otimes_k U. \]
Viewing the right-hand side as 
$\Hom_\gpsch( H^1(A,\cO_A)^*, U)$,
it follows that there is a \textit{universal vector extension}
\begin{equation}\label{eqn:uni} 
1 \longrightarrow H^1(A,\cO_A)^* \longrightarrow
E(A) \longrightarrow A \longrightarrow 1,
\end{equation}
from which every extension (\ref{eqn:vec}) is obtained
via push-out by a unique linear map $H^1(A,\cO_A)^* \to U$.

\begin{remarks}\label{rem:cag}
(i) When $\car(k) > 0$, every abelian variety $A$ still 
has a universal vector extension $E(A)$, as shown by 
the above arguments. Yet note that $E(A)$ classifies extensions 
of $A$ by vector groups, which form a very special class 
of smooth connected unipotent commutative groups in this setting.

(ii) Little is known on the structure of commutative 
algebraic groups over imperfect fields, due to the failure 
of Chevalley's structure theorem. We will present a partial
remedy to that failure in Corollary \ref{cor:pa}, as a consequence
of a structure result for commutative algebraic groups in positive
characteristics (Theorem \ref{thm:cagc}).
\end{remarks}

\subsection{Semi-abelian varieties}
\label{subsec:sav}

\begin{definition}\label{def:com}
A \textit{semi-abelian variety} is an algebraic group 
obtained as an extension
\begin{equation}\label{eqn:sab}
1 \longrightarrow T \longrightarrow G 
\stackrel{q}{\longrightarrow} A \longrightarrow 1,
\end{equation}
where $T$ is a torus and $A$ an abelian variety.
\end{definition}

\begin{remarks}\label{rem:sab}
(i) With the above notation, $q$ is the Albanese homomorphism. 
It follows that $T$ and $A$ are uniquely determined by $G$.

(ii) Every semi-abelian variety is smooth and connected;
it is also commutative in view of Corollary \ref{cor:ros}. 
Moreover, the multiplication map $n_G$ is an isogeny for any 
non-zero integer $n$, since this holds for abelian varieties 
(by Lemma \ref{lem:mult}) and for tori.

(iii) Given a semi-abelian variety $G$ over $k$ and an extension
of fields $K$ of $k$, the $K$-algebraic group $G_K$
is a semi-abelian variety. 
\end{remarks}

In the opposite direction, we will need:

\begin{lemma}\label{lem:sabext}
Let $G$ be an algebraic group. If $G_{\bar{k}}$ is a semi-abelian
variety, then $G$ is a semi-abelian variety.
\end{lemma} 

\begin{proof}
Clearly, $G$ is smooth, connected and commutative. Also, 
arguing as at the end of the proof of Theorem \ref{thm:cheva},
we obtain a finite extension $K$ of $k$ 
and an exact sequence of algebraic groups over $K$   
\[ 1 \longrightarrow T' \longrightarrow G_K 
\stackrel{q'}{\longrightarrow} A' \longrightarrow 1,\]
where $T'$ is a torus and $A'$ an abelian variety;
here $q'$ is the Albanese homomorphism.
On the other hand, Theorem \ref{thm:cheva} yields an exact sequence
\[ 1 \longrightarrow N \longrightarrow G 
\stackrel{q}{\longrightarrow} A \longrightarrow 1,\]
where $N$ is connected and affine, and $A$ is an abelian variety;
also, $q$ is the Albanese homomorphism. By Remark \ref{rem:ext},
$N_K$ contains $T'$ and the quotient $N_K/T'$ is infinitesimal;
as a consequence, $T'$ is the largest subtorus of $N_K$.
Consider the largest subgroup scheme $M \subseteq N$ of multiplicative
type; then $N/M$ is unipotent and connected. Since the formation
of $M$ commutes with field extensions, we have
$(N/M)_K = N_K/M_K$. The latter is a quotient of $N_K/T'$;
hence $N_K/M_K$ is infinitesimal, and so is $N/M$. 
Let $T := M^0_\red$; this is the largest subtorus of $N$
and $M/T$ is infinitesimal, hence so is $N/T$. It follows that
$G/T=: A$ is proper; since $G$ is smooth and connected, 
$A$ is an abelian variety. Thus, $G$ is a semi-abelian variety;
moreover, $T_K = T'$ and $A_K = A'$ in view of Remark
\ref{rem:sab} (i).
\end{proof}

\begin{remark}\label{rem:sav}
More generally, if $G_K$ is a semi-abelian variety for some field
extension $K$ of $k$, then $G$ is a semi-abelian variety as well.
This may be checked as in Remarks \ref{rem:aag} and \ref{rem:av},
by adapting the proof of \cite[1.1.9]{CGP}.
\end{remark}

Next, we obtain a geometric characterization of semi-abelian
varieties:

\begin{proposition}\label{prop:line}
Let $G$ be a smooth connected algebraic group over a perfect
field $k$. Then $G$ is a semi-abelian variety if and only
if every morphism (of schemes) $f : \bA^1 \to G$ is constant.
\end{proposition}

\begin{proof}
Assume that $G$ is semi-abelian and consider the exact
sequence (\ref{eqn:sab}). Then 
$q \circ f : \bA^1 \to A$ is constant by Lemma \ref{lem:rat}.
Thus, there exists $g \in G(\bar{k})$ such that $f_{\bar{k}}$
factors through the translate $g T_{\bar{k}} \subseteq G_{\bar{k}}$.
So we may view $f_{\bar{k}}$ as a morphism 
$\bA^1_{\bar{k}} \to (\bA^1_{\bar{k}} \setminus 0)^n$ for some 
positive integer $n$. Since every morphism 
$\bA^1_{\bar{k}} \to \bA^1_{\bar{k}} \setminus 0$ is constant, 
we see that $f$ is constant.

Conversely, assume that every morphism $\bA^1 \to G$ is constant. 
By \cite[14.3.10]{Sp}, it follows that every smooth connected unipotent 
subgroup of $G$ is trivial. In particular, the unipotent radical of 
$G_\aff$ is trivial, i.e., $G_\aff$ is reductive. Also, $G_\aff$ has no 
root subgroups and hence is a torus. So $G$ is a semi-abelian variety. 
\end{proof}

As a consequence, semi-abelian varieties are stable under
group extensions in a strong sense:

\begin{corollary}\label{cor:extsab}

Let $G$ be a smooth connected algebraic group and $N \subseteq G$
a subgroup scheme. 

\begin{enumerate}

\item If $G$ is a semi-abelian variety, then so is $G/N$. 
If in addition $N$ is smooth and connected, 
then $N$ is a semi-abelian variety as well. 

\item If $N$ and $G/N$ are semi-abelian varieties, then so is $G$.

\item Let $f : G \to H$ be an isogeny, where $H$ is a semi-abelian 
variety. Then $G$ is a semi-abelian variety.

\end{enumerate}

\end{corollary}

\begin{proof}
We may assume $k$ algebraically closed by Lemma \ref{lem:sabext}.

(1) We have a commutative diagram of extensions
\[
\xymatrix{
1 \ar[r] & T \cap N \ar[r] \ar[d] & N \ar[r] \ar[d] &
B \ar[r] \ar[d] & 1 \\
1 \ar[r] & T \ar[r] & G \ar[r] & A \ar[r] & 1, \\}
\]
where $T$ is a torus, $A$ an abelian variety, and $B$ a subgroup 
scheme of $A$. This yields an exact sequence
\[ 1 \longrightarrow T/T \cap N \longrightarrow G/N
\longrightarrow A/B \longrightarrow 1, \]
where $T/T \cap N$ is a torus and $A/B$ is an abelian variety.
Thus, $G/N$ is a semi-abelian variety. The assertion on $N$
follows from Proposition \ref{prop:line}.

(2) Denote by $q: G \to G/N$ the quotient morphism. Let 
$f : \bA^1 \to G$ be a morphism; then $q \circ f$ is constant 
by Proposition \ref{prop:line} again. Translating by some 
$g \in G(k)$, we may thus assume that $f$ factors through $N$. 
Then $f$ is constant; this yields the assertion.

(3) Consider the Albanese homomorphism $\alpha_H : H \to \Alb(H)$
and the kernel $N$ of the composition $\alpha_H \circ f$. Then 
$N$ is an extension of $\Ker(\alpha_H)$ (a torus) by $\Ker(f)$ 
(a finite group scheme). As a consequence, $N$ is affine and
$N/T$ is finite, where $T \subseteq N$ denotes the largest 
subtorus. Since $G/N \cong \Alb(H)$ is an abelian variety, it
follows that $G/T$ is an abelian variety as well. Thus,
$G$ is a semi-abelian variety.
\end{proof}

\begin{remarks}\label{rem:line}
(i) Proposition \ref{prop:line} fails over any imperfect field $k$,
Indeed, as in Remark \ref{rem:qp} (iv), consider the subgroup scheme 
$G \subset \bG_a^2$ defined by $y^p = x + a x^p$,
where $p := \car(k)$ and $a \in k \setminus k^p$. Then 
every morphism $f : \bA^1 \to G$ is just given by 
$x(t),y(t) \in k[t]$ such that  
\[ y(t)^p =  x(t) + a \, x(t)^p. \] 
Thus, $x(t) = z(t)^p$ for a unique $z(t) \in k^{1/p}[t]$ such that 
\[ y(t) = z(t) + a^{1/p} \, z(t)^p. \]
Consider the monomial of highest degree in $z(t)$, say
$a_n t^n$. If $n \geq 1$, then the monomial of highest degree
in $y(t)$ is $a^{1/p} \, a_n^p \, t^{np}$. Since $a^{1/p} \notin k$ 
and $a_n^p \in k$, this contradicts the fact that
$y(t) \in k[t]$. So $n = 0$,
i.e., $z(t)$ is constant; then so are $x(t)$ and $y(t)$.

(ii) Consider a semi-abelian variety $G$ and 
a semi-abelian subvariety $H \subseteq G$. Then the induced
homomorphism between Albanese varieties, $\Alb(H) \to \Alb(G)$,
has a finite kernel that may be arbitrary large. For example,
let $H$ be a non-trivial abelian variety over an algebraically 
closed field; then $H$ contains a copy of the constant group 
scheme $\bZ/\ell$ for any prime number $\ell \neq \car(k)$ 
(see \cite[p.~64]{Mum}). Choosing a root of unity of order 
$\ell$ in $k$, we obtain a closed immersion 
$j : \bZ/\ell \to \bG_m$ and, in turn, a commutative diagram 
of extensions
\[
\xymatrix{
1 \ar[r] & \bZ/\ell \ar[r] \ar[d]_{j} & H \ar[r] \ar[d] &
A \ar[r] \ar[d]_{\id} & 1 \\
1 \ar[r] & \bG_m \ar[r] & G \ar[r] & A \ar[r] & 1, \\}
\]
where $A$ is an abelian variety. So $G$ is a semi-abelian variety 
containing $H$, and the kernel of the homomorphism 
$H = \Alb(H) \to \Alb(G) = A$ is $\bZ/\ell$. 

(iii) Consider again a semi-abelian variety $G$ and let
$H \subseteq G$ be a subgroup scheme. Then $H^0_\red$ is
a semi-abelian variety. To see this, we may replace
$G$, $H$ with $G/H_\ant$, $H/H_\ant$, and hence assume that
$H$ is affine. Let $T$ be the largest torus of $G$, then
$H/T \cap H$ is affine and isomorphic to a subgroup scheme 
of the abelian variety $G/T$. Thus, $H/T \cap H$ is finite;
it follows that $(H/T \cap H)^0_\red$ is trivial, and hence
$H^0_\red \subseteq T$. So $H^0_\red = (T \cap H)^0_\red$
is a torus. 
\end{remarks}

Next, we extend Lemma \ref{lem:alb} to morphisms to 
semi-abelian varieties:

\begin{lemma}\label{lem:salb}
Let $X$, $Y$ be varieties equipped with $k$-rational points
$x_0$, $y_0$ and let $f : X \times Y \to G$ be a morphism to
a semi-abelian variety. Then we have identically
\[ f(x,y) \, f(x,y_0)^{-1} \, f(x_0,y)^{-1}  \, f(x_0,y_0) = e. \]
\end{lemma}

\begin{proof}
We may assume that $f(x_0,y_0) = e$. Let $q : G \to A$ denote
the Albanese homomorphism. By Lemma \ref{lem:alb},
we have identically
\[ (q \circ f)(x,y) = (q \circ f)(x,y_0) \, (q \circ f)(x_0,y). \]
Thus, the morphism
\[ \varphi : X \times Y \longrightarrow G, \quad
(x,y) \longmapsto f(x,y) \, f(x,y_0)^{-1} \, f(x_0,y)^{-1} \]
factors through the torus $T:= \Ker(q)$. Also, 
we have identically $\varphi(x,y_0) = \varphi(x_0,y) = e$.
We now show that $\varphi$ is constant. For this, we
may assume $k$ algebraically closed; then $T \cong \bG_m^n$
and accordingly, 
$\varphi = \varphi_1 \times \cdots \times \varphi_n$
for some $\varphi_i : X \times Y \to \bG_m$. Equivalently,
$\varphi \in \cO(X \times Y)^*$ (the unit group of the
algebra $\cO(X \times Y)$). By \cite[Lem.~2.1]{FI},
there exists $u_i \in \cO(X)^*$ and $v_i \in \cO(Y)^*$ 
such that $\varphi_i(x,y) = u_i(x) \,v_i(y)$ identically. 
As $\varphi_i(x,y_0) = \varphi_i(x_0,y) = 1$,
it follows that $\varphi_i = 1$; thus, $\varphi$ factors
through $e$.
\end{proof}

\begin{remarks}\label{rem:salb}
(i) In view of Corollary \ref{cor:extsab} and Lemma \ref{lem:salb},
\textit{every smooth connected algebraic group
$G$ admits a universal homomorphism to a semi-abelian variety,
which satisfies the assertions of Proposition \ref{prop:alba}.}
In particular, the kernel $N$ of this homomorphism is connected.
If $k$ is perfect, then one may check that 
$N = \cR_u(G_\aff) \cdot \cD(G_\aff)$, where $\cR_u$ denotes the
unipotent radical (and $\cD$ the derived subgroup); 
as a consequence, $N$ is smooth. This fails over an arbitrary 
field $k$, as shown by Example \ref{ex:ray} again (then
$N$ is the kernel of the Albanese homomorphism).

(ii) More generally,
\textit{every pointed variety admits a universal morphism
to a semi-abelian variety}, as follows from \cite[Thm.~7]{Se1}
over an algebraically closed field, and from \cite[Thm.~A1]{Wit} 
over an arbitrary field. This universal morphism, which gives back 
that of Theorem \ref{thm:alba}, is still called the Albanese 
morphism.
\end{remarks}

Finally, we discuss the structure of semi-abelian varieties
by adapting the approach to the classification of vector 
extensions (\S \ref{subsec:cag}).
Consider first the extensions
\begin{equation}\label{eqn:dual} 
1 \longrightarrow \bG_m \longrightarrow G 
\stackrel{q}{\longrightarrow} A 
\longrightarrow 1, 
\end{equation}
where $A$ is a prescribed abelian variety. Any such extension
yields a $\bG_m$-torsor over $A$, or equivalently a line bundle
over that variety. This defines a map
\[ \Ext^1(A, \bG_m) \longrightarrow \Pic(A). \] 
By the Weil-Barsotti formula (see e.g. \cite[III.17, III.18]{Oo}), 
this map is injective and its image is the subgroup 
of translation-invariant line bundles over $A$; this is the
group of $k$-rational points of the dual abelian variety
$\widehat{A}$. We may thus identify $\Ext^1(A,\bG_m)$ with 
$\widehat{A}(k)$. Also, using a Poincar\'e sheaf, 
we will view the points of $\widehat{A}(k)$ as the
algebraically trivial invertible sheaves $\cL$ on $A$ equipped
with a rigidification at $0$, i.e., an isomorphism 
$k \stackrel{\cong}{\to} \cL_0$.

Denoting by $\cL \in \widehat{A}(k)$ the invertible sheaf that
corresponds to the extension (\ref{eqn:dual}), we obtain
an isomorphism of sheaves of algebras over $A$
\[ 
q_*(\cO_G) \cong \bigoplus_{n = -\infty}^{\infty} \cL^{\otimes n}, 
\] 
since $G$ is the $\bG_m$-torsor over $A$ associated with $\cL$.
The multiplication in the right-hand side is given by the
natural isomorphisms 
\[ \cL^{\otimes m} \otimes_{\cO_A} \cL^{\otimes n} 
\stackrel{\cong}{\longrightarrow}
\cL^{\otimes (m + n)}. \]

For an arbitrary extension (\ref{eqn:sab}), consider the
associated extension 
\begin{equation}\label{eqn:sabsep} 
1 \longrightarrow T_{k_s} \longrightarrow G_{k_s} 
\stackrel{q_{k_s}}{\longrightarrow} A_{k_s} \longrightarrow 1 
\end{equation}
and denote by $\Lambda$ the character group of $T$. Then 
the split torus $T_{k_s}$ is canonically isomorphic to 
$\Hom_\gpsch(\Lambda,(\bG_m)_{k_s})$. Moreover, every 
$\lambda \in \Lambda$ yields an extension of the form 
(\ref{eqn:dual}) via pushout by
$\lambda : T_{k_s} \to  (\bG_m)_{k_s}$. This defines a map
\[
\gamma_{k_s} : \Ext^1(A_{k_s}, T_{k_s}) \to \Hom_\gp(\Lambda,\widehat{A}(k_s))
\]
which is readily seen to be an isomorphism by identifying
$T_{k_s}$ with $(\bG_m)_{k_s}^n$ and accordingly $\Lambda$
with $\bZ^n$. Likewise, we obtain an isomorphism of
sheaves of algebras over $A_{k_s}$:
\begin{equation}\label{eqn:sh}
q_*(\cO_G)_{k_s} \cong \bigoplus_{\lambda \in \Lambda} c(\lambda),
\end{equation}
where $c: \Lambda \to \widehat{A}(k_s)$ denotes the map classifying
the extension. Here again, the multiplication of the right-hand 
side is given by the natural isomorphisms
\[ c(\lambda) \otimes_{\cO_A} c(\mu) 
\stackrel{\cong}{\longrightarrow} c(\lambda + \mu). \]
By construction, $\gamma_{k_s}$ is equivariant for the natural actions 
of $\Gamma$. Composing $\gamma_{k_s}$ with the base change map
$\Ext^1(A,T) \to \Ext^1(A_{k_s}, T_{k_s})$ (which is $\Gamma$-invariant), 
we thus obtain a map
\begin{equation}\label{eqn:cla}
\gamma : \Ext^1(A, T) \to \Hom^\Gamma_\gp(\Lambda,\widehat{A}(k_s)).
\end{equation}

By Galois descent, this yields:

\begin{proposition}\label{prop:cla}
Let $A$ be an abelian variety and $T$ a torus. Then the
map (\ref{eqn:cla}) is an isomorphism.
\end{proposition}

\subsection{Structure of anti-affine groups}
\label{subsec:saag}

Let $G$ be an anti-affine group. Recall from Lemma 
\ref{lem:anti} and Proposition \ref{prop:rig} that 
$G$ is smooth, connected and commutative; also,
every quotient of $G$ is anti-affine.

\begin{proposition}\label{prop:aacarp}
When $\car(k) = p > 0$, every anti-affine group over $k$ 
is a semi-abelian variety.
\end{proposition}

\begin{proof}
By Lemma \ref{lem:sabext}, we may assume $k$ algebraically closed.
Using the isomorphism (\ref{eqn:pf}), we may further assume that
$G$ lies in an exact sequence
\[ 1 \longrightarrow U \longrightarrow G 
\stackrel{q}{\longrightarrow} A \longrightarrow 1, \]
where $U$ is unipotent and $A$ is an abelian variety. We may then
choose a positive integer $n$ such that the power map $p^n_U$ 
is zero. Then $p^n_G$ factors through a morphism 
$f : G/U = A \to G$. Since $f$ is surjective and the square
\[ 
\xymatrix{
G \ar[r]^-{p^n_G} \ar[d]_q & G \ar[d]^{q} \\ 
A \ar[r]^-{p^n_A} & A \\}
\]
commutes, the composition $q \circ f : A \to A$ equals $p^n_A$ 
and hence is an isogeny. Let $B$ be the image of $f$; then
$B$ is an abelian subvariety of $G$, isogenous to $A$ via $q$.
It follows that $G = U \cdot B$ and $U \cap B$ is finite.
Thus, $U/U \cap B \cong G/B$ is affine (as a quotient of $U$)
and anti-affine (as a quotient of $G$), hence trivial.
Since $U$ is smooth and connected, it must be trivial as well.
\end{proof}

Returning to an arbitrary field, we now characterize those 
semi-abelian varieties that are anti-affine:

\begin{proposition}\label{prop:aasab}
Let $G$ be a semi-abelian variety, extension of an abelian
variety $A$ by a torus $T$ with character group $\Lambda$, 
and let $c: \Lambda \to \widehat{A}(k_s)$ be the map classifying 
this extension. Then $G$ is anti-affine if and only if
$c$ is injective.
\end{proposition}

\begin{proof}
Since $G$ is anti-affine if and only if $G_{\bar{k}}$ is anti-affine,
we may assume $k$ algebraically closed. Then the isomorphism
(\ref{eqn:sh}) yields 
\[ \cO(G) = H^0(A,q_*(\cO_G)) \cong 
\bigoplus_{\lambda \in \Lambda} H^0(A, c(\lambda)). \]
Thus, $G$ is anti-affine if and only if $H^0(A,c(\lambda)) = 0$
for all non-zero $\lambda$. So it suffices to check that 
$H^0(A,\cL) = 0$ for any non-zero $\cL \in \widehat{A}(k)$.

If $H^0(A,\cL) \neq 0$, then $\cL \cong \cO_A(D)$ for some non-zero 
effective divisor $D$ on $A$. Since $\cL$ is algebraically trivial, 
the intersection number $D \cdot C$ is zero for any irreducible 
curve $C$ on $A$. Now choose a smooth point $x \in D_\red(k)$; 
then there exists an irreducible curve $C$ through $x$ which
intersects $D_\red$ transversally at that point. Then $D \cdot C > 0$,
a contradiction.
\end{proof}

\begin{corollary}\label{cor:aasab}
The anti-affine semi-abelian varieties over a field $k$ are classified
by the pairs $(A,\Lambda)$, where $A$ is an abelian variety
over $k$ and $\Lambda \subset \widehat{A}(k_s)$ is a $\Gamma$-stable
free abelian subgroup of finite rank.
\end{corollary}

If the ground field $k$ is finite, then the group 
$\widehat{A}(k_s) = \widehat{A}(\bar{k})$
is the union of the subgroups $\widehat{A}(K)$, where $K$ runs over
all finite extensions of $k$. Since all these subgroups are finite, 
$\widehat{A}(k_s)$ is a torsion group. In view of Proposition
\ref{prop:aacarp} and Corollary \ref{cor:aasab}, this readily yields:

\begin{corollary}\label{cor:fin}
Any anti-affine group over a finite field is an abelian variety.
\end{corollary}

Using the Rosenlicht decomposition (Theorem \ref{thm:ros}),
this yields in turn:

\begin{corollary}\label{cor:ari}
Let $G$ be a smooth connected algebraic group over a 
finite field $k$. Then there is a unique decomposition 
$G = L \cdot A$, where $L \trianglelefteq G$ is smooth, 
connected and affine, and $A \subseteq G$ is an abelian variety. 
Moreover, $L \cap A$ is finite.
\end{corollary}

\begin{remark}\label{rem:nonsplit}
Returning to an abelian variety $A$ over an arbitrary field $k$, 
assume that there exists an invertible sheaf $\cL \in \widehat{A}(k)$
of infinite order (such a pair $(A,\cL)$ exists unless $k$ is 
algebraic over a finite field, as seen in Example \ref{ex:ell}). 
Consider the associated extension
\[ 1 \longrightarrow \bG_m \longrightarrow G \longrightarrow A 
\longrightarrow 1. \]
Then $G$ is anti-affine by Proposition \ref{prop:aasab}, 
and hence the above extension is not split. In fact, 
it does not split after pull-back by any isogeny $B \to A$:
otherwise, it would split after pull-back by $n_A : A \to A$
for some positive integer $n$, or equivalently, after push-forward
under the $n$th power map of $\bG_m$. But this push-forward
amounts to replacing $\cL$ with $\cL^{\otimes n}$, which is still
of infinite order.
\end{remark}

Next, we turn to the classification of anti-affine groups
in characteristic $0$. As a first step, we obtain:

\begin{lemma}\label{lem:aa}
Let $G$ be a connected commutative algebraic group over a 
field of characteristic $0$. Let $T \subseteq G$ be the 
largest torus and $U \subseteq G$ the largest unipotent
subgroup scheme. Then $G$ is anti-affine if and only if 
$G/U$ and $G/T$ are anti-affine.
\end{lemma}

\begin{proof}
If $G$ is anti-affine, then so are its quotients $G/U$
and $G/T$. Conversely, assume that $G/U$ and $G/T$
are anti-affine. Then $G/G_\ant \cdot U$ is anti-affine, and
also affine as a quotient of $G/G_\ant$. Thus, $G = G_\ant \cdot U$ 
and hence $G/G_\ant \cong U/U \cap G_\ant$ is unipotent. Likewise, 
one shows that $G/G_\ant$ is a torus. Thus, $G/G_\ant$ is trivial. 
\end{proof}

With the notation and assumptions of the above lemma,
$G/U$ is a semi-abelian variety and $G/T$ is a vector 
extension of $A$. So to complete the classification,
it remains to characterize those vector extensions 
that are anti-affine:

\begin{proposition}\label{prop:aavec}
Assume that $\car(k) = 0$. 
Let $G$ be an extension of an abelian variety $A$ by
a vector group $U$ and denote by $\gamma: H^1(A,\cO_A)^* \to U$
the linear map classifying the extension. Then $G$ is
anti-affine if and only if $\gamma$ is surjective.
\end{proposition}

\begin{proof}
We have to show that the universal vector extension $E(A)$
is anti-affine, and every anti-affine vector extension 
(\ref{eqn:vec}) is a quotient of $E(A)$. 

Let $V := H^1(A,\cO_A)^*$; then $V = E(A)_\aff$ with the
notation of the Rosenlicht decomposition. By that decomposition,
$(E(A)_\ant)_\aff \subseteq V \cap E(A)_\aff$ and the quotient
is finite. Since $V \cap E(A)_\aff$  is a vector group,
we obtain the equality $(E(A)_\ant)_\aff = V \cap E(A)_\aff =: W$. 
In view of Remark \ref{rem:ros}, this yields a commutative 
diagram of extensions
\[
\xymatrix{
1 \ar[r] & W \ar[r] \ar[d]_{\iota} & E(A)_\ant \ar[r] \ar[d] &
A \ar[r] \ar[d]_{\id} & 1 \\
1 \ar[r] & V \ar[r] & E(A) \ar[r] & A \ar[r] & 1, \\}
\]
where $\iota$ is injective.
As a consequence, every extension of $A$ by $\bG_a$ is obtained
by pushout from the top extension, i.e., the resulting map
$W^* \to \Ext^1(A, \bG_a)$ is surjective. Since the bottom
extension is universal, it follows that $W = V$ and 
$E(A)_\ant = E(A)$.

Next, let $G$ be a vector extension of $A$. If $G$ is anti-affine,
then the classifying map $\gamma : E(A) \to G$ is surjective 
by Lemma \ref{lem:quotient}. The converse assertion follows from 
the fact that every quotient of an anti-affine group is anti-affine.
\end{proof}

As a consequence, \textit{the anti-affine vector extensions of $A$ 
are classified by the linear subspaces $V \subseteq H^1(A,\cO_A)$}
by assigning to any such extension, the image of the transpose
of its classifying map. Combining this result with the isomorphism 
(\ref{eqn:pf}), Lemma \ref{lem:aa} and Proposition \ref{prop:aasab}, 
we obtain the desired classification:

\begin{theorem}\label{thm:aa}
When $\car(k) > 0$, the anti-affine groups are classified
by the pairs $(A,\Lambda)$, where $A$ is an abelian variety
over $k$ and $\Lambda \subseteq \widehat{A}(k_s)$ is a $\Gamma$-stable
free abelian subgroup of finite rank.

When $\car(k) = 0$, the anti-affine groups are classified
by the triples $(A,\Lambda,V)$, where $(A,\Lambda)$
is as above and $V \subseteq H^1(A,\cO_A)$ is a linear subspace.
\end{theorem}

\subsection{Commutative algebraic groups (continued)}
\label{subsec:cagc}

We first show that every group as in the title has a 
``semi-abelian radical'': 

\begin{lemma}\label{lem:sab}
Let $G$ be a commutative algebraic group.

\begin{enumerate}

\item $G$ has a largest semi-abelian subvariety that we will
denote by $G_\sab$. Moreover, $(G/G_\sab)_\sab$ is trivial.

\item The formation of $G_\sab$ commutes with algebraic field
extensions.

\end{enumerate}

\end{lemma}

\begin{proof}
(1) This follows from the stability of semi-abelian varieties under
taking quotients and extensions (Corollary \ref{cor:extsab}) 
by arguing as in the proof of Lemma \ref{lem:ar}.

(2) When $\car(k) = 0$, the statement is obtained by Galois descent
as in the proof of Lemma \ref{lem:ar} again. When $\car(k) = p > 0$,
we have $G_\ant \subseteq G_\sab$ by Proposition \ref{prop:aacarp};
also, $G_\sab/G_\ant = (G/G_\ant)_\sab$ as follows from Corollary 
\ref{cor:extsab} again. Since $G/G_\ant$ is affine, 
$(G/G_\ant)_\sab$ is just its largest subtorus. As the formations
of $G_\ant$ and of the largest subtorus commute with field
extensions, the assertion follows.
\end{proof}

Next, we characterize those commutative algebraic groups that
have a trivial semi-abelian radical:

\begin{lemma}\label{lem:fexp}
If $\car(k)= p > 0$, then the following conditions are equivalent 
for a commutative algebraic group $G$:

\begin{enumerate}

\item $G_\sab$ is trivial.

\item $G$ is affine and its largest subgroup of multiplicative
type is finite.

\item The multiplication map $n_G$ is zero for some positive 
integer $n$.

\end{enumerate}

If in addition $G$ is smooth and connected, these conditions 
are equivalent to $G$ being unipotent.
\end{lemma}

\begin{proof}
(1)$\Rightarrow$(2) Note that $G_\ant$ is trivial in view of 
Proposition \ref{prop:aacarp}. Thus, $G$ is affine. Also, $G$ 
contains no non-trivial torus; this yields the assertion.

(2)$\Rightarrow$(3) By Theorem \ref{thm:cag},
we have an exact sequence 
\[ 1 \longrightarrow M \longrightarrow G \longrightarrow U 
\longrightarrow 1, \]
where $M$ is of multiplicative type and $U$ is unipotent.
Then $M$ is finite, and hence killed by $n_M$ for some
positive integer $n$. Also, recall that $U$ is killed by 
$p^m_U$ for some $m$. It follows that $np^m_G = 0$.

(3)$\Rightarrow$(1) This follows from the fact that
$n_H \neq 0$ for any non-trivial semi-abelian variety $H$
and any $n \neq 0$.

When $G$ is smooth and connected, the condition (2) implies
that $G$ is unipotent, in view of Theorem \ref{thm:cag}
again. Conversely, if $G$ is unipotent, then it clearly
satisfies the condition (2).
\end{proof}

We say that a commutative algebraic group $G$ 
\textit{has finite exponent} if it satisfies the above 
condition (3). When $\car(k) = 0$, this just means that
$G$ is finite; when $\car(k) > 0$, this amounts to $G$
being an extension of a unipotent algebraic group by a finite
group scheme of multiplicative type.

\begin{theorem}\label{thm:cagc}
Assume that $\car(k) = p > 0$ and consider a commutative
algebraic group $G$ and its largest semi-abelian subvariety
$G_\sab$. 

\begin{enumerate}

\item The quotient $G/G_\sab$ has finite exponent.

\item There exists a subgroup scheme $H \subseteq G$ 
such that $G = G_\sab \cdot H$ and $G_\sab \cap H$ is finite. 

\item If $G$ is smooth and connected, then $G/G_\sab$ 
is unipotent. If in addition $k$ is perfect, then we may 
take for $H$ the largest smooth connected unipotent subgroup 
scheme of $G$.

\end{enumerate}

\end{theorem}

\begin{proof}
(1) This follows from Lemmas \ref{lem:sab} and \ref{lem:fexp}.

(2) By (1), we may choose $n > 0$ such that 
$n_{G/G_\sab} = 0$. Then $n_G$ factors through $G_\sab$;
also, recall that $n_{G_\sab}$ is an isogeny. It follows
that $G = G_\sab \cdot \Ker(n_G)$ and 
$G_\sab \cap \Ker(n_G)$ is finite.

(3) The first assertion is a consequence of Lemma 
\ref{lem:fexp} again. 

Assume $k$ perfect and consider the Rosenlicht decomposition 
$G = G_\aff \cdot G_\ant$. Then $G_\aff = T \times U$, 
where $T$ is a torus and $U$ a smooth connected unipotent group. 
Thus, $G = T \cdot U \cdot G_\ant$. Also, $T \cdot G_\ant$ 
is a semi-abelian subvariety of $G$; moreover, 
$G/T \cdot G_\ant$ is isomorphic to a quotient of $U$,
and hence contains no semi-abelian variety. Thus, 
$T \cdot G_\ant = G_\sab$ and hence $G = G_\sab \cdot U$. 
Moreover, $G_\sab \cap U$ is clearly finite.
\end{proof}

In analogy with the structure of commutative affine algebraic
groups (Theorem \ref{thm:cag}), note that the class of commutative
algebraic groups of finite exponent is stable under taking
subgroup schemes, quotients, commutative group extensions, and
extensions of the ground field. Moreover, when $G$ is a semi-abelian
variety and $H$ a commutative algebraic group of finite exponent,
every homomorphism $\varphi: G \to H$ is constant, and every 
homomorphism $\psi  : H \to G$ factors through a finite subgroup
scheme of $G$.

Also, note the following analogue of Lemma \ref{lem:quotient},
which yields that the assignment $G \mapsto G_\sab$ is close
to being exact: 

\begin{lemma}\label{lem:ex}
Assume that $\car(k) = p$. Let $G$ be a commutative algebraic
group and $H \subseteq G$ a subgroup scheme. 

\begin{enumerate}

\item $H_\sab \subseteq G_\sab \cap H$ and the quotient is finite.

\item The quotient map $q : G \to G/H$ yields an isomorphism
\[ G_\sab/G_\sab \cap H \stackrel{\cong}{\longrightarrow}
(G/H)_\sab. \]

\end{enumerate}

\end{lemma}

\begin{proof}
(1) This follows readily from Corollary \ref{cor:extsab} and 
Lemma \ref{lem:sab}.

(2) Observe that $q$ restricts to a closed immersion
of group schemes 
\[ \iota : G_\sab/G_\sab \cap H \longrightarrow G/H \]
that we will regard as an inclusion. 
Also, by Theorem \ref{thm:cagc}, the quotient 
$G/G_\sab \cdot H$ has finite exponent. Since 
$G/G_\sab \cdot H \cong (G/H)/ (G_\sab \cdot H/H)$
is isomorphic to the cokernel of $\iota$, it follows
that $G_\sab/G_\sab \cap H \supseteq (G/H)_\sab$.
But the opposite inclusion holds by Corollary \ref{cor:extsab} 
again; this yields the assertion.  
\end{proof}

\begin{remark}\label{rem:cagc}
When $\car(k) = 0$, the quotient $G/G_\sab$ is not necessarily
unipotent for a smooth connected commutative group $G$. This
happens for example when $G$ is the universal vector extension 
of a non-trivial abelian variety; then $G_\sab$ is trivial. 

Still, the structure of such a group $G$ reduces somehow
to that of its anti-affine part. Specifically, one may check 
that there exists a subtorus $T \subseteq G$ and a vector 
subgroup $U \subseteq G$ such that the multiplication map
$T \times U \times G_\ant \to G$ is an isogeny; moreover,
$T$ is uniquely determined up to isogeny, and $U$ is uniquely
determined.
\end{remark}

As an application of Theorem \ref{thm:cagc}, we present 
a remedy to (or a measure of) the failure of Chevalley's 
structure theorem over imperfect fields. To state it, 
we need the following:

\begin{definition}\label{def:pa}
A smooth connected algebraic group $G$ is called a
\textit{pseudo-abelian variety} if it does not contain any 
non-trivial smooth connected affine normal subgroup scheme.
\end{definition}

\begin{remarks}\label{rem:pa}
(i) By Lemma \ref{lem:ar}, every smooth connected
algebraic group $G$ lies in a unique exact sequence
\[ 1 \longrightarrow L \longrightarrow G \longrightarrow Q
\longrightarrow 1, \]
where $L$ is smooth, connected and linear, and $Q$ is a 
pseudo-abelian variety.

(ii) If $k$ is perfect, then every pseudo-abelian variety
is just an abelian variety, as follows from Theorem 
\ref{thm:chev}. But there exist non-proper pseudo-abelian
varieties over any imperfect field, as shown by Example 
\ref{ex:ray}.
\end{remarks}

\begin{corollary}\label{cor:pa}
Let $G$ be a pseudo-abelian variety. Then $G$ is commutative 
and lies in a unique exact sequence
\begin{equation}\label{eqn:pa}
1 \longrightarrow A \longrightarrow G \longrightarrow U 
\longrightarrow 1, 
\end{equation}
where $A$ is an abelian variety and $U$ is unipotent.
\end{corollary}

\begin{proof}
The commutativity of $G$ follows from Corollary \ref{cor:ros}.

For the remaining assertion, we may assume that $\car(k) = p > 0$.
By Theorem \ref{thm:cagc}, $G$ lies in a unique extension
\[ 1 \longrightarrow G_\sab \longrightarrow G \longrightarrow U
\longrightarrow 1, \]
where $G_\sab$ is a semi-abelian variety and $U$ is unipotent. 
Since $G$ contains no non-trivial torus, $G_\sab$ is an abelian 
variety. This shows the existence of the exact sequence 
(\ref{eqn:pa}); for its uniqueness, just observe that $A = G_\ant$.
\end{proof}

\medskip

\noindent
\textit{Notes and references}. 

Theorem \ref{thm:ros} is due to Rosenlicht (see 
\cite[Cor.~5, p.~140]{Ro1}). This result is very useful for 
reducing questions about general algebraic groups to the linear 
and anti-affine cases; see \cite{MZ} for a recent application. 

When $\car(k) = p >0$, characterizing Lie algebras of 
\textit{smooth} algebraic groups among $p$-Lie algebras seems 
to be an open problem. It is well-known that every 
finite-dimensional $p$-Lie algebra is the Lie algebra of some 
\textit{infinitesimal} group scheme; see \cite[II.7.3, II.7.4]{DG} 
for this result and further developments.

The proof of Theorem \ref{thm:equiv} is adapted from 
\cite[Thm.~4.13]{Br4}. The quasi-projectivity of homogeneous 
spaces is a classical result, see e.g. \cite[Cor.~VI.2.6]{Ra}. 
Also, the existence of equivariant compactifications 
of certain homogeneous spaces having no separable point 
at infinity has attracted recent interest, see \cite{Ga, GGM}. 

In positive characteristics, the existence of (equivariant 
or not) \textit{regular} projective compactifications 
of homogeneous spaces is an open question.

Example \ref{ex:cag} is due to Raynaud, see 
\cite[XVII.C.5]{SGA3}.

The algebraic group considered in Remark \ref{rem:line} (i) 
is an example of a $k$-\textit{wound} unipotent group $G$,
i.e., every morphism $\bA^1 \to G$ is constant. This notion
plays an important r\^ole in the structure of smooth connected
unipotent groups over imperfect fields, see \cite[B.2]{CGP}.

The notion of Albanese morphism extends to non-pointed varieties
by replacing semi-abelian varieties with semi-abelian torsors.
More specifically, for any variety $X$, there exists a semi-abelian
variety $\Alb^0_X$, an $\Alb^0_X$-torsor $\Alb^1_X$ (over $\Spec(k)$)
and a morphism 
\[ u_X : X \longrightarrow \Alb^1_X \]
such that for any semi-abelian variety $A^0$, any torsor $A^1$ 
under $A^0$, and any morphism $f : X \to A^1$, 
there exists a unique morphism of varieties 
\[ g^1 : \Alb^1_X \to A^1 \] 
such that $g^1 \circ u_X = f$ and there exists a unique morphism 
of algebraic groups
\[ g^0 : \Alb^0_X \to A^0 \] 
such that $g_1$ is $g_0$-equivariant (see \cite[App.~A]{Wit}).

The structure of anti-affine algebraic groups has been obtained 
in \cite{Br1} and \cite{SS} independently. Our exposition follows
that of \cite{Br1} with some simplifications. 

The notion of a pseudo-abelian variety is due to Totaro 
in \cite{To}, as well as Corollary \ref{cor:pa} and further 
results about these varieties. In particular, it is shown
that every smooth connected commutative group of exponent $p$ 
occurs as the unipotent quotient of some pseudo-abelian variety, 
see \cite[Cor.~6.5, Cor.~7.3]{To}). This yields many more
examples of pseudo-abelian varieties than those constructed
in Example \ref{ex:ray}. Yet a full description of pseudo-abelian 
varieties is an open problem.

\section{The Picard scheme}
\label{sec:tps}

\subsection{Definitions and basic properties}
\label{subsec:dbp}

\begin{definition}\label{def:pic}
Let $X$ be a scheme. The \textit{relative Picard functor},
denoted by $\Pic_{X/k}$, is the commutative group functor that assigns 
to any scheme $S$ the group $\Pic(X \times S)/p_2^* \Pic(S)$,
where $p_2 : X \times S \to S$ denotes the projection, and to
any morphism of schemes $f : S' \to S$, the homomorphism
induced by pull-back.
\end{definition}

If $X$ is equipped with a $k$-rational point $x$, then
for any scheme $S$, the map $x \times \id: S \to X \times S$ 
is a section of $p_2 : X \times S \to S$. Thus, 
\[ (x \times \id)^* : \Pic(X \times S) \to \Pic(S) \]
is a retraction of 
\[ p_2^* : \Pic(S) \to \Pic(X \times S). \]
So we may view $\Pic_{X/k}(S)$ as the group of isomorphism
classes of invertible sheaves on $X \times S$, trivial
along $x \times S$. If in addition $S$ is equipped with
a $k$-rational point $s$, then we obtain a pull-back map
$s^* : \Pic_{X/k}(S) \to \Pic_{X/k}(s) = \Pic(X)$
with kernel isomorphic to 
$\Pic(X \times S)/p_1^* \Pic(X) \times p_2^* \Pic(S)$.
Indeed, the map 
\[ s^* \times x^* : \Pic(X \times S) \longrightarrow 
\Pic(X) \times \Pic(S) \]
is a retraction of 
\[ p_1^* \times p_2^* : \Pic(X) \times \Pic(S) \longrightarrow 
\Pic(X \times S). \]
The kernel of $s^* \times x^*$ is called the group of
\textit{divisorial correspondences}.

\begin{theorem}\label{thm:pic}
Let $X$ be a proper scheme having a $k$-rational point.

\begin{enumerate}

\item $\Pic_{X/k}$ is represented by a locally algebraic group
(that will still be denoted by $\Pic_{X/k}$).

\item $\Lie(\Pic_{X/k}) = H^1(X,\cO_X)$.

\item If $H^2(X,\cO_X) = 0$ then $\Pic_{X/k}$ is smooth.

\end{enumerate}

\end{theorem}

\begin{proof}
(1) This is proved in \cite[II.15]{Mur} via a characterization
of commutative locally algebraic groups among commutative
group functors, in terms of seven axioms. When $X$ is projective,
there is an alternative proof via the Hilbert scheme, see
\cite[Thm.~4.8]{Kl}.

The assertion (2) follows from \cite[Thm.~5.11]{Kl}, and (3) from 
\cite[Prop.~5.19]{Kl}.
\end{proof}

With the notation and assumptions of the above theorem,
the neutral component, $\Pic_{X/k}^0$, is a connected algebraic 
group by Theorem \ref{thm:neut}. The group of connected 
components, $\pi_0(\Pic_{X/k})$, is called the 
\textit{N\'eron-Severi group}; we denote it by $\NS_{X/k}$.
The formation of $\Pic_{X/k}$ commutes with field extensions; 
hence the same holds for $\Pic^0_{X/k}$ and $\NS_{X/k}$. Also,
the commutative group $\NS_{X/k}(\bar{k})$ is finitely generated
in view of \cite[XIII.5.1]{SGA6}.

\begin{remarks}\label{rem:pic}
(i) For any pointed scheme $(S,s)$, we have a natural isomorphism
of groups
\begin{equation}\label{eqn:pic}
\Hom_{\ptsch}(S, \Pic_{X/k}) \cong 
\Pic(X \times S)/p_1^*\Pic(X) \times p_2^* \Pic(S),
\end{equation}
where the left-hand side denotes the subgroup of 
$\Hom(S,\Pic_{X/k})$ consisting of those $f$ such that $f(s) = 0$,
i.e., the kernel of $s^* : \Pic_{X/k}(S) \to \Pic(X)$.

(ii) If $X$ is an abelian variety, then $\Pic^0_{X/k}$ is the dual
abelian variety, $\widehat{X}$. 
\end{remarks}

\begin{definition}\label{def:picvar}
Let $X$ be a proper scheme over a perfect field $k$, having
a $k$-rational point. The \textit{Picard variety}
$\Pic^0(X)$ is the reduced scheme $(\Pic_{X/k}^0)_\red$.
\end{definition}

With the above notation and assumptions, $\Pic^0(X)$ is 
a smooth connected commutative algebraic group; its formation
commutes with field extensions. 

Returning to a proper scheme $X$ over an arbitrary ground field 
$k$, having a $k$-rational point, we say that 
\textit{the Picard variety of $X$ exists}
if $\Pic^0(X) := (\Pic^0_{X/k})_\red$ is a subgroup scheme of 
the Picard scheme. We will see that this holds under mild
assumptions on the singularities of $X$.

\subsection{Structure of Picard varieties}
\label{subsec:spv}

Throughout this subsection, we consider a proper variety $X$ 
equipped with a $k$-rational point $x$.

\begin{proposition}\label{prop:picalb}
The Albanese variety of $X$ is canonically isomorphic to the dual 
of the largest abelian subvariety of $\Pic^0_{X/k}$.
\end{proposition}
  
\begin{proof}
Let $A$ be an abelian variety. Then we have 
\[ \Hom_\gpsch(A,\Pic^0_{X/k}) = \Hom_\ptsch(A,\Pic^0_{X/k}) \]
in view of Proposition \ref{prop:rig}. Moreover, 
\[ \Hom_\ptsch(A,\Pic^0_{X/k}) = \Hom_\ptsch(A, \Pic_{X/k}). \]
Using the isomorphism (\ref{eqn:pic}), this yields an isomorphism
\[ \Hom_\gpsch(A,\Pic^0_{X/k}) \cong 
\Pic(A \times X)/p_1^* \Pic(A) \times p_2^* \Pic(X), \]
which is contravariant in $A$ and $X$. Exchanging the roles of
$A$, $X$ and using (\ref{eqn:pic}) again, we obtain
an isomorphism 
\[ \Hom_\gpsch(A,\Pic^0_{X/k}) \cong \Hom_\ptsch(X,\Pic^0_A), \]
contravariant in $A$ and $X$ again. This readily yields 
the assertion (and reproves the existence of the Albanese morphism 
in this setting).
\end{proof}

\begin{proposition}\label{prop:picn}
When $X$ is geometrically normal, its Picard variety
exists and is an abelian variety.
\end{proposition} 

\begin{proof}
By Lemma \ref{lem:prop}, it suffices to show that $\Pic^0_{X/k}$
is proper. For this, we may assume $k$ algebraically closed. 
In view of Theorem \ref{thm:chev}, it suffices in turn 
to show that $\Pic^0_{X/k}$ contains no non-trivial smooth 
connected affine subgroup.
As any such subgroup contains a copy of $\bG_a$ or $\bG_m$ 
(see e.g. \cite[3.4.9, 6.2.5, 6.3.4]{Sp}), we are reduced 
to checking that every homomorphism from $\bG_a$ or $\bG_m$ 
to $\Pic^0_{X/k}$ is constant. But we clearly have 
\[ \Hom_\gpsch(\bG_a,\Pic^0_{X/k}) \subseteq
\Hom_\ptsch(\bA^1,\Pic_{X/k}). \]
Moreover,
\[ \Hom_\ptsch(\bA^1,\Pic_{X/k}) \cong 
 \Pic(X \times \bA^1)/p_1^* \Pic(X) \]
in view of (\ref{eqn:pic}) and the triviality of $\Pic(\bA^1)$.
Since $X$ is normal, the divisor class group
$\Cl(X \times \bA^1)$ is isomorphic to $\Cl(X)$ via 
$p_1^*$; moreover, for any Weil divisor $D$ on $X$,
the pull-back $p_1^*(D)$ is Cartier if and only if
$D$ is Cartier. Thus, the map
$p_1^* : \Pic(X) \to \Pic(X \times \bA^1)$ is an 
isomorphism. As a consequence, every homomorphism
$\bG_a \to \Pic^0_{X/k}$ is constant. 

Arguing similarly with $\bG_a$ replaced by $\bG_m$,
we obtain 
\[ \Hom_\gpsch(\bG_m,\Pic^0_{X/k}) \subseteq 
\Pic(X \times (\bA^1 \setminus 0))/p_1^* \Pic(X). \]
Also, the pull-back map
$\Cl(X \times \bA^1) \to \Cl(X \times (\bA^1 \setminus 0))$
is surjective. It follows that 
$p_1^* : \Cl(X) \to \Cl(X \times (\bA^1 \setminus 0))$ 
is an isomorphism and restricts to an isomorphism
on Picard groups. Thus, every homomorphism 
$\bG_m \to \Pic^0_{X/k}$ is constant as well.
\end{proof}

\begin{definition}\label{def:sn}
A scheme $X$ is \textit{semi-normal} if $X$ is reduced and every
finite bijective morphism of schemes $f : Y \to X$ that induces
an isomorphism on all residue fields is an isomorphism.
\end{definition}

Examples of semi-normal schemes include of course normal
varieties, and also divisors with smooth normal crossings.
Nodal curves are semi-normal; cuspidal curves are not.
By \cite[Cor.~5.7]{GT}, semi-normality is preserved under 
separable field extensions. But it is not preserved 
under arbitrary field extensions, as shown by Example 
\ref{ex:pic} below. We say that a scheme $X$ is 
\textit{geometrically semi-normal} if $X_{\bar{k}}$ is semi-normal.

\begin{proposition}\label{prop:picsn}
When $X$ is geometrically semi-normal, its Picard variety exists
and is a semi-abelian variety.
\end{proposition}

\begin{proof}
Using Lemma \ref{lem:sab}, we may assume $k$ perfect.
By Proposition \ref{prop:line}, it suffices to show that
every morphism $\bA^1 \to \Pic^0_{X/k}$ is constant. By
(\ref{eqn:pic}), we have
\[ \Hom_\ptsch(\bA^1,\Pic_{X/k}) \subseteq 
\Pic(X \times \bA^1)/p_1^* \Pic(X). \] 
So it suffices in turn to show that the map
\[ p_1^* : \Pic(X) \longrightarrow \Pic(X \times \bA^1) \]
is an isomorphism. When $X$ is affine, this follows from
\cite[Thm.~3.6]{Tr}. For an arbitrary variety $X$, we
consider the first terms of the Leray spectral sequence
associated with the morphism $p_1 : X \times \bA^1 \to X$ and 
the sheaf $\cO_{X \times \bA^1}^*$ consisting of the units of 
the structure sheaf. This yields an exact sequence
\[ 0 \longrightarrow H^1(X, p_{1*}(\cO_{X \times \bA^1}^*)) 
\stackrel{p_1^*}{\longrightarrow} 
H^1(X \times \bA^1, \cO_{X \times \bA^1}^*)
\longrightarrow H^0(X, R^1 p_{1*}(\cO_{X \times \bA^1}^*)). \]
Also, the natural map $\cO_X^* \to p_{1*}(\cO_{X \times \bA^1}^*)$
is an isomorphism, since $R[t]^* = R^*$ for any integral
domain $R$. Thus, 
\[ H^1(X, p_{1*}(\cO_{X \times \bA^1}^*)) =  H^1(X,\cO_X^*)
= \Pic(X). \]
Moreover, 
$H^1(X \times \bA^1, \cO_{X \times \bA^1}^*) = \Pic(X \times \bA^1)$.
Thus, it suffices to show that 
$R^1 p_{1*}(\cO_{X \times \bA^1}^*) = 0$. Recall that 
$R^1 p_{1*}(\cO_{X \times \bA^1}^*)$ is the sheaf on $X$ 
associated with the presheaf $U \mapsto \Pic(U \times \bA^1)$.
When $U$ is affine, we already saw that the map 
$p_1^* : \Pic(U) \to \Pic(U \times \bA^1)$ is an isomorphism. 
Moreover, for any $x \in X$ and $L \in \Pic(U)$, there exists 
an open affine neighborhood $V$ of $x$ in $U$ such that 
$L\vert_V$ is trivial. This yields the desired vanishing.
\end{proof}

The above proposition does not extend to semi-normal schemes,
as shown by the following:

\begin{example}\label{ex:pic}
Let $k$ be an imperfect field. Set $p:= \car(k)$ and choose 
$a \in k \setminus k^p$. Like in Remark \ref{rem:qp} (iv),
consider the regular projective curve $X$ defined as the
zero scheme of $y^p - x z^{p-1} - a x^p$ in $\bP^2$. Then 
$X$ is not geometrically semi-normal, since 
$X_{\bar{k}}$ is a cuspidal curve: the zero scheme of 
$(y- a^{1/p} x)^p - x z^{p-1}$. We now show that $\Pic^0_{X/k}$ is 
a smooth connected unipotent group, non-trivial if $p \geq 3$.
 
The smoothness of $\Pic^0_{X/k}$ follows from Theorem \ref{thm:pic} 
(3). Also, the group $\Pic_{X/k}^0(k_s)$ is non-trivial and killed by 
$p$ when $p \geq 3$, see \cite[Thm.~6.10.1, Lem. 6.11.1]{KMT}. 
It follows that $\Pic^0_{X/k}$ is non-trivial and killed by 
$p$ as well. By using the structure of commutative
algebraic groups, e.g., Lemma \ref{lem:fexp}, 
this implies that $\Pic^0_{X/k}$ is unipotent. 
\end{example}

Next, we present a classical example of a smooth projective
surface for which the Picard scheme is not smooth:

\begin{example}\label{ex:igusa}
Assume that $k$ is algebraically closed of characteristic $2$.
Let $E$ be an ordinary elliptic curve, i.e., it has a (unique)
$k$-rational point $z_0$ of order $2$. Let $F$ be another
elliptic curve and consider the automorphism $\sigma$ of
$E \times F$ such that 
\[ \sigma(z,w) = (z + z_0, -w) \]
identically. Then $\sigma$ has order $2$ and fixes no point of
$E \times F$. Thus, there exists a quotient morphism
\[ q : E \times F \longrightarrow X \]
for the group $\langle \sigma \rangle$ generated by $\sigma$; 
moreover, $q$ is finite and \'etale, and $X$ is a smooth 
projective surface. Denote by $E/\langle z_0 \rangle$ the quotient 
of $E$ by the subgroup generated by $z_0$. Then the projection 
$E \times F \to E$  descends to a morphism 
\[ \alpha : X \longrightarrow E/\langle z_0 \rangle. \]

We claim that $\alpha$ is the Albanese morphism of the pointed
variety $(X,x)$, where $x := q(0,0)$. Consider indeed a
morphism 
\[ f : X \longrightarrow A, \] 
where $A$ is an abelian variety and $f(x) = 0$. 
Composing $f$ with $q$ yields a $\sigma$-invariant morphism 
\[ g : E \times F \longrightarrow A, \quad (0,0) \longmapsto 0. \]
By Lemma \ref{lem:alb}, we have 
$g(z,w) = g(z,0) + g(0,w)$ identically. Moreover, the map
$h : F \to A$, $w \mapsto g(0,w)$ is a homomorphism by
Proposition \ref{prop:alba}. In particular, $h(- w) = - h(w)$.
But $h(-w) = h(w)$ by $\sigma$-invariance. Thus, $h$
factors through the kernel of $2_F$, a finite group scheme.
Since $F$ is smooth and connected, it follows that $h$
is constant. Thus, $g(z,w)  = g(z,0) = g(z + z_0, 0)$;
this yields our claim.

Combining that claim with Propositions \ref{prop:picalb} 
and \ref{prop:picn}, we see that 
\[ \Pic^0(X) \cong \widehat{E/\langle z_0 \rangle} 
\cong E/\langle z_0 \rangle. \]

Next, we claim that \textit{the tangent sheaf $T_X$ is trivial}.
Indeed, since $q$ is \'etale, we have 
\begin{equation}\label{eqn:tan}
q^*(T_X) \cong T_{E \times F} \cong 
\cO_{E \times F} \otimes_k (\Lie(E) \oplus \Lie(F)). 
\end{equation}
These isomorphisms are equivariant for the natural
action of $\sigma$ on $\cO_{E \times F}$ and its trivial
action on $\Lie(E) \oplus \Lie(F)$ (recall that 
$\sigma$ acts on $E$ by a translation, and on $F$ by $-1$).
Thus, we obtain
\[ T_X \cong q_*(q^*(T_X))^{\langle \sigma \rangle} \cong 
q_*(\cO_{E \times F})^{\langle \sigma \rangle} 
\otimes_k (\Lie(E) \oplus \Lie(F))
\cong \cO_X \otimes_k (\Lie(E) \oplus \Lie(F)).\]
This yields the claim.

By that claim and the Riemann-Roch theorem, we
obtain $\chi(\cO_X) = 0$. Since $h^0(\cO_X) = 1$
and $h^2(\cO_X) = h^0(\omega_X) = h^0(\cO_X) = 1$,
this yields $h^1(\cO_X) = 2$. In other words, 
the Lie algebra of $\Pic_{X/k}$ has dimension $2$; thus,
$\Pic_{X/k}$ \textit{is not smooth}.
\end{example}

\medskip

\noindent
\textit{Notes and references}. 

As mentioned in the introduction, a detailed reference for
Picard schemes is Kleiman's article \cite{Kl}; see also
\cite{BLR}, especially Chapter 8 for general results on
the Picard functor, and Chapter 9 for applications to
relative curves.

Proposition \ref{prop:picn} is well-known, see e.g. 
\cite[Thm.~5.4]{Kl}. 

Proposition \ref{prop:picsn} is due to Alexeev when $k$
is algebraically closed, see \cite[Thm.~4.1.7]{Al}.

Example \ref{ex:igusa} is due to Igusa, see \cite{Ig}.

Assume that $k$ is perfect and consider a proper scheme $X$ 
having a $k$-rational point. Then the Picard variety of $X$ 
lies in a unique exact sequence of the form (\ref{eqn:com}),  
\[ 1 \longrightarrow T_X \times U_X \longrightarrow \Pic^0(X)
\longrightarrow A_X \longrightarrow 1, \]
where $T_X$ is a torus, $U_X$ a smooth connected commutative
algebraic group, and $A_X$ an abelian variety. The affine part
$T_X \times U_X$ has been described by Geisser in \cite{Ge};
in particular, the dual of the character group of $T_X$ is
isomorphic to the \'etale cohomology group 
$H^1_\et(X_{\bar{k}}, \bZ)$ as a Galois module (see 
\cite[Thm.~1]{Ge}, and \cite[Thm.~7.1]{We}, \cite[Cor.~4.2.5]{Al} 
for closely related results). Also, when $X$ is reduced, 
$U_X$ is the kernel of the pull-back map 
$f^* : \Pic_{X/k} \to \Pic_{X^+/k}$, where $f: X^+ \to X$ 
denotes the semi-normalization (see \cite[Thm.~3]{Ge}).

It is an open problem whether every smooth connected
algebraic group $G$ over a perfect field $k$ is 
the Picard variety of some proper scheme $X$. When $k$ 
is algebraically closed of characteristic $0$, one can
construct an appropriate scheme $X$ by using the structure 
of $G$ described in \S \ref{subsec:cag}, see 
\cite[Thm.~1.1]{Br3}. In fact, $X$ may be taken projective,
and normal except at finitely many points. Yet when
$\car(k) > 0$, the unipotent part of the Picard variety 
of a scheme satisfying these conditions is not arbitrary, 
see \cite[Thm.~1.2]{Br3}.

In fact, it is not even known whether every torus is the largest
subtorus of some Picard variety; equivalently, whether every 
free abelian group equipped with an action of the Galois 
group $\Gamma$ is obtained as $H^1_\et(X_{\bar{k}}, \bZ)$
for some proper scheme $X$ (which can be assumed semi-normal).

In another direction, the structure of Picard schemes over
imperfect fields is largely unknown; see \cite[Ex.~3.1]{To}
for a remarkable example.

\section{The  automorphism group scheme}
\label{sec:tags}

\subsection{Basic results and examples}
\label{subsec:bra}

Recall from \S \ref{subsec:ags} the automorphism group functor
of a scheme $X$, i.e., the group functor $\Aut_X$ that assigns to 
any scheme $S$ the automorphism group of the $S$-scheme $X \times S$. 
We have the following representability result for $\Aut_X$, analogous
to that for the Picard scheme (Theorem \ref{thm:pic}):

\begin{theorem}\label{thm:aut}
Let $X$ be a proper scheme. 

\begin{enumerate}

\item The group functor $\Aut_X$ is represented by a locally 
algebraic group (that will be still denoted by $\Aut_X$).

\item $\Lie(\Aut_X) = H^0(X,T_X)$, where $T_X$ denotes the
sheaf of derivations of the structure sheaf $\cO_X$.

\item If $H^1(X,T_X) = 0$, then $\Aut_X$ is smooth.

\end{enumerate}

\end{theorem}

\begin{proof}
(1) This is obtained in \cite[Thm.~3.7]{MO} via an axiomatic 
characterization of locally algebraic groups among group functors, 
which generalizes that of \cite{Mur}. When $X$ is projective, 
the result follows from the existence of the Hilbert scheme. 
More specifically, the functor of endomorphisms is represented 
by an open subscheme $\End_X$ of the Hilbert scheme  
$\Hilb_{X \times X}$, by sending each endomorphism $u$ to its graph 
(the image of $\id \times u: X \to X \times X$), see \cite[Thm.~I.10]{Ko1}. 
Moreover, $\Aut_X$ is represented by an open subscheme of 
$\End_X$ in view of \cite[Lem.~I.10.1]{Ko1}.

(2) See \cite[Lem.~3.4]{MO} and also \cite[II.4.2.4]{DG}.

(3) This follows from \cite[III.5.9]{SGA1}.
\end{proof}

With the above notation and assumptions, we say that $\Aut_X$ 
is the \textit{automorphism group scheme} of $X$; its neutral 
component, $\Aut_X^0$, is a connected algebraic group.
The formations of $\Aut_X$ and $\Aut_X^0$ commute with field 
extensions. For any group scheme $G$, the datum of a $G$-action 
on $X$ is equivalent to a homomorphism $G \to \Aut_X$.

\begin{example}\label{ex:curve}
Let $C$ be a smooth, projective, geometrically irreducible curve
of genus $g \geq 2$. Then $T_C$ is the dual of the canonical
sheaf $\omega_C$, and hence $\deg(T_C) = 2 - 2 g < 0$; 
it follows that $H^0(C,T_C) = 0$. Thus, $\Aut_C$ is \'etale; 
equivalently, $\Aut_C^0$ is trivial. 

In fact, $\Aut_C$ \textit{is finite}. To see this, it suffices 
to show that $\Aut_C(\bar{k})$ is finite; thus,
we may assume $k$ algebraically closed. Let $f \in \Aut_C(k)$ and
consider its graph $\Gamma_f \subset C \times C$. Choose a point
$x \in C(k)$; then $\cO_C(x)$ is an ample invertible sheaf on 
$C$, and hence $\cL := \cO_C(x) \boxtimes \cO_C(x)$ is an ample
invertible sheaf on $C \times C$. Moreover, the pull-back of
$\cL$ to $\Gamma_f \cong C$ is isomorphic to $\cO_C(x + f(x))$;
by the Riemann-Roch theorem, the Hilbert polynomial 
$P: n \mapsto \chi(C,\cO_C(n (x + f(x))))$ is independent of $f$.
As a consequence, $\Aut_C$ is equipped with an immersion into the 
Hilbert scheme $\Hilb^P_{C \times C}$. Since the latter is projective, 
this yields the assertion.

By the above argument, $\Aut_C$ is an algebraic group for any
geometrically irreducible curve $C$ (take for $x$ a smooth
closed point of $C$).
\end{example}

\begin{example}\label{ex:ab}
Let $A$ be an abelian variety. Then the action of $A$ on itself
by translation yields a homomorphism 
\[ \tau : A \longrightarrow \Aut_A^0. \]
Clearly, $\Ker(\tau)$ is trivial and hence $\tau$ is a closed
immersion. Since $A$ is smooth and $\Aut_A^0$ is an
irreducible scheme of finite type, of dimension at most 
\[ 
h^0(T_A) = h^0(\cO_A \otimes_k \Lie(A)) = \dim \Lie(A) = \dim(A), 
\]
it follows that \textit{$\tau$ is an isomorphism.}

Also, one readily checks that 
\[ \Aut_A \cong \Aut_{A,0} \ltimes A, \]
where $\Aut_{A,0} \subseteq \Aut_A$ denotes the stabilizer of
the origin. It follows that $\Aut_{A,0} \cong \pi_0(\Aut_A)$ is
\'etale. 

If $A$ is an elliptic curve, then $\Aut_{A,0}$ is finite,
as may be seen by arguing as in the preceding example.
But this fails for abelian varieties of higher dimension, for
example, when $A = B \times B$ for a non-trivial abelian variety 
$B$: then $\Aut_{A,0}$ contains the constant group scheme
$\GL_2(\bZ)$ acting by linear combinations of entries.
\end{example}
 
Next, we present an application of Theorem \ref{thm:che} to 
the structure of connected automorphism group schemes.
Recall that a variety $X$ is \textit{uniruled} if there exist
an integral scheme of finite type $Y$ and a dominant rational map 
$\bP^1 \times Y \dasharrow X$ such that the induced map
$\bP^1_y \dasharrow X$ is non-constant for some $y \in Y$.
Also, $X$ is uniruled if and only if $X_{\bar{k}}$ is uniruled,
see \cite[IV.1.3]{Ko1}.

\begin{proposition}\label{prop:uni}
Let $X$ be a proper variety. If $X$ is not uniruled, then
$\Aut^0_X$ is proper.
\end{proposition}

\begin{proof}
We may assume $k$ algebraically closed. If $\Aut^0_X$ is not
proper, then it contains a connected affine normal subgroup 
scheme $N$ of positive dimension, as follows from Theorem 
\ref{thm:che}. In turn, the reduced subgroup scheme $N_\red$
contains a subgroup scheme $H$ isomorphic to $\bG_a$ or $\bG_m$. 
Since $H$ acts faithfully on $X$, the action morphism 
$H \times X \to X$ yields a uniruling.
\end{proof}

\begin{example}\label{ex:igusabis}
Assume that $k$ is algebraically closed of characteristic $2$.
Let $X$ be the smooth projective surface constructed in Example 
\ref{ex:igusa}. We claim that $\Aut_X$ \textit{is not smooth}.

To see this, recall that the tangent sheaf $T_X$ is trivial; 
hence $\Lie(\Aut_X)$ has dimension $2$. On the other hand,
$X$ is equipped with an action of the elliptic curve $E$,
via its action on $E \times F$ by translations on itself
(which commutes with the involution $\sigma$). This yields
a homomorphism $E \to \Aut_X$ that factors through 
\[ f : E \longrightarrow (\Aut_X^0)_\red =: G. \] 
Since $X$ contains an $E$-stable curve isomorphic to $E$
(the image of $E \times y$, where $y \in F(k)$ and 
$y \neq -y$), the action of $E$ on $X$ is faithful, and
hence $f$ is a closed immersion. We now show that $f$ 
\textit{is an isomorphism}; this implies the claim for 
dimension reasons.

First, note that every morphism $\bP^1 \to X$ is constant,
since the Albanese morphism $\alpha : X \to E/\langle z_0 \rangle$
has its fibers at all closed points isomorphic to the elliptic
curve $F$, and $E/\langle z_0 \rangle$ is an elliptic curve.
In particular, $X$ is not uniruled. By the above proposition,
it follows that $G$ is an abelian variety. 

Combining this with Proposition \ref{prop:act}, we see that
$C_G(x)$ is finite for any $x \in X(k)$. Thus, the $G$-orbits
of $k$-rational points of $X$ are abelian varieties, isogenous 
to $G$. Since $X$ is not an abelian variety (as its Albanese 
morphism is not an isomorphism), it follows that 
$\dim(G) \leq 1$. Thus, $G$ must be the image of $f$.
\end{example}

\subsection{Blanchard's lemma}
\label{subsec:bl}

Consider a group scheme $G$ acting on a scheme $X$,
and a morphism of schemes $f : X \to Y$. In general,
the $G$-action on $X$ does not descend to an action
on $Y$; for example, when $G = X$ is an elliptic curve
acting on itself by translations, and $f$ is the 
quotient by $\bZ/2$ acting via $x \mapsto \pm x$.
Yet we will obtain a descent result under the 
assumption that $f$ is proper and $f_*(\cO_X) = \cO_Y$; 
then $f$ is surjective and its fibers are connected 
by \cite[III.4.3.2, III.4.3.4]{EGA}. Such a descent result
was first proved by Blanchard in the setting of
holomorphic transformation groups, see \cite[I.1]{Bl}.

\begin{theorem}\label{thm:bla}
Let $G$ be a connected algebraic group, $X$ a $G$-scheme
of finite type, $Y$ a scheme of finite type and $f : X \to Y$ 
a proper morphism such that $f^\#: \cO_Y \to f_*(\cO_X)$ 
is an isomorphism. Then there exists a unique action of $G$ 
on $Y$ such that $f$ is equivariant.
\end{theorem}

\begin{proof}
Let $a : G \times X \to X$ denote the action. 
We show that there is a unique morphism 
$b: G \times Y \to Y$ such that the square
\[
\xymatrix{
G \times X \ar[r]^-{a} \ar[d]_{\id \times f} & X \ar[d]^{f} \\ 
G \times Y \ar[r]^-{b} & Y  \\}
\]
commutes. By \cite[II.8.11]{EGA}, it suffices to check that
the morphism 
\[ \id \times f : G \times X \longrightarrow G \times Y \]
is proper, the map 
\[ (\id \times f)^\# : \cO_{G \times Y} \longrightarrow 
(\id \times f)_*(\cO_{G \times X}) \] 
is an isomorphism, and the composition
\[ f \circ a : G \times X \longrightarrow Y \]
is constant on the fibers of $\id \times f$.

Since $f$ is proper, $\id \times f$ is proper as well. 
Consider the square
\[
\xymatrix{
G \times X \ar[r]^-{p_X} \ar[d]_{\id \times f} & X \ar[d]^{f} \\ 
G \times Y \ar[r]^-{p_Y} & Y,  \\}
\]
where $p_X$, $p_Y$ denote the projections. As this square
is cartesian and both horizontal arrows are flat, the
natural map 
\[ p_Y^* f_* (\cO_X) \longrightarrow
(\id \times f)_* p_X^*(\cO_X), \]
is an isomorphism, and hence so is the natural map
$\cO_{G \times Y} \to (\id \times f)_*(\cO_{G \times X})$.
We now check that $f \circ a$ is constant
on the fibers of $\id \times f$. Let $K$ be a field
extension of $k$ and let $g \in G(K)$, $y \in Y(K)$. 
Then the fiber of $\id \times f$ at $(g,y)$ is just
$g \times f^{-1}(y)$. So it suffices to show that the
morphism
\[ h : G_K \times f^{-1}(y) \longrightarrow Y_K, \quad
(g,x) \longmapsto f(g \cdot x) \]
is constant on $g \times f^{-1}(y)$. But this follows
e.g. from the rigidity lemma \ref{lem:rig} applied
to the irreducible components of $f^{-1}(y)$,
since $h$ is constant on $e \times f^{-1}(y)$, and
$f^{-1}(y)$ is connected.   

It remains to show that $a$ is an action of the
group scheme $G$. Note that $e$ acts on $X$ via 
the identity; moreover, the composite morphism of sheaves
\[ \cO_Y \stackrel{b^\#}{\longrightarrow} b_*(\cO_{G \times Y}) 
\stackrel{(e \times \id)^\#}{\longrightarrow} 
b_*( \cO_{e \times Y})  \cong \cO_Y \]
is the identity, since so is the analogous morphism 
\[ \cO_X \stackrel{a^\#}{\longrightarrow} a_*(\cO_{G \times X}) 
\stackrel{(e \times \id)^\#}{\longrightarrow} 
a_*( \cO_{e \times X})  \cong \cO_X \]
and $f_*(\cO_X) = \cO_Y$. Likewise, the square
\[
\xymatrix{
G \times G \times Y \ar[r]^-{\id \times b} \ar[d]_{m \times \id} 
& G \times Y \ar[d]^{b} \\
G \times Y \ar[r]^-{b} & Y \\}
\]
commutes on closed points, and the corresponding square 
of morphisms of sheaves commutes as well, since the analogous
square with $Y$ replaced by $X$ commutes.
\end{proof}

\begin{corollary}\label{cor:dir}
Let $f : X \to Y$ be a morphism of proper schemes such that
$f^\# : \cO_Y \to f_*(\cO_X)$ is an isomorphism. 

\begin{enumerate}

\item 
$f$ induces a homomorphism 
\[ f_* : \Aut^0_X \longrightarrow \Aut^0_Y. \]

\item If $f$ is birational, then $f_*$ is a closed immersion. 

\end{enumerate}

\end{corollary}

\begin{proof}
(1) This follows readily from the above theorem applied to the action of 
$\Aut^0_X$ on $X$.

(2) By Proposition \ref{prop:hom}, it suffices to
check that $\Ker(f_*)$ is trivial.
Let $S$ be a scheme and $u \in \Aut^0_X(S)$ such that
$f_*(u) = \id$. As $f$ is birational, there exists
a dense open subscheme $V \subseteq Y$ such that $f$ pulls
back to an isomorphism $f^{-}(V) \to V$. Then $u$ pulls back
to the identity on $f^{-1}(V) \times S$. Since the latter
is dense in $X \times S$, we obtain $u = \id$.
\end{proof}

The above corollary applies to a birational morphism 
$f: X \to Y$, where $X$ and $Y$ are proper varieties
and $Y$ is normal: then $f^\# : \cO_Y \to f_*(\cO_X)$ is
an isomorphism by Zariski's Main Theorem. 

Corollary \ref{cor:dir}(1) also applies to the two projections 
$p_1 : X \times Y \to X$ and $p_2: X \times Y \to Y$,
where $X$, $Y$ are proper varieties: then $\cO(X) = k = \cO(Y)$ 
and hence $p_{1*}(\cO_{X \times Y}) = \cO_X$,
$p_{2*}(\cO_{X \times Y}) = \cO_Y$ in view of Lemma \ref{lem:aff}. 
This implies readily the following:

\begin{corollary}\label{cor:prod}
Let $X$, $Y$ be proper varieties. Then the homomorphism
\[ p_{1*} \times p_{2*} : \Aut^0_{X \times Y} \longrightarrow
\Aut^0_X \times \Aut^0_Y \]
is an isomorphism with inverse the natural homomorphism
\[ \Aut^0_X \times \Aut^0_Y \longrightarrow \Aut^0_{X \times Y},
\quad (u,v) \longmapsto ((x,y) \mapsto (u(x), v(y))). \]
\end{corollary}

\subsection{Varieties with prescribed connected automorphism
group}
\label{subsec:vpcag}

\begin{theorem}\label{thm:auto}
Let $G$ be a smooth connected algebraic group of dimension $n$
over a perfect field $k$. Then there exists a normal projective 
variety $X$ such that $G \cong \Aut_X^0$ and 
$\dim(X) \leq 2 n + 1$.
\end{theorem}

The proof will occupy the rest of this subsection.
We will use the action of $G \times G$ on $G$ via left
and right multiplication; this identifies $G$ 
with the homogeneous space $(G \times G)/\diag(G)$, 
and $e$ to the base point of that homogeneous space.
By Theorem \ref{thm:equiv}, we may choose a 
projective compactification $Y$ of $G$ which is 
$G \times G$-equivariant, i.e., $Y$ is equipped with two
commuting $G$-actions, on the left and on the right.
We may further assume that $Y$ is normal in view of
Proposition \ref{prop:norm}.

Denote by $\Aut^G_Y$ the centralizer of $G$ in $\Aut_Y$
relative to the right $G$-action. Then $\Aut^G_Y$ is
a closed subgroup scheme of $\Aut_Y$ by Theorem 
\ref{thm:norcent}. Also, the left $G$-action on $Y$ yields
a homomorphism 
\[ \varphi : G \longrightarrow \Aut_Y^G. \]

\begin{lemma}\label{lem:autg}
The above map $\varphi$ is an isomorphism.
\end{lemma}

\begin{proof}
Since the left $G$-action on itself is faithful,
the kernel of $\varphi$ is trivial and hence $\varphi$
is a closed immersion. We show that $\varphi$ is 
surjective on $\bar{k}$-points: if 
$u \in \Aut_Y^G(\bar{k})$, then $u$ stabilizes the
open orbit of the right $G_{\bar{k}}$-action, i.e., 
$G_{\bar{k}}$. As $u$ commutes with that action, it follows 
that the pull-back of $u$ to $G_{\bar{k}} \subseteq Y_{\bar{k}}$ 
is the left multiplication by some $g \in G(\bar{k})$. 
Since $G_{\bar{k}}$ is dense in $Y_{\bar{k}}$, 
we conclude that $u = \varphi(g)$.

As $G$ is smooth, it suffices to check that 
$\Lie(\varphi)$ is an isomorphism to complete the proof. 
We have $\Lie(\Aut^G_Y) = H^0(Y,T_Y)^G = \Der^G(\cO_Y)$,
where $\Der(\cO_Y)$ denotes the Lie algebra of derivations
of $\cO_Y$, and $\Der^G(\cO_Y)$ the Lie subalgebra of 
invariants under the right $G$-action. Moreover, 
the pull-back $j : \Der(\cO_Y) \to \Der(\cO_G)$ 
is injective by the density of $G$ in $Y$. 
Also, recall that
$\Der^G(\cO_G) \cong \Lie(G)$ and this identifies the
composition $j \circ \Lie(\varphi)$ with the identity 
of $\Lie(G)$. This yields the desired statement. 
\end{proof}

For any closed subscheme $F \subseteq G$, we denote
by $\Aut_Y^F$ the centralizer of $F$ in $\Aut_Y$,
where $F$ is identified with 
$e \times F \subseteq e \times G$. Then again,
$\Aut_Y^F$ is a closed subgroup scheme of $\Aut_Y$; 
we denote its neutral component by $\Aut_Y^{F,0}$.

\begin{lemma}\label{lem:autf}
With the above notation, there exists a finite \'etale 
subscheme $F \subset G$ such that 
$\Aut_Y^{F,0} = \Aut_Y^{G,0}$.
\end{lemma}
 
\begin{proof}
Since $G$ is smooth, $G(\bar{k})$ is dense in $G_{\bar{k}}$ 
and hence the above lemma yields an isomorphism 
$G_{\bar{k}} \cong \Aut^{G(\bar{k})}(Y_{\bar{k}})$ with
an obvious notation. Consider the family of
(closed) subgroup schemes 
$\Aut^{F,0}_{Y_{\bar{k}}} \subseteq \Aut^0_{Y_{\bar{k}}}$,
where $F$ runs over the finite subsets 
$F \subseteq G(\bar{k})$, stable by the Galois group
$\Gamma$. Since $\Aut^0_{Y_{\bar{k}}}$ is
of finite type, we may choose $F$ so that
$\Aut^{F,0}_{Y_{\bar{k}}}$ is minimal in this family.
Let $\Omega$ be a $\Gamma$-orbit in $G(\bar{k})$, then
$\Aut^{F \cup \Omega,0}_{Y_{\bar{k}}} \subseteq \Aut^{F,0}_{Y_{\bar{k}}}$
and hence equality holds. Thus,
$\Aut^{F,0}_{Y_{\bar{k}}} = \Aut^{G(\bar{k}),0}_{Y_{\bar{k}}} 
= G_{\bar{k}}$.
This yields the assertion by using the bijective 
correspondence between finite \'etale subschemes
of $G$ and finite $\Gamma$-stable subsets of $G(\bar{k})$. 
\end{proof}

Next, consider the diagonal homomorphism
$\Aut^0_Y \to \Aut^0_Y \times \Aut^0_Y$.
By Corollary \ref{cor:prod}, we may identify the right-hand
side with $\Aut^0_{Y \times Y}$; this yields a closed
immersion of algebraic groups
\[ \Delta : \Aut^0_Y \longrightarrow \Aut^0_{Y \times Y}. \]
Also, for any finite \'etale subscheme $F \subseteq G$, 
the morphism
\[ \Gamma: F \times Y \longrightarrow Y \times Y, \quad
(g,x) \longmapsto (x, g \cdot x)\]
is finite, and hence we may view its image $Z$ as
a closed reduced subscheme of $Y \times Y$.  
Note that $Z$ is stable by the $G$-action via 
$\Delta \circ \varphi$; also, $Z_{\bar{k}}$ is 
the union of the graphs of the automorphisms 
$g \in F(\bar{k})$ of $Y$.

\begin{lemma}\label{lem:graph}
Keep the above notation.

\begin{enumerate}

\item $\Delta$ identifies $\Aut^0_Y$ with 
$\Aut^0_{Y \times Y, \diag(Y)}$ (the neutral component of 
the stabilizer of the diagonal in $\Aut_{Y \times Y}$). 

\item If $F$ is a finite \'etale subscheme of $G$ 
that satisfies the assertion of Lemma \ref{lem:autf} 
and contains $e$, then $\Delta$ identifies $\varphi(G)$ 
with $\Aut^0_{Y \times Y, Z, \red}$ (the reduced neutral component 
of the stabilizer of $Z$ in $\Aut_{Y \times Y}$).

\end{enumerate}

\end{lemma}

\begin{proof}
(1) Just observe that for any scheme $S$ and for any
$u,v \in \Aut_X(S)$, the automorphism $u \times v$
of $(X \times S) \times_S (X \times S)$ stabilizes
the diagonal if and only if $u = v$.

(2) We may assume $k$ algebraically closed; then
we may view $F$ as a finite subset of $G(k)$. 
Also, by Proposition \ref{prop:norm}, the reduced subgroup
$\Aut^0_{Y \times Y, Z,\red}$ stabilizes the graph 
$\Gamma_f$ of any $f \in F$, since these graphs
form the irreducible components of $Z_{\bar{k}}$. 
Now observe that for any scheme $S$ and for any 
$f,g \in \Aut_X(S)$, the automorphism $g \times g$
of $(X \times S) \times_S (X \times S)$ stabilizes
$\Gamma_f$ if and only if $g$ commutes with $f$. 
Thus, $\Delta$ identifies $\varphi(G)$ with
$\Aut^0_{Y \times Y, Z, \red}$.
\end{proof}

We now choose a finite \'etale subscheme $F \subseteq G$
that satisfies the assumption of Lemma \ref{lem:autf} 
and contains $e$. Denote by 
\[ f: X \longrightarrow Y \times Y \]
the normalization of the blowing-up of $Y \times Y$ along $Z$.
Then the $G$-action on $Y \times Y$ via $\Delta \circ \varphi$
lifts to a unique action on $X$; in other words, we have
a homomorphism
\[ f^* : G \longrightarrow \Aut^0_X. \] 
As $f$ is birational, $f^*$ is a closed immersion
by the argument of Corollary \ref{cor:dir} (ii).
Also, $Y \times Y$ is normal, since $Y$ is normal
and $k$ is perfect. Thus, the above corollary yields  
a closed immersion
\[ f_* : \Aut^0_X \longrightarrow \Aut^0_{Y \times Y}. \]
Moreover, the composition $f_* \circ f^*$ is an isomorphism
of $G$ to its image $(\Delta \circ \varphi)(G)$, since 
$f$ is birational. We will identify $G$ with its image,
and hence $f_* \circ f^*$ with $\id$.

Consider first the case where $\car(k) = 0$. Then $\Aut^0_X$ 
is smooth and hence stabilizes the exceptional locus 
$E \subset X$ of the birational morphism $f$. As a consequence,
the action of $\Aut^0_X$ on $Y \times Y$ via $f_*$ stabilizes 
the image $f(E) \subset Y \times Y$. 
If in addition $n \geq 2$, then $f(E) = Z$ and hence
$f_*$ sends $\Aut^0_X$ to $\Aut^0_{Y \times Y,Z}$, 
i.e. to $G$ by Lemma \ref{lem:graph}. 
It follows that $f^*$ is an isomorphism with inverse $f_*$. 

If $n = 1$, then $E$ is empty. We now reduce to the case 
where $n = 2$ as follows: we choose a smooth, projective, 
geometrically integral curve $C$ of genus $\geq 1$, equipped with 
a $k$-rational point $c$. Then $\Aut^0_{C,c}$ is trivial, 
as seen in Examples \ref{ex:curve} and \ref{ex:ab}. Using 
Corollary \ref{cor:prod}, it follows that the natural map
\[ \Aut^0_{Y \times Y, Z} \longrightarrow
\Aut^0_{Y \times Y \times C, Z \times c} \]
is an isomorphism; by Lemma \ref{lem:graph}, this yields an
isomorphism 
\[ G \cong \Aut^0_{Y \times Y \times C, Z \times c}. \] 
Also, note that $Y \times Y \times C$ is a normal projective variety. 
We now consider the morphism $f' : X' \to Y \times Y \times C$
obtained as the normalization of the blowing-up along $Z \times c$. 
Since the latter has codimension $2$ in $Y \times Y \times C$, 
the exceptional locus $E' \subset X'$ satisfies 
$f'(E') = Z \times c$. So the above argument yields again that 
$f'^*$ is an isomorphism with inverse $f'_*$. This completes 
the proof of Theorem \ref{thm:auto} in characteristic $0$.

Next, assume that $\car(k) = p > 0$. We will use 
the following additional result:

\begin{lemma}\label{lem:lie}
Let $f : X \to Y \times Y$ be as above and assume that 
$n - 1$ is not a multiple of $p$ (in particular, $n \geq 2$). 
Then the differential
\[ 
\Lie(f^*) : \Lie(G) \longrightarrow \Lie(\Aut^0_X)
\]
is an isomorphism.
\end{lemma}

\begin{proof}
As $\Lie(f_*) \circ \Lie(f^*) = \id$, it suffices
to show that the image of $\Lie(f_*)$ is contained
in $\Lie(G)$. For this, we may assume that $k = \bar{k}$.
We may thus view $F$ as a finite subset of $G(k)$
containing $e$.

We will use the action of $\Lie(\Aut^0_X) = \Der(\cO_X)$
on the ``jacobian ideal'' of $f$, defined as follows.
Recall that the sheaf of differentials,  
$\Omega^1_X = \cI_{\diag(X)}/\cI^2_{\diag(X)}$ is equipped
with a linearization for $\Aut_X$ (see \cite[I.6]{SGA3} 
for background on linearized sheaves). Likewise, 
$\Omega^1_{Y \times Y}$ is equipped with a linearization
for $\Aut_{Y \times Y}$ and hence for $\Aut^0_X$ acting
via $f_*$. Moreover, the canonical map 
\[ f^*(\Omega^1_{Y \times Y}) \longrightarrow \Omega^1_X \]
is a morphism of $\Aut^0_X$-linearized sheaves, since
it arises from the canonical map 
$f^{-1}(\cI_{\diag(Y \times Y)}) \to \cI_{\diag(X)}$.
Denoting by $\Omega^{2n}_{Y \times Y}$ the $2n$th exterior power
of $\Omega^1_{Y \times Y}$ and defining $\Omega^{2n}_X$ likewise,
this yields a morphism of $\Aut^0_X$-linearized sheaves
\[ f^*(\Omega^{2n}_{Y \times Y}) \longrightarrow \Omega^{2n}_X \]
and in turn, a morphism of $\Aut^0_X$-linearized sheaves
\[ Hom(\Omega^{2n}_X, f^*(\Omega^{2n}_{Y \times Y}))
\longrightarrow End( \Omega^{2n}_X). \]
The image $\cJ_f$ of this morphism is an $\Aut_X^0$-linearized
subsheaf of the sheaf of algebras $End(\Omega^{2n}_X)$. 
In particular, for any open subvariety $U$ of $X$, 
the Lie algebra $\Der(\cO_X)$ acts on the algebra 
$\End(\Omega^{2n}_U)$ by derivations that stabilize 
$\Gamma(U,\cJ_f)$. 

Denote by $Y_\reg \subseteq Y$ the regular (or smooth) locus
and consider the open subvariety 
$V \subseteq Y_\reg \times Y_\reg$ consisting of those points
that lie in at most one graph $\Gamma_g$, where $g \in F$.
Then $Z \cap V$ is a disjoint union of smooth varieties of
dimension $n$, and is dense in $Z$. Let $U := f^{-1}(V)$; 
then the pull-back \[ f_U : U \longrightarrow V \]
is the blowing-up along $Z \cap V$, and hence $U$ is smooth.
Thus, the sheaf $\Omega^{2n}_U$ is invertible and 
$\cJ_{f_U}$ is just a sheaf of ideals of $\cO_U$. 
A classical computation in local coordinates shows that
$\cJ_{f_U} = \cO_U(-(n-1)E)$, where $E$ denotes the exceptional
divisor of $f_U$. Hence we obtain an injective map
\[ \Der(\cO_X) \to \Der(\cO_U,\cJ_{f_U}) = 
\Der(\cO_U,\cO_U(-(n-1)E)) \]
with an obvious notation. Since $n-1$ is not a multiple of
$p$, we have 
\[ \Der(\cO_U,\cO_U(-(n-1)E)) = \Der(\cO_U,\cO_U(-E)). \]
(Indeed, if $D \in \Der(\cO_U,\cO_U(-(n-1)E)))$ and $z$ is
a local generator of $\cO_U(E)$ at $x \in X$, then
$D(z^{n-1}) = (n-1) z^{n-2} D(z) \in z^{n-1} \cO_{X,x}$
and hence $D(z) \in z \cO_{X,x}$). Also, the natural map
\[ \Der(\cO_U) \longrightarrow \Der(f_{U*}(\cO_U))
= \Der(\cO_V) \]   
is injective and sends $\Der(\cO_U,\cO_U(-E))$ to
$\Der(\cO_V,f_{U*}(\cO_U(-E)))$. Moreover,
$f_{U*}(\cO_U(-E))$ is the ideal sheaf of $Z \cap V$,
and hence is stable by $\Der(\cO_X)$ acting via the
composition
\[ \Der(\cO_X) \longrightarrow \Der(f_*(\cO_X)) = 
\Der(\cO_{Y \times Y}) \longrightarrow \Der(\cO_V). \]
It follows that the image of $\Lie(f_*)$ stabilizes
the ideal sheaf of the closure of $Z \cap V$ in 
$Y \times Y$, ie., of $Z$. By arguing as in the 
proof of Lemma \ref{lem:graph} (2), we conclude that 
$\Lie(f_*)$ sends $\Der(\cO_X)$ to $\Lie(G)$. 
\end{proof} 

As $f^*$ is a closed immersion and $G$ is smooth, the above 
lemma completes the proof of Theorem \ref{thm:auto} 
when $p$ does not divide $n-1$. Next, when $p$ divides $n-1$, 
we replace $Y \times Y$ (resp. $Z$) with $Y \times Y \times C$
(resp. $Z \times c$) for a pointed curve $(C,c)$ as above. 
This replaces the codimension $n$ of $Z$ in $Y \times Y$ with 
$n + 1$, and hence we obtain a normal projective variety $X'$  
of dimension $2 n + 1$ such that $G \cong \Aut^0_{X'}$.

\begin{remarks}\label{rem:auto}
(i) For any smooth connected \textit{linear} algebraic group
$G$ of dimension $n$, there exists a normal projective 
\textit{unirational} variety of dimension at most $2n + 2$ such
that $G \cong \Aut^0_X$. Indeed, the variety $G$ is unirational
(see \cite[XIV.6.10]{SGA3}) and hence so is $Y \times Y$ with the
notation of the above proof. Also, in that proof, the pointed curve 
$(C,c)$ may be replaced with a pair $(S,C)$, where $S$ is a smooth
rational projective surface such that $\Aut^0_S$ is trivial, and
$C \subset S$ is a smooth, geometrically irreducible curve. Such a
pair is obtained by taking for $S$ the blowing-up of $\bP^2$ at 
$4$ points in general position, and for $C$ an exceptional curve.

(ii) If $\car(k) = 0$ then every connected algebraic group $G$
is the connected automorphism group of some \textit{smooth} projective
variety, as follows from the existence of an equivariant resolution 
of singularities (see \cite[Prop.~3.9.1, Thm.~3.36]{Ko2}).

(iii) Still assuming $\car(k) = 0$, consider a finite-dimensional
Lie algebra $\fg$. Then $\fg$ is algebraic if and only if 
$\fg \cong \Der(\cO_X)$ for some proper scheme $X$, as follows 
by combining Corollary \ref{cor:lie}, Theorem \ref{thm:aut}
and Theorem \ref{thm:auto}. Moreover, $X$ may be chosen smooth, 
projective and unirational by the above remarks.
\end{remarks}

\medskip

\noindent
\textit{Notes and references}.

Most results of \S \ref{subsec:bra} are taken from \cite{MO}. 
Those of \S \ref{subsec:bl} are algebraic analogues of classical 
results about holomorphic transformation groups, see \cite[2.4]{Ak}.

In the setting of complex analytic varieties, the problem of 
realizing a given Lie group as an automorphism group has been
extensively studied. It is known that every finite group is 
the automorphism group of a smooth projective complex curve
(see \cite{Gr}); moreover, every compact connected real Lie group
is the automorphism group of a bounded domain (satisfying 
additional conditions), see \cite{BD,SZ}. Also, every connected
real Lie group of dimension $n$ is the automorphism group of
a Stein complete hyperbolic manifold of dimension $2n$ 
(see \cite{Win,Ka}). Theorem \ref{thm:auto}, obtained
in \cite[Thm.~1]{Br2}, may be viewed as an algebraic analogue 
of the latter result. The proof presented here is a streamlined 
version of that in \cite{Br2}. 

There are still many open questions about automorphism group
schemes. For instance, can one realize any algebraic group
over an arbitrary field as the full automorphism group scheme 
of a proper scheme? Also, very little is known on 
the group of connected components $\pi_0(\Aut_X)$, where 
$X$ is a proper scheme, or on the analogously defined 
group $\pi_0(\Aut(M))$, where $M$ is a compact complex
manifold (then $\Aut(M)$ is a complex Lie group, possibly
with infinitely many components). As mentioned in \cite{CZ}, 
it is not known whether there exists a compact complex 
manifold $M$ for which $\pi_0(\Aut(M))$ is not finitely 
generated.

\medskip

\noindent
Note added in proof: such an example (with $M$ projective algebraic) 
has been constructed by John Lesieutre, see arXiv:1609.06391.

\medskip

\noindent
{\bf Acknowledgements}.
These are extended notes of a series of talks given at Tulane 
University for the 2015 Clifford Lectures. I am grateful to 
the organizer, Mahir Can, for his kind invitation, and to 
the other speakers and participants for their interest 
and stimulating discussions.

These notes are also based on a course given at Institut Camille
Jordan, Lyon, during the 2014 special period on algebraic groups
and representation theory. I also thank this institution for its
hospitality, and the participants of the special period who made 
it so successful.

Last but not least, I warmly thank Rapha\"el Achet, Mahir Can, Bruno Laurent,
Preena Samuel and an anonymous referee for their careful reading 
of preliminary versions of this text and their very helpful comments.

\bibliographystyle{amsalpha}

\begin{thebibliography}{A}


\bibitem[EGA]{EGA}
A. Grothendieck,
\textit{\'El\'ements de g\'eom\'etrie alg\'ebrique
(r\'edig\'es avec la collaboration de J. Dieudonn\'e) : 
IV. \'Etude locale des sch\'emas et des morphismes de sch\'emas}, 
Publ. Math. IH\'ES \textbf{20} (1964), 5--259, 
\textbf{24} (1965), 5--231, \textbf{28} (1966), 5--255, 
\textbf{32} (1967), 5--361.


\bibitem[SGA1]{SGA1}
A. Grothendieck, M. Raynaud, 
\textit{Rev\^etements \'etales et groupe fondamental (SGA1)},
Soc. Math. de France, Paris, 2003.


\bibitem[SGA3]{SGA3}
M. Demazure, A. Grothendieck, 
\textit{Sch\'emas en groupes I, II, III (SGA 3)}, 
Springer Lecture Notes in Math. \textbf{151}, \textbf{152}, 
\textbf{153} (1970); revised version edited by P. Gille and P. Polo, 
vols. \textbf{I} and \textbf{III}, Soc. Math. de France, Paris, 2011. 


\bibitem[SGA6]{SGA6}
P. Berthelot, A. Grothendieck, L. Illusie, 
\textit{Th\'eorie des intersections et th\'eor\`eme de Riemann-Roch 
(SGA 6)}, Springer Lecture Notes in Math. \textbf{225}, 
Springer-Verlag,‎ Berlin-New York, 1971.


\bibitem{Ak}
D. Akhiezer,
\textit{Lie group actions in complex analysis},
Aspects Math. \textbf{E27},
Friedr. Vieweg \& Sohn, Braunschweig, 1995.


\bibitem{Al}
V. Alexeev,
\textit{Complete moduli in the presence of semiabelian
group action}, Ann. Math. \textbf{155} (2002), 611--708.


\bibitem{Ar}
S. Arima, 
\textit{Commutative group varieties},
J. Math. Soc. Japan \textbf{12} (1960), 227--237. 


\bibitem{Ba}
I. Barsotti, 
\textit{Structure theorems for group-varieties},
Ann. Mat. Pura Appl. (4) \textbf{38} (1955), 77--119.


\bibitem{BD}
E. Bedford, J. Dadok,
\textit{Bounded domains with prescribed group of automorphisms}, 
Comment. Math. Helv. \textbf{62} (1987), 561--572.


\bibitem{Bl}
A. Blanchard,
\textit{Sur les vari\'et\'es analytiques complexes},
Ann. Sci. \'Ec. Norm. Sup. \textbf{73} (1956), no. 3, 157--202.

 
\bibitem{Bo}
A. Borel, 
\textit{Linear algebraic groups, second ed.}, 
Graduate Texts in Math. \textbf{126}, Springer, 1991.


\bibitem{BS}
A. Borel, J.-P. Serre,
\textit{Th\'eor\`emes de finitude en cohomologie galoisienne},
Comment. Math. Helv. \textbf{39} (1964) 111--164.

 
\bibitem{BLR}
S. Bosch, W. L\"utkebohmert, M. Raynaud, 
\textit{N\'eron models}, 
Ergebnisse Math. Grenzg. (3) \textbf{21}, Springer, 1990. 


\bibitem{Br1}
M. Brion, 
\textit{Anti-affine algebraic groups}, 
J. Algebra \textbf{321} (2009), 934--952.


\bibitem{Br2}
M. Brion, 
\textit{On automorphisms and endomorphisms of projective varieties},
in: Automorphisms in birational and affine geometry, 59--81, 
Springer Proc. Math. Stat. \textbf{79}, Springer, Cham, 2014.


\bibitem{Br3}
M. Brion, 
\textit{Which algebraic groups are Picard varieties?},
Sci. China Math. \textbf{58} (2014), 461--478. 


\bibitem{Br4}
M. Brion, 
\textit{On extensions of algebraic groups with finite quotient},
Pacific J. Math. \textbf{279} (2015), 153--170.


\bibitem{BSU}
M. Brion, P. Samuel, V. Uma, 
\textit{Lectures on the structure of algebraic groups 
and geometric applications},
CMI Lecture Series in Mathematics \textbf{1}, 
Hindustan Book Agency, New Delhi; 
Chennai Mathematical Institute (CMI), Chennai, 2013.


\bibitem{CZ}
S. Cantat, A. Zeghib, 
\textit{Holomorphic actions, Kummer examples, and Zimmer program}, 
Ann. Sci. \'Ec. Norm. Sup. (4) \textbf{45} (2012), no. 3, 447--489. 


\bibitem{Ch1}
C. Chevalley,
\textit{Th\'eorie des groupes de Lie. Tome III. 
Th\'eor\`emes g\'en\'eraux sur les alg\`ebres de Lie}, 
Actualit\'es Sci. Ind. \textbf{1226}, Hermann, Paris, 1955.


\bibitem{Ch2}
C. Chevalley, 
\textit{Une d\'emonstration d'un th\'eor\`eme 
sur les groupes alg\'ebriques},
J. Math. Pures Appl. (9) \textbf{39} (1960), 307--317.


\bibitem{Co1}
B. Conrad, 
\textit{A modern proof of Chevalley's theorem on algebraic groups}, 
J. Ramanujan Math. Soc. \textbf{17} (2002), no. 1, 1--18.


\bibitem{Co2}
B. Conrad, 
\textit{Chow's $K/k$-image and $K/k$-trace, and the Lang-N\'eron theorem}.
Enseign. Math. (2) \textbf{52} (2006), no. 1-2, 37--108.


\bibitem{CGP}
B. Conrad, O. Gabber, G. Prasad,
\textit{Pseudo-reductive groups. Second edition},
New Math. Monogr. \textbf{26}, Cambridge Univ. Press, Cambridge, 2015.


\bibitem{De}
O.~Debarre,
\textit{Higher-dimensional algebraic geometry},
Universitext, Springer-Verlag, New York, 2001.


\bibitem{DG}
M. Demazure, P. Gabriel, 
\textit{Groupes alg\'ebriques}, Masson, Paris, 1970. 


\bibitem{FI}
R. Fossum, B. Iversen,
\textit{On Picard groups of algebraic fibre spaces},
J. Pure Applied Algebra \textbf{3} (1973), 269--280.


\bibitem{Ga}
O. Gabber,
\textit{On pseudo-reductive groups and compactification theorems},
Oberwolfach Reports \textbf{9} (2012), 2371--2374.


\bibitem{GGM}
O. Gabber, P. Gille, L. Moret-Bailly,
\textit{Fibr\'es principaux sur les corps valu\'es hens\'eliens},
Algebr. Geom. \textbf{5} (2014), 573--612.


\bibitem{Ge}
T. Geisser, 
\textit{The affine part of the Picard scheme},
Compos. Math. \textbf{145} (2009), no. 2, 415--422.


\bibitem{GT}
S. Greco, C. Traverso, 
\textit{On seminormal schemes},
Compos. Math. \textbf{40} (1980), no. 3, 325--365.


\bibitem{Gr}
L. Greenberg, 
\textit{Maximal groups and signatures}, 
in: Discontinuous Groups and Riemann Surfaces, 207--226, 
Ann. of Math. Studies \textbf{79}, 
Princeton Univ. Press, Princeton, NJ, 1974.


\bibitem{Ig}
J. Igusa,
\textit{On some problems in abstract algebraic geometry,}
Proc. Nat. Acad. Sci. U.S.A. \textbf{41} (1955), 964--967. 


\bibitem{Ka}
S.-J. Kan,
\textit{Complete hyperbolic Stein manifolds with 
prescribed automorphism groups}, 
Comment. Math. Helv. \textbf{82} (2007), no. 2, 371--383.


\bibitem{Kl}
S. L. Kleiman, 
\textit{The Picard scheme}, 
in: Fundamental algebraic geometry, 235--321, 
Math. Surveys Monogr. \textbf{123}, Amer. Math. Soc., 2005.


\bibitem{KMT}
T. Kambayashi, M. Miyanishi, M. Takeuchi,
\textit{Unipotent algebraic groups},
Lecture Notes in Math. \textbf{414}, 
Springer-Verlag, Berlin-New York, 1974.


\bibitem{Ko1}
J. Koll\'ar,
\textit{Rational curves on algebraic varieties},
Ergeb. Math. Grenzgeb. (3) \textbf{32},
Springer-Verlag, Berlin, 1996.


\bibitem{Ko2}
J. Koll\'ar,
\textit{Lectures on resolution of singularities},
Ann. of Math. Stud. \textbf{166}, Princeton Univ. Press,
Princeton, 2007.


\bibitem{Li}
Q. Liu,
\textit{Algebraic geometry and arithmetic schemes},
Oxf. Grad. Texts Math. \textbf{6}, 
Oxford Univ. Press, Oxford, 2002.



\bibitem{Lu}
W. L\"utkebohmert, 
\textit{On compactification of schemes},
Manuscripta Math. \textbf{80} (1993), no. 1, 95--111. 


\bibitem{MO}
H. Matsumura, F. Oort,
\textit{Representability of group functors, and automorphisms 
of algebraic schemes}, 
Invent. Math. \textbf{4} (1967), 1--25.


\bibitem{MM}
B. Mazur, W. Messing, 
\textit{Universal extensions and one dimensional crystalline cohomology},
Lecture Notes in Math. \textbf{370}, Springer-Verlag, Berlin-New York, 1974.


\bibitem{MZ}
S. Meng, D.-Q. Zhang,
\textit{Jordan property for non-linear algebraic groups 
and projective varieties}, arXiv:1507.02230.
 

\bibitem{Mi}
J. S. Milne,
\textit{A proof of the Barsotti-Chevalley theorem on algebraic groups},
preprint, arXiv:1311.6060v2.


\bibitem{Mum}
D. Mumford,
\textit{Abelian varieties. 
With appendices by C. P. Ramanujam and Yuri Manin}, 
Corrected reprint of the second (1974) edition,
Tata Inst. Fundam. Res. Stud. Math. \textbf{5},
Hindustan Book Agency, New Delhi, 2008. 


\bibitem{MFK}
D. Mumford, J. Fogarty, F. Kirwan,
\textit{Geometric invariant theory},
Ergeb. Math. Grenzgeb. (2) \textbf{34}, 
Springer-Verlag, Berlin, 1994. 


\bibitem{Mur}
J. P. Murre,
\textit{On contravariant functors from the category of pre-schemes 
over a field into the category of abelian groups 
(with an application to the Picard functor)},
Publ. Math. IH\'ES \textbf{23} (1964), 5--43.
 

\bibitem{Na1}
M. Nagata,
\textit{Imbedding of an abstract variety in a complete variety},
J. Math. Kyoto Univ. \textbf{2} (1962), 1--10.


\bibitem{Na2}
M. Nagata,
\textit{A generalization of the imbedding problem of an abstract 
variety in a complete variety},
J. Math. Kyoto Univ. \textbf{3} (1963), 89--102.


\bibitem{Oo}
F. Oort, 
\textit{Commutative group schemes}, 
Lecture Notes in Math. \textbf{15}, 
Springer-Verlag, Berlin-New York, 1966. 


\bibitem{Ra}
M. Raynaud, 
\textit{Faisceaux amples sur les sch\'emas en groupes 
et les espaces homog\`enes}, 
Lecture Notes in Math. \textbf{119}, Springer-Verlag, 1970.


\bibitem{Ro1}
M. Rosenlicht, 
\textit{Some basic theorems on algebraic groups},
Amer. J. Math. \textbf{78} (1956), 401--443. 


\bibitem{Ro2}
M. Rosenlicht,
\textit{Extensions of vector groups by abelian varieties},
Amer. J. Math. \textbf{80} (1958), 685--714.


\bibitem{Ro3}
M. Rosenlicht, 
\textit{Toroidal algebraic groups},
Proc. Amer. Math. Soc. \textbf{12} (1961), 984--988. 


\bibitem{SZ}
R. Saerens, W. Zame, 
\textit{The isometry groups of manifolds and the automorphism 
groups of domains},
Trans. Amer. Math. Soc. \textbf{301} (1987), no. 1, 413--429. 


\bibitem{SS}
C. Sancho de Salas, S. Sancho de Salas, 
\textit{Principal bundles, quasi-abelian varieties 
and structure of algebraic groups}, 
J. Algebra \textbf{322} (2009), no. 8, 2751--2772.


\bibitem{Se1}
J.-P. Serre,
\textit{Groupes alg\'ebriques et corps de classes.
Deuxi\`eme \'edition}, Hermann, Paris, 1975.


\bibitem{Se2}
J.-P. Serre,
\textit{Morphismes universels et vari\'et\'e d'Albanese},
in: Expos\'es de s\'eminaires (1950-1999), 141--160,
Soc. Math. France (2001). 


\bibitem{ST}
I. R. Shafarevich, J. Tate,
\textit{The rank of elliptic curves},
Dokl. Akad. Nauk SSSR \textbf{175} (1967), 770--773. 


\bibitem{Sp}
T. A. Springer,
\textit{Linear algebraic groups, second edition},
Prog. Math. \textbf{9}, Birkh\"auser, Boston, 1998.


\bibitem{To}
B. Totaro, 
\emph{Pseudo-abelian varieties},
Ann. Sci. \'Ec. Norm. Sup. (4) \textbf{46} (2013), no. 5, 693--721.


\bibitem{Tr}
C. Traverso, 
\textit{Seminormality and Picard group}, 
Ann. Scuola Norm. Sup. Pisa (3) \textbf{24} (1970), 585--595.


\bibitem{Ul}
D. Ulmer,
\textit{Elliptic curves with large rank over function fields},
Ann. of Math. (2) \textbf{155} (2002), no. 1, 295--315.


\bibitem{We}
C. A. Weibel, 
\textit{Pic is a contracted functor},
Invent. Math. \textbf{103} (1991), no. 2, 351--377. 


\bibitem{Win}
J. Winkelmann,
\textit{Realizing connected Lie groups as automorphism groups
of complex manifolds},
Comment. Math. Helv. \textbf{79} (2004), no. 2, 285--299.


\bibitem{Wit}
O. Wittenberg, 
\textit{On Albanese torsors and the elementary obstruction},
Math. Ann. \textbf{340} (2008), no. 4, 805--838. 


\end{thebibliography}

\end{document}